\documentclass{compositio}

\usepackage{amsmath,amsfonts,yhmath,hyperref}
\usepackage{amsthm}
\usepackage{amssymb}
\usepackage[latin9]{inputenc}
\usepackage[nottoc,notlof,notlot]{tocbibind}
\usepackage{mathtools}
\usepackage{enumitem}
\usepackage{tikz-cd}
\usepackage{multirow}
\usepackage{mathrsfs}

\usepackage{xparse}

\usepackage{array}

\usepackage{pgf,tikz}

\setitemize{label=$\triangleright$}

\usepackage{dsfont}
\usetikzlibrary{arrows}
\usetikzlibrary[patterns]

\newcommand{\prim}{\mathrm{prim}}

\newcommand{\NN}{\mathbb{N}}
\newcommand{\ZZ}{\mathbb{Z}}
\newcommand{\RR}{\mathbb{R}}
\newcommand{\CC}{\mathbb{C}}

\newcommand{\QQ}{\mathbb{Q}}

\newcommand{\TT}{\mathbb{T}}

\newcommand{\Dfk}{\mathfrak{D}}

\newcommand{\pfk}{\mathfrak{p}}

\newcommand{\mfk}{\mathfrak{m}}

\newcommand{\CCC}{\mathscr{C}}

\renewcommand{\P}{\mathcal{P}}
\renewcommand{\L}{\mathcal{L}}
\newcommand{\N}{\mathcal{N}}
\newcommand{\M}{\mathcal{M}}

\newcommand{\R}{\mathcal{R}}

\newcommand{\dd}{ \mathrm{d} }
\newcommand{\refined}{\mathrm{ref}}

\newcommand{\val}{\mathrm{val}}

\newcommand{\im}{\mathfrak{Im}}

\newtheorem{theo}{Theorem}[section]
\newtheorem*{theom}{Theorem}
\newtheorem{prop}[theo]{Proposition}

\newtheorem{coro}[theo]{Corollary}
\newtheorem{lem}[theo]{Lemma}

\newtheorem{conj}{Conjecture}[section]

\theoremstyle{definition}
\newtheorem{defi}[theo]{Definition}

\theoremstyle{remark}
\newtheorem{remark}[theo]{Remark}
\newenvironment{rem}[1]{
    \begin{remark}#1}{
    \xqed{\blacklozenge}\end{remark}
}

\theoremstyle{remark}
\newtheorem{example}[theo]{Example}
\newenvironment{expl}[1]{
    \begin{example}#1}{
    \xqed{\lozenge}\end{example}
}

\newcommand{\xqed}[1]{
    \leavevmode\unskip\penalty9999 \hbox{}\nobreak\hfill
    \quad\hbox{\ensuremath{#1}}}

\def\pearl (#1) at (#2:#3) label=#4; {
    \node[draw,rectangle, minimum width=0.5cm, minimum height = 0.5 cm] (#1) at (#2:#3) {$#4$} ;
}

\def\flatpearl (#1) at (#2:#3) label=#4; {
    \node[draw,circle, minimum width=0.5cm, minimum height = 0.5 cm] (#1) at (#2:#3) {$#4$} ;
}

\def\dottedpearl (#1) at (#2:#3) label=#4; {
    \node[draw,dotted,circle, minimum width=0.5cm, minimum height = 0.5 cm] (#1) at (#2:#3) {$#4$} ;
}

 
\begin{document}
 
 
\title{Multiple cover formula for complex and refined invariants in abelian surfaces}
\author{Thomas Blomme}

\begin{abstract}
This paper deals with the computation of enumerative invariants of abelian surfaces through the tropical approach. We develop a pearl diagram algorithm similar to the floor diagram algorithm used in toric surfaces that concretely solves the tropical problem. 

These diagrams can be used to prove specific cases of Oberdieck's multiple cover formula that reduce the computation of invariants for non-primitive classes to the primitive case, getting rid of all diagram considerations and providing short explicit formulas. The latter can be used to prove the quasi-modularity of generating series of classical invariants, and the polynomiality of coefficients of fixed codegree in the refined invariants.
\end{abstract}

\maketitle

\tableofcontents

\section{Introduction}

\subsection{Preliminaries and problem setting}

\subsubsection{Abelian surfaces.} A complex torus is constructed as the quotient of the algebraic torus $N_{\CC^*}=N\otimes\CC^*$, where $N$ is a lattice, by a second lattice $\Lambda\subset N_{\CC^*}$. A complex torus is called an \textit{abelian surface} if it can be endowed with the choice of a positive line bundle. This condition on the complex torus is known as \textit{Riemann bilinear relation} and are explained in chapter 2.6 of \cite{griffiths2014principles}. Abelian surfaces form a family of complex manifolds that appear in many places throughout algebraic geometry.

Their tropical counterparts, called \textit{tropical abelian surfaces}, are real tori $N_\RR/\Lambda$, obtained by quotienting a real vector space $N_\RR$ containing a lattice $N$, by a second lattice $\Lambda\subset N_\RR$, such that the result can also be endowed with a positive line bundle. This condition can be expressed in term of the inclusion $\Lambda\subset N_\RR$: taking basis of $N$ so that $N_\RR=\RR^2$ and $S$ a matrix whose columns span $\Lambda\subset\RR^2$, the condition amounts to the existence of a matrix $C$, formally a map $\Lambda^*\to N$, such that the matrix product $CS^T$ is symmetric.

\subsubsection{Curves in abelian surfaces.} One way to understand a complex manifold is to study the curves that it contains. To do this, one can count the curves of fixed degree and genus that satisfy given conditions, for instance passing through the right number of points. In the case of the projective plane, one can count degree $d$ and genus $g$ curves passing through $3d-1+g$ points. In the complex case, the answer happens only to depend on the nature of the constraints and not on the geometric choice of the constraints. In some cases, these invariants coincide with Gromov-Witten invariants.

For abelian surfaces, we count curves of genus $g$ realizing the class $C$ passing through $g$ points, where $C$ is represented by a $2\times 2$ integer matrix that can be assumed to be diagonal with positive entries. The result does not depend on the choice of the points nor the abelian surface as long as it contains curves in the class $C$. These invariants are denoted $\N_{g,C}$. They are considered in \cite{bryan1999generating} and \cite{bryan2018curve}. Moreover, it is possible not to fix the homology class $C$ but rather the linear system to which the curves belong. This amounts to fixing a line bundle $\L$ and look for curves that are zero-loci of sections of $\L$. As the Picard group of an abelian surface is of dimension $2$ (it is in fact given by the dual complex torus), fixing the linear system is a codimension $2$ condition. Thus, we can count the number of genus $g$ curves in a fixed linear system passing through $g-2$ points in general position. We also get an invariant denoted this time by $\N^{FLS}_{g,C}$.

\subsubsection{The tropical approach.} As the number of curves of a fixed degree and genus does not depend on the choice of the constraints, it is possible to try to compute this number by choosing the constraints in a degenerated way, called the \textit{tropical limit}. This approach was first implemented in the case of toric surfaces by G. Mikhalkin \cite{mikhalkin2005enumerative}. Close to the tropical limit, complex solutions to the enumerative problem break into several pieces whose structure is encoded in a \textit{tropical curve}, which is sufficient to recover the complex curves. The main result of \cite{mikhalkin2005enumerative} is the \textit{correspondence theorem}: the count of complex curve solution to the enumerative problem is equal to the number of tropical solutions, provided the latter are counted with a suitable multiplicity. The correspondence theorem has since been generalized to different settings with various techniques \cite{shustin2002patchworking, nishinou2006toric, tyomkin2012tropical, mandel2020descendant}.

One of the latest generalizations of the correspondence theorem is presented in \cite{nishinou2020realization}. Using the log-geometry techniques from the correspondence theorem by T. Nishinou and B. Siebert \cite{nishinou2006toric}, T. Nishinou gives a correspondence theorem for complex curves of genus $g$ passing through $g$ points in general position inside an abelian surface. This extends the tropical techniques beyond the toric scope. His correspondence theorem yields a tropical way of computing the invariant $\N_{g,C}$. The multiplicity provided by \cite{nishinou2020realization} is computed in \cite{blomme2022abelian1}: if $h:\Gamma\to\TT A$ is a tropical curve passing through $g$ points, its multiplicity is equal to $\delta_\Gamma m_\Gamma$, where $\delta_\Gamma$ is the gcd of the weights of the edges of the tropical curve $\Gamma$, and $m_\Gamma$ is the usual product of vertex multiplicities, as presented in \cite{mikhalkin2005enumerative}.

The correspondence theorem from \cite{nishinou2020realization} is generalized in \cite{blomme2022abelian2} to work for the enumerative problem consisting in counting the genus $g$ curves in a fixed linear system that pass through $g-2$ points in general position, so that one gets the invariant $\N^{FLS}_{g,C}$. The multiplicity of a tropical curve $h:\Gamma\to\TT A$ is now of the form $\delta_\Gamma\Lambda_\Gamma^\Sigma m_\Gamma$, where $\delta_\Gamma$ and $m_\Gamma$ are as before, and $\Lambda_\Gamma^\Sigma$ is the index of the lattice spanned by the unmarked loops of $\Gamma$ inside $H_1(\TT A,\ZZ)$. See Section \ref{section tropical invariants} or \cite{blomme2022abelian2} for more details.

\subsubsection{Refined invariants.} Correspondence theorems in tropical geometry provide multiplicities to tropical curves so that their counts yield complex invariants. In our case, we have two multiplicities, one for each of the enumerative problems, so that the counts of solutions using each of them give the invariants $\N_{g,C}$ and $\N_{g,C}^{FLS}$. The multiplicities provided by \cite{mikhalkin2005enumerative} as well as the above multiplicities are products over the vertices of the tropical curves. In \cite{block2016refined}, F. Block and L. G\"ottsche proposed to replace the vertex multiplicities by their quantum analog:
$$\prod_V m_V \text{ is replaced by }\prod_V(q^{m_V/2}-q^{-m_V/2})\in\ZZ[q^{\pm 1/2}].$$
This yields a refined multiplicity which is Laurent polynomial in a new variable $q$. Surprisingly, the refined count is still invariant. This was proven by I. Itenberg and G. Mikhalkin \cite{itenberg2013block} in the toric setting. Several interpretations of these tropical invariants in the classical setting have been provided, but they remain quite mysterious. See \cite{bousseau2019tropical} and \cite{mikhalkin2017quantum}.

Existence of refined invariants in the abelian setting is proved in \cite{blomme2022abelian1,blomme2022abelian2}. Counting curves passing through $g$ points, we prove that the counts using multiplicities $m_\Gamma=\prod m_V$ and $\delta_\Gamma m_\Gamma$ are invariants. The corresponding invariant counts are denoted by $M_{g,C}$ and $N_{g,C}$. Each of these multiplicities is refined as follows:
$$\begin{array}{>{\displaystyle}l}
\prod_V m_V \text{ is replaced by }\prod_V(q^{m_V/2}-q^{-m_V/2})\in\ZZ[q^{\pm 1/2}],\\
\delta_\Gamma\prod_V m_V \text{ is replaced by }\sum_{k|\delta_\Gamma}k^{2g-2}\varphi(k)\prod_V(q^{m_V/2k}-q^{-m_V/2k})\in\ZZ[q^{\pm 1/2}],\\
\end{array}$$
where $\varphi$ is the Euler function. Notice that the second refinement is not the one used in \cite{blomme2022abelian1} and \cite{blomme2022abelian2}. Its invariance follows from the fact that we have invariance counting curves with a fixed $\delta_\Gamma$, so that the recipe for the multiplicity may depend on the value of $\delta_\Gamma$. The corresponding invariant counts are denoted by $BG_{g,C}$ and $R_{g,C}$ respectively. The $BG$ stands for Block-G\"ottsche, who introduced the refined multiplicity, while $R$ stands for ``refined". The justification of such a weird refinement comes from the formula satisfied in Theorem \ref{theorem multiple cover formulas point}. The counts for curves in a fixed linear system are also invariants and are denoted with an additional ``FLS".

\subsubsection{Explicit computations and floor diagrams.} The use of a correspondence theorem only reduces a complex problem to a tropical one. The latter still needs to be resolved in order to get concrete values, or prove some regularity statements. The idea leading us from complex geometry to tropical geometry can still be applied: by choosing the constraints in a degenerated way, the tropical curves break into several pieces whose structure can be encoded in an object called \textit{diagram}, and the number of solutions encoded in a diagram can be recovered with only the knowledge of the diagram.

This approach has been implemented for counting degree $d$ and genus $g$ curves passing through $3d-1+g$ points in $\CC P^2$ (as well as several other toric surfaces) by E. Brugall\'e and G. Mikhakin \cite{brugalle2007enumeration, brugalle2008floor}. It was used to prove several regularity statements, see for instance \cite{mikhalkin2010labeled}. The \textit{floor diagram algorithm} in toric surfaces relies on the fact that it is possible to ``project" the problem onto a line, reducing the enumeration of tropical curves to the enumeration of some specific tropical covers of the line. Tropical covers are maps between graphs that satisfy a balancing condition. They were studied in \cite{caporaso2014gonality} and \cite{cavalieri2010tropical}. Floor diagrams have since been generalized to other settings, see for instance \cite{block2016fock}.

In our problem, it is not possible to project the constraints onto a line, but it is possible to do so on a circle. We obtain \textit{circular diagrams}. Such diagrams were already considered in \cite{boehm2018tropical} and \cite{bohm2020counts} to count covers of elliptic curves and curves in the trivial line bundle over an elliptic curve. They also found a use to deal with Gromov-Witten invariants of bielliptic surfaces \cite{blomme2024bielliptic}. In this paper, we have similar diagrams whose enumeration with different multiplicities gives the desired invariants. However, their main interests consists in proving some \textit{multiple cover formulas}. After that, one does not need the diagrams to compute the values of the invariants.

Floor diagrams also carry a relation to operators on Fock spaces that has been studied by Y. Cooper and R. Pandharipande \cite{cooper2017fock}, as well as F. Block and L. G\"ottsche \cite{block2016fock, block2016refined} where the connection to refined invariants is also dealt with. Fock spaces can for instance be used to compute generating series of Hurwitz numbers and Gromov-Witten invariants of toric surfaces. This is done by expressing the Gromov-Witten invariants as matrix elements of operators on a Fock space. In the setting of circular diagrams such as in \cite{bohm2020counts}, it is possible to express the invariants as the trace of an operator on the Fock space, as done in \cite{cooper2017fock} for Gromov-Witten invariants of the trivial line bundle over an elliptic curve. We do not provide such an approach in our case although there might be one.

\subsection{Results}

\subsubsection{Pearl diagram algorithm.} The first contribution of the present paper is to provide an algorithm that solves each of the tropical enumerative problems in a suitably chosen situation: the abelian surface as well as the point constraints are chosen in a \textit{stretched} way. This choice reduces each tropical enumerative problem to an enumeration of finite decorated graphs with some suitable multiplicity.

More precisely, let $\TT A$ be a \textit{stretched} tropical abelian surface, \textit{i.e.} coming from a lattice spanned by $\left(\begin{smallmatrix} L & \varepsilon \\ \varepsilon & l \end{smallmatrix}\right)$ with $L\gg l\gg\varepsilon$, up to bounded coefficients depending on the chosen class of curves. In other terms, it has a long horizontal direction, and a small vertical direction. We prove that tropical curves solution to our problem in $\TT A$ have edges with horizontal slopes, and complement of these horizontal edges are subgraphs each homeomorphic to a circle going around the vertical direction of the torus. Contracting these components, which amounts to contract the short vertical direction of $\TT A$ yields the combinatorial objects called \textit{pearl diagrams}. Conversely, given a pearl diagram, we explain how to recover the tropical curves it encodes, which amounts to find the vertical coordinate of the horizontal edges, and the shape of the components in their complement. This technical part is achieved in Propositions \ref{proposition curves in diagram point case} and \ref{proposition curves in diagram linear system case}.

In the case of primitive classes of the form $\left(\begin{smallmatrix} 1 & 0 \\ 0 & n \end{smallmatrix}\right)$, the enumeration of pearl diagrams is completely straightforward and they allow one to recover the explicit expressions for the invariants that were already proved in \cite{bryan1999generating}.

The pearl diagrams have two drawbacks. First, their enumeration outside the primitive case may be considered difficult. Furthermore, the multiplicities that we need to use to count pearl diagrams are also quite complicated. This is due to the fact that it is not possible to get a grasp on the gcd of the tropical curves encoded in a diagram when performing the reconstruction from Proposition \ref{proposition curves in diagram point case}. Thus, it is at first not possible to get the count of tropical curves using the multiplicity $\delta_\Gamma m_\Gamma$ without performing additional manipulations, carried out in Section \ref{sec-diagram-mult} for the point case.

In total, there are a lot of diagrams to count, and those have obscure multiplicities. For these reasons, there is probably no way to briefly present the algorithm and we thus refer the interested reader to the rest of the paper. Definition of pearl diagrams can be found in Section \ref{section pearl diagrams definition}, and the multiplicities are provided by Theorem \ref{theorem diagram enumeration g points} for the point case, and Theorem \ref{theorem diagram enumeration linear system} for the FLS case. The good news is that, surprisingly, the two drawbacks previously mentioned cancel each other in some sense: it is possible to use the diagrams to prove an easy formula so that one can compute the invariants without having to know how to enumerate pearl diagrams.

\subsubsection{Multiple cover formulas.} The pearl diagrams have a \textit{bidegree} $(d_1,d_2)$, and their enumeration is made over the pearl diagrams of a fixed bidegree. The main interesting property is that pearl diagrams of bidegree $(d_1,d_2)$ can be used to enumerate curves in a class $\left(\begin{smallmatrix} d_1 & a \\ 0 & d_2 \\ \end{smallmatrix}\right)$ for any choice of $a$. Of course, this enumeration is made using different multiplicities. As the counts are known to depend only on the equivalence class of $\left(\begin{smallmatrix} d_1 & a \\ 0 & d_2 \\ \end{smallmatrix}\right)$, we extract identities between enumerative invariants from relations between the different diagram multiplicities. More precisely, taking $a=1$, we obtain the curve count for the primitive class, but expressed as a sum over the bidegree $(d_1,d_2)$ pearl diagrams with some suitable multiplicity. The miracle lies in the small convolution formula that relates the multiplicities used to recover the counts in the classes $\left(\begin{smallmatrix} d_1 & 1 \\ 0 & d_2 \\ \end{smallmatrix}\right)$ and $\left(\begin{smallmatrix} d_1 & 0 \\ 0 & d_2 \\ \end{smallmatrix}\right)$ respectively. We thus derive the following formulas that enable an effective computation of all the invariants.

\begin{theom}\ref{theorem multiple cover formulas point}
We have the following multiple cover formulas
$$\begin{array}{c>{\displaystyle}r>{\displaystyle}l}
(i) & M_{g,\left(\begin{smallmatrix}
d_1 & 0 \\
0 & d_2 \\
\end{smallmatrix}\right)} & =\sum_{k|\mathrm{gcd}(d_1,d_2)}k^{4g-4}N_{g,\left(\begin{smallmatrix}
1 & 0 \\
0 & d_1 d_2/k^2 \\
\end{smallmatrix}\right)} , \\
(ii) & N_{g,\left(\begin{smallmatrix}
d_1 & 0 \\
0 & d_2 \\
\end{smallmatrix}\right)} & =\sum_{k|\mathrm{gcd}(d_1,d_2)}k^{4g-3}N_{g,\left(\begin{smallmatrix}
1 & 0 \\
0 & d_1 d_2/k^2 \\
\end{smallmatrix}\right)} , \\
\end{array}$$
$$\begin{array}{c>{\displaystyle}r>{\displaystyle}l}
(iii) & BG_{g,\left(\begin{smallmatrix}
d_1 & 0 \\
0 & d_2 \\
\end{smallmatrix}\right)} & =\sum_{k|\mathrm{gcd}(d_1,d_2)}k^{2g-2}BG_{g,\left(\begin{smallmatrix}
1 & 0 \\
0 & d_1 d_2/k^2 \\
\end{smallmatrix}\right)}(q^k). \\
(iv) & R_{g,\left(\begin{smallmatrix}
d_1 & 0 \\
0 & d_2 \\
\end{smallmatrix}\right)} & =\sum_{k|\mathrm{gcd}(d_1,d_2)}k^{2g-1}R_{g,\left(\begin{smallmatrix}
1 & 0 \\
0 & d_1 d_2/k^2 \\
\end{smallmatrix}\right)}(q^k). \\
\end{array}$$
\end{theom}

Theorem \ref{theorem multiple cover formula linear case} gives analog results for invariants in a fixed linear system. The above multiple cover formulas (ii) is a particular case of the more general Oberdieck's multiple cover formula conjectured by G. Oberdieck in \cite{oberdieck2022gromov}, and dealing with reduced Gromov-Witten invariants of abelian surfaces. The formula is as follows.

\begin{conj}\label{conj oberdieck formula}\cite{oberdieck2022gromov}
Let $\alpha$ be a tautological class in the cohomology ring $H^\ast(\mathcal{M}_{g,n})$ and $\gamma_i\in H^\ast(\CC A,\RR)$ be arbitrary insertions, where $\CC A$ is an abelian surface. Then, one has
$$\langle \alpha ;\gamma_1,\dots,\gamma_n\rangle^{\CC A}_{g,C}=\sum_{k| C} k^{3g-3+n-\deg\alpha}\langle \alpha ;\varphi_k(\gamma_1),\dots,\varphi_k(\gamma_n)\rangle^{\CC A_k}_{g,\varphi_k(C/k)},$$
where $\varphi_k: H^\ast(\CC A,\RR)\to H^\ast(\CC A_k,\QQ)$ is a morphism of algebra that preserves the canonical element and takes $C/k$ to a primitive curve class of the same norm.
\end{conj}

As it allows any cohomology class insertion, this formula is very general, and is compatible with expressions for the invariants computed in \cite{bryan2018curve}. Oberdieck's multiple cover formula was conjectured following the multiple cover formulas for reduced Gromov-Witten invariants of K3 surfaces, which have a slightly more ancient history. The $g=0$ multiple cover formula for the latter goes back to Aspinwall-Morrison and Yau-Zaslow in the 90's and was finally proven in \cite{klemm2010noether}. The multiple cover formula where the class $\alpha$ is chosen to be a $\lambda$-class insertion goes back to Katz-Klemm-Vafa \cite{katz1999m} and was proven by R. Pandharipande and R. Thomas \cite{pandharipande2016katz}. Using the latter families of examples, G. Oberdieck and R. Pandharipande proposed in \cite{oberdieck2016curve} a multiple cover formula for K3 surfaces, which is proven for the divisibility $2$ case in \cite{bae2021curves} by Y. Bae and T-H. Buelles.

Taking $\alpha=1$ (no class from $\M_{g,n}$), $n=g$ and every $\gamma_i$ to be a point insertion, we get formula (ii) in Theorem \ref{theorem multiple cover formulas point}. The relation between tropical curves and pearl diagrams on one side and the correspondence theorem \cite{nishinou2020realization} are enough to prove this particular case of Oberdieck's formula. This approach may provide a new perspective on the tackling of multiple cover formulas. These are also arguments to support the validity of the more general formula.

\subsubsection{Relation to other Gromov-Witten invariants.} In the toric case, P. Bousseau proved in \cite{bousseau2019tropical} that the tropical refined invariants are related to the Gromov-Witten invariants with insertions of $\lambda$-classes, which are the Chern classes of the Hodge bundle. The correspondence is as follows: consider the refined invariant $BG_{g_0,d}(q)$ for genus $g_0$ degree $d$ curves in the projective plane, and set $q=e^{iu}$. Expanding in $u$ yields the series of Gromov-Witten invariants with $3d-1+g_0$ point insertions and a $\lambda_{g-g_0}$ class insertion.

\medskip

If in the conjectured multiple cover formula, we now take $n=g_0$ point insertions, and $\alpha=\lambda_{g-g_0}$. The multiple cover formula from \cite{oberdieck2022gromov} makes the following prediction:
$$\langle \lambda_{g-g_0};\mathrm{pt}^{g_0}\rangle_{g,C}=\sum_{k|C} k^{2g-3+2g_0}\langle \lambda_{g-g_0};\mathrm{pt}^{g_0}\rangle_{g,(C/k)_\mathrm{prim}}.$$
For a fixed choice of $C$ and $g_0$, let
$$\R_{g_0,C}=\sum _{g\geqslant g_0}\langle \lambda_{g-g_0};\mathrm{pt}^{g_0}\rangle_{g,C} u^{2g-2},$$
which is the generating series of the Gromov-Witten invariants with point insertions and a $\lambda$-class. Its first term is $N_{g_0,C}u^{2g_0-2}$. Multiplying by $u^{2g-2}$ and making a sum over all $g\geqslant g_0$ of multiple cover relations, we get
\begin{align*}
\R_{g_0,C}(u) & =\sum_{g\geqslant g_0}\sum_{k|C} k^{2g_0-1}\langle \lambda_{g-g_0};\mathrm{pt}^{g_0}\rangle_{g,(C/k)_\mathrm{prim}}(ku)^{2g-2} \\
 & = \sum_{k|C} k^{2g_0-1}\R_{g_0,(C/k)_\mathrm{prim}}(ku). \\
\end{align*}

Some of these invariants have been computed in \cite{bryan2018curve} in the FLS case. However, as \cite[Proposition 2]{bryan2018curve} allows one to trade the FLS condition to point constraints, we stay in the case of point insertions. The values are already known in the following cases:
\begin{itemize}
\item $\langle \lambda_{g-g_0};\mathrm{pt}^{g_0}\rangle_{g,C}$ for any $g_0$ if $C$ is primitive,
\item $\langle \lambda_{g-2};\mathrm{pt}^{2}\rangle_{g,C}$ for any curve class.
\end{itemize}
Furthermore, for the two above cases, through the change of variable $q=e^{iu}$, one has
$$\R_{g_0,C}(u)=(-1)^{g_0-1}R_{g_0,C}(q),$$
and values are compatible with the multiple cover formula in the second case.

This suggests that \cite{bousseau2019tropical} can be generalized to the realm of abelian surfaces, where the tropical multiple cover formula for refined invariants should be equivalent to the multiple cover formula for Gromov-Witten invariants with insertions of $\lambda$-classes. This also suggests that the generalization of P. Bousseau's correspondence result \cite{bousseau2019tropical} to the abelian surface case should use the refinement $\sum_{k|\delta_\Gamma}k^{2g_0-2}\varphi(k)\prod_V(q^{m_V/2k}-q^{-m_V/2k})$ of $\delta_\Gamma\prod m_V$ rather than the naive refinement $\delta_\Gamma\prod(q^{m_V/2}-q^{-m_V/2})$ used in \cite{blomme2022abelian1,blomme2022abelian2}, leading to the following conjecture.

\begin{conj}\label{conj corresp}
Through the change of variable $q=e^{iu}$, one has
$$\R_{g_0,C}=(-1)^{g_0-1} R_{g_0,C}(q).$$
\end{conj}

The multiple cover formula for refined invariants (Theorem \ref{theorem multiple cover formulas point} (iv)) can be seen as evidence for both Conjectures \ref{conj oberdieck formula} and \ref{conj corresp} in the case of $\lambda$-class insertions.

\subsubsection{Regularity results.} Using the multiple cover formulas, it is possible to prove the following regularity results. The first is about the generating series of the invariants. Let $g$ and $d$ be fixed positive integers. We consider the generating series for the invariants $N_{g,(d,dn)}$ and $N^{FLS}_{g,(d,dn)}$:
$$F_{g,d}(y)=\sum_{n=1}^\infty N_{g,(d,dn)}y^n, \ F^{FLS}_{g,d}(y)=\sum_{n=1}^\infty N^{FLS}_{g,(d,dn)}y^n.$$
The generating series $F_{g,1}$ and $F^{FLS}_{g,1}$ have already been computed in \cite{bryan1999generating}. As they are a polynomial in $DG_2(y)$, with $G_2$ the first Eisenstein series and $D=y\frac{\dd}{\dd y}$, they are quasi-modular forms. Quasi-modularity is a desirable property since it for instance allows for a strong control of the coefficients. We generalize this result to the generating series for the non-primitive case.

\begin{theom}\ref{theorem quasimodularity}
The series $F_{g,d}$ and $F^{FLS}_{g,d}$ are quasi-modular forms for some finite index subgroup of $SL_2(\ZZ)$ containing $\left(\begin{smallmatrix} 1 & 1 \\ 0 & 1 \\ \end{smallmatrix}\right)$ and they have the following expressions:
\begin{align*}
F_{g,d}(e^{2i\pi\tau}) & =g\sum_{\substack{k|d \\ 0\leqslant l < k^2 }}\left(\frac{d}{k}\right)^{4g-3} \big(DG_2(e^{2i\pi(\tau+l)/k^2})\big)^{g-1}, \\
F^{FLS}_{g,d}(e^{2i\pi\tau}) & =\sum_{\substack{k|d \\ 0\leqslant l < k^2 }}\left(\frac{d}{k}\right)^{4g-1} \big( DG_2(e^{2i\pi(\tau+l)/k^2})\big)^{g-2}D^2G_2(e^{2i\pi(\tau+l)/k^2}). \\
\end{align*}
\end{theom}

We finish proving an analog of the result of E. Brugall\'e and A. Puentes \cite{brugalle2020polynomiality} in the abelian setting. They prove a polynomiality statement for the refined invariants of Hirzebruch surfaces and some weighted projective spaces. In the case of the projective plane, the result is as follows: for any $g$ and $p$, the coefficient of codegree $p$ $\langle BG_{g,d}(q)\rangle_p$ of the refined invariant $BG_{g,d}(q)$ is a polynomial in $d$ for $d$ big enough. In the abelian setting, we prove the following.

\begin{theom}\ref{theorem polynomiality}
The coefficient $\langle R_{g,(d_1,d_2)}\rangle_p$ of codegree $p$ in $R_{g,(d_1,d_2)}$ is a polynomial in $d_1 d_2$ for $n>\max(g-1,2p)$.
\end{theom}

The result from \cite{brugalle2020polynomiality} may be seen as adual version to the G\"ottsche conjecture \cite{gottsche1998conjectural} (proved in \cite{tzeng2010proof} and \cite{kool2011short}). The latter asserts that the number of curves with a prescribed number of nodes $\delta$ in a linear system $L$ on a surface $S$ passing through a given number of points is a polynomial, provided the linear system is sufficiently ample. For fixed genus, we do not have polynomiality, but thanks to \cite{brugalle2020polynomiality}, we recover some polynomiality when considering the refined invariants. Furthermore, G\"ottsche's conjecture asserts more: the polynomial is in some sense universal as it only depends on the numbers $c_2(S)$, $K_S^2$, $L^2$ and $L\cdot K_S$ of the pair $(S,L)$. The obtained polynomials are known as \textit{node polynomials}. Such an universality property for polynomials from \cite{brugalle2020polynomiality} remains out of reach due to a lack of definition of the refined invariants in a general setting. Expanding the range of examples is a first step to hint a dual G\"ottsche conjecture.

\subsection{Structure of the paper}

The paper is organized as follows. The second section is independent of the rest of the paper and contains considerations on the arithmetic convolution, some generalities on quasi-modular forms, and the definition of pearl diagrams, which are some decorated graphs. The third section gathers result from \cite{blomme2022abelian1,blomme2022abelian2} to recall the definition of the tropical invariants we compute. Fourth and fifth sections constitute the core of the paper and explain the correspondence between tropical curves and pearl diagrams, and how to use them to prove the multiple cover formulas. In the last section, we recover the expression of the invariants for primitive classes using tropical geometry, prove the quasi-modularity of the generating series and polynomiality for coefficients of refined invariants.

\medskip

\textit{Acknowledgments.} The author is grateful to R. Pandharipande and G. Oberdieck for pointing out the existence of Oberdieck's multiple cover formula as well as helpful discussions on the matter. The author also thanks Ilia Itenberg for reviewing a first version of the paper, and the anonymous referee for helpful remarks that improved greatly the paper. Research was supported in part by the SNSF grant 204125.

\section{Little toolbox}

\subsection{A small disgression on the arithmetic convolution}
\label{section convolution}

\subsubsection{Classical convolution.} The classical \textit{arithmetic convolution} is defined as follows: for $u,v:\NN^*\to\RR$ functions, we set
$$(u\ast v)(n)=\sum_{k|n} u(k) v(n/k).$$
This is an associative operation, with neutral element the Dirac function that takes value $1$ at $1$ and $0$ elsewhere. A function is invertible if and only if its value at $1$ is non-zero. We say that a function $u$ is multiplicative if whenever $\mathrm{gcd}(n,m)=1$ we have $u(nm)=u(n)u(m)$. Such a function needs only to be defined on powers of prime numbers. The set of multiplicative functions is stable by convolution. We introduce the following multiplicative functions:
\begin{itemize}
\item the sum of divisors $\sigma_1(n)=\sum_{d|n}d$, and in particular $\sigma_1(p^r)=1+p+\cdots+p^r=\frac{p^{r+1}-1}{p-1}$,\\
\item the power functions $\epsilon_\alpha(n)=n^\alpha$, in particular $\epsilon_0=\mathds{1}$ is constant equal to $1$ and $\epsilon_1=\mathrm{id}$,
\item the Euler function $\varphi(n)=\sum_{\mathrm{gcd}(d,n)=1} 1$, and in particular $\varphi(p^r)=p^{r-1}(p-1)$, but also $\sum_{k|n}\varphi(k)=n$, or in other terms $\varphi\ast\mathds{1}=\mathrm{id}$.

\end{itemize}

\subsubsection{Deformed convolution.} We consider the following variation on the arithmetic cnvolution. If we now consider functions $U,V:\NN^*\to\ZZ[q^{\pm 1/2}]$ with value in the ring of Laurent polynomials, we set:
$$(U\ast V)_n(q)=\sum_{k|n}U_k(q)V_{n/k}(q^k).$$
This operation is not commutative anymore, but it can be checked that it is still associative. This deformation restricts to the usual arithmetic convolution on the set of constant polynomials.

\subsubsection{$\NN^*$-action and convolution.} We finish with a last variant of the arithmetic convolution, which is a way to define a module structure on some space of functions. Let $X$ be a set endowed with an action of the multiplicative monoid $\NN^*$. In the paper, $X$ is taken to be the set of curve classes, or the set of pearl diagrams. We say that an element $x\in X$ is divisible by $n\in\NN^*$ if there is a $x'\in X$ such that $x=n\cdot x'$. We assume that each element of $x$ has only a finite number of divisors, and that the maps $x\to n\cdot x$ are injective, so that we can speak of $x/n$ if $x$ is divisible by $n$. Let $u:\NN^*\to \RR$ be an arithmetic function, and $F:X\to \RR$ be a function defined on $X$. We define the convolution $u\ast F:X\to\RR$ of $u$ with $F$ as follows:
$$(u\ast F)(x)=\sum_{n|x}u(n)F(x /n).$$
This multiplication is associative in the sense that $(u\ast v)\ast F=u\ast(v\ast F)$. We proceed similarly if the codomain is $\ZZ[q^{\pm 1/2}]$ rather than $\RR$ to deform the convolution: if $F:X\to\ZZ[q^{\pm 1/2}]$, we set
$$(u\ast F)(x)(q)=\sum_{n|x}u(n)F(x /n)(q^n).$$

\subsection{About quasi-modular forms}

We here give some elementary results about quasi-modular forms and we refer to the litterature for a more complete introduction on the subject. See for instance \cite{serre1970cours} or \cite{gunning2016lectures}.

\subsubsection{Modular forms for $SL_2(\ZZ)$.} Modular forms are holomorphic functions on the Poincar\'e half-plane $\mathbb{H}=\{z\in\CC:\im z>0\}$ that satisfy some transformation property under the action of $SL_2(\ZZ)$ by homography. The latter are maps of the form $\tau\mapsto\frac{a\tau+b}{c\tau+d}$. More precisely, a function $f:\mathbb{H}\to\CC$ is a modular form of weight $2k$ if for any $\left(\begin{smallmatrix} a & b \\
c & d \\
\end{smallmatrix}\right)\in SL_2(\ZZ)$ we have
$$f\left(\frac{a\tau+b}{c\tau+d}\right)=(c\tau+d)^{2k}f(\tau),$$
and is bounded as $\im\tau$ goes to $\infty$. As these functions are $1$-periodic, setting $y=e^{2i\pi\tau}$, they can be expressed as series in this new variable $y$. It can be shown that modular forms for $SL_2(\ZZ)$ are homogeneous polynomials in the Eisenstein series $G_4(y)$ and $G_6(y)$ (See \cite[Theorem 5 chapter IV]{gunning2016lectures}).

\subsubsection{Quasi-modular forms for $SL_2(\ZZ)$.} Quasi-modular forms are functions that satisfy a different transformation law under the action of $SL_2(\ZZ)$. It can be shown that quasi-modular forms are exactly the polynomials in the Eisenstein series $G_2$, $G_4$ and $G_6$. Moreover, they form a ring which is stable by the differential operator $D=y\frac{\dd}{\dd y}$. Quasi-modularity is a desirable property since it allows for a rather strong control of the coefficients. For instance, they are at most of polynomial growth.

\subsubsection{Modularity for congruence subgroups.} The above considerations generalize as follows: the transformation law may only be satisfied for a finite index subgroup $G$ of $SL_2(\ZZ)$, in which case we say that the function $f$ is a (quasi-)modular form for $G$. Examples of such subgroups are provided by the congruence subgroups:
$$\Gamma(N)=\left\{ \left(\begin{smallmatrix} a & b \\ c & d \\\end{smallmatrix}\right)\in SL_2(\ZZ) \mid  \left(\begin{smallmatrix} a & b \\ c & d \\\end{smallmatrix}\right)\equiv \left(\begin{smallmatrix} 1 & 0 \\ 0 & 1 \\\end{smallmatrix}\right)\ \mathrm{mod}\ N\right\}.$$
Conversely, for a given function $f$ that is quasi-modular for some subgroup of $SL_2(\ZZ)$, we denote by $\Gamma(f)$ the maximal subgroup of matrices that satisfy the law transformation.

\subsubsection{Constructing new quasi-modular forms out of old ones.} We have the following lemma which tells how affine transformations affect the modular group.

\begin{lem}\label{lemma uasi-modularity transformation}
Let $f$ be a (quasi-)modular form for some subgroup $\Gamma(f)$ of $SL_2(\ZZ)$. Then for any choice of $(r,s)\in\QQ^*\times\QQ$, the function $g(\tau)=f(r\tau+s)$ is (quasi-)modular and $\Gamma(g)$ contains $\left(\begin{smallmatrix}
1/r & -s/r \\
0 & 1 \\
\end{smallmatrix}\right)\Gamma(f)\left(\begin{smallmatrix}
r & s \\
0 & 1 \\
\end{smallmatrix}\right)\cap SL_2(\ZZ)$.
\end{lem}

\begin{proof}
We prove the result for modular forms. The result for quasi-modular forms is obtained with the same manipulations. Assume $f$ is a modular form for $\Gamma(f)$ of weight $2k$. We have the following equality in $SL_2(\RR)$:
$$\begin{pmatrix}
r & s \\
0 & 1 \\
\end{pmatrix}\begin{pmatrix}
a & b \\
c & d \\
\end{pmatrix}\begin{pmatrix}
1/r & -s/r \\
0 & 1 \\
\end{pmatrix}=\begin{pmatrix}
a+\frac{r}{s}c & rb+s(d-a)-\frac{s^2 }{r}c \\
\frac{c}{r} & d-\frac{s}{r}c \\
\end{pmatrix}.$$
If $\left(\begin{smallmatrix} a & b \\ c & d \\ \end{smallmatrix}\right) \in\left(\begin{smallmatrix}
1/r & -s/r \\
0 & 1 \\
\end{smallmatrix}\right)\Gamma(f)\left(\begin{smallmatrix}
r & s \\
0 & 1 \\
\end{smallmatrix}\right)$, then the above product lies in $\Gamma(f)$ and we can apply the modular transformation law for $f$:
\begin{align*}
g\left(\frac{a\tau+b}{c\tau+d}\right) & = f\left(r\frac{a\tau+b}{c\tau+d}+s\right) \\
& = f\left( \begin{pmatrix}
r & s \\
0 & 1 \\
\end{pmatrix}\begin{pmatrix}
a & b \\
c & d \\
\end{pmatrix}\begin{pmatrix}
1/r & -s/r \\
0 & 1 \\
\end{pmatrix}\cdot(r\tau+s)\right) \\
& = \left( \frac{c}{r}(r\tau+s)+d-\frac{s}{r}c\right)^{2k}f(r\tau+s) \\
& = (c\tau+d)^{2k}g(\tau).\\
\end{align*}
Thus, $g$ is a modular form of weight $2k$ and its modular group contains the elements of the above subgroup.
\end{proof}

\subsection{Abstract marked pearl diagrams}
\label{section pearl diagrams definition}

Pearl diagrams are similar to floor diagrams introduced in \cite{blomme2021floor} and also to the diagrams considered in \cite{blomme2024bielliptic} or \cite{bohm2020counts}. The main difference with \cite{blomme2021floor} is that the stretching direction loops on itself.

\begin{defi}\label{def-pearl-diagram}
A \textit{pearl diagram} $\Dfk$ (or just \textit{diagram}) is a connected oriented weighted labeled graph with the following properties:
	\begin{enumerate}[label=(\alph*)]
	\item Vertices are labeled by elements of $\{1,\dots,g\}$, and split into two disjoint families:
		\begin{itemize}
		\item the set $V_\circ(\Dfk)$ of \textit{flat vertices}, which are bivalent and denoted with a circle $\circ$ when needed,
		\item the set $V_\square(\Dfk)$ of \textit{pearls}, not necessarily bivalents which are assigned a degree $d_\pfk\in\ZZ_{>0}$, denoted by a square $\square$ when needed.
		\end{itemize}
	\item Each edge $e$ is assigned a weight $w_e\in\ZZ_{>0}$, such that the graph is balanced: sum of incoming weights is equal to sum of outgoing weights.
	\item Each edge $e$ with extremities labeled by $i_e^\pm$ is assigned a length $\ell_e\in\ZZ_{\geqslant 0}$ satisfying the following:
		\begin{itemize}
		\item if $i_e^-<i_e^+$, then $\ell_e\geqslant 0$,
		\item if $i_e^-\geqslant i_e^+$, then $\ell_e>0$.
		\end{itemize}
	\item The complement of flat vertices is connected and without cycle.
	\end{enumerate}
\end{defi}

Choose $0<x_1<\cdots<x_g<1$ yielding $g$ points in $\RR/\ZZ$. The choice of a pearl diagram $\Dfk$ induces a map $\pi_\Dfk:\Dfk\to\RR/\ZZ$ as follows:
	\begin{itemize}
	\item A vertex with label $1\leqslant i\leqslant g$ is mapped to $x_i$.
	\item An edge $e$ with extremities labeled $i_e^-$ and $i_e^+$ and length $\ell_e$ is mapped to the path from $x_{i_e^-}$ to $x_{i_e^+}$ that passes $\ell_e$ times over $0\in\RR/\ZZ$:
	$$t\in [x_{i_e^-}; \ell_e+x_{i_e^+}]\longmapsto t\in\RR/\ZZ.$$
	\end{itemize}
The map $\pi_\Dfk$ is a tropical cover in the sense of \cite[Definition 3.1]{bohm2020counts}: it is a graph map that satisfies the balancing condition. Tropical covers are maps between graphs that satisfy a balancing condition. They were studied in \cite{caporaso2014gonality} and \cite{cavalieri2010tropical}. We now define the genus of a pearl diagram. As this quantity is already hidden in Definition \ref{def-pearl-diagram}, we give it inside a lemma rather than a definition.

\begin{lem}\label{lem-genus-pearl-diagram}
Let $\Dfk$ be a pearl diagram. The genus of the pearl diagram is defined to be $g(\Dfk)=b_1(\Dfk)+|V_\square(\Dfk)|$. It matches the total number of vertices $g$.
\end{lem}

\begin{proof}
The Euler characteristic of $\Dfk$ is $1-b_1(\Dfk)$. By assumption, the complement of flat vertices is connected and without cycle. Thus, its Euler characteristic is $1$:
$$\left(1-b_1(\Dfk)\right)+|V_\circ(\Dfk)|=1,$$
using the fact that flat vertices are bivalent. We deduce that
$$b_1(\Dfk)+|V_\square(\Dfk)| = |V_\circ(\Dfk)|+|V_\square(\Dfk)|=|V(\Dfk)|=g.$$
\end{proof}

Let $\Dfk$ be a pearl diagram. Its bidegree $(d_1,d_2)$ is the following pair of integers:
	\begin{itemize}
	\item $d_1=\displaystyle\sum_{e\in E(\Dfk)} w_e\ell_e$ is the degree of the tropical cover $\pi_\Dfk$,
	\item $d_2=\displaystyle\sum_{\pfk\in V_\square(\Dfk)} d_\pfk$ is the sum of degrees of pearls.
	\end{itemize}
	We define the gcd $\delta_\Dfk$ of $\Dfk$ to be the gcd of edge weights and pearl degrees:
	$$\delta_\Dfk = \gcd\left( (w_e)_e,(d_\pfk)_\pfk \right).$$
	A diagram is said to be \textit{primitive} if its gcd is $1$. Otherwise, it is possible to write $\Dfk=\delta_\Dfk\cdot \widehat{\Dfk}$, where $\widehat{\Dfk}$ is a primitive diagram, and $\cdot$ denotes the action of $\NN^*$ over the set of diagrams, multiplying the weight of every edge and degree of every pearl by a common integer. Notice that the lengths $\ell_e$ are unchanged.

\begin{expl}
We depict several examples of pearl diagrams on Figure \ref{figure examples pearl diagrams}, up to the choice of $d_\pfk$ for square vertices. The orientation of the graphs goes counterclockwise, and the graph is displayed so that the map to $\RR/\ZZ$ is apparent. The weights are only marked if they are bigger than $2$. The lengths of the edges can be recovered as the number of times they intersect a half-line going away from the center and passing between the marked vertices $g$ and $1$.
	\begin{itemize}
	\item The first diagram has genus $2$, and the degree of the tropical cover is $2$. Both edges have length $1$.
	\item The second diagram has genus $4$ and the degree of te tropical cover is $3$. The edges between $4-1$ and $1-3$ have length $1$.
	\item The last diagram has genus $7$, the tropical cover has degree $3$, and only the edge $7-1$ has length $1$.
	\end{itemize}
\end{expl}

\begin{figure}
\begin{center}
\begin{tabular}{ccc}
\begin{tikzpicture}[line cap=round,line join=round,x=0.75cm,y=0.75cm]
	\pearl (1) at (-150:2) label=1;
	\flatpearl (2) at (-30:2) label=2;
	
	\draw (1) to[out=-60,in=0,looseness=2.5] (90:1) to[out=180,in=-120,looseness=2.5] (2) ;
	\draw (2) to[out=60,in=0] (90:2) to[out=180,in=-240] (1) ;
\end{tikzpicture} & 
\begin{tikzpicture}[line cap=round,line join=round,x=0.75cm,y=0.75cm]
	\pearl (1) at (180:2) label=1;
	\flatpearl (2) at (-90:2) label=2;
	\pearl (3) at (0:2) label=3;
	\flatpearl (4) at (90:2) label=4;
	
	\draw (1) to[out=-90,in=180] (2);
	\draw (2) to[out=0,in=-90] (3);
	\draw (3) to[out=90,in=0] node[midway,above right] {$2$} (4);
	\draw (1) to[out=90,in=180] node[midway,above left] {$2$} (4);
	\draw (1) to[out=-60,in=-90] (0:1.5) to[out=90,in=90] (180:1.5) to[out=-90,in=-120] (3);
	\end{tikzpicture}
&
\begin{tikzpicture}[line cap=round,line join=round,x=0.75cm,y=0.75cm]
	\pearl (1) at (120:3) label=1;
	\pearl (2) at (180:1.5) label=2;
	\flatpearl (3) at (-120:3) label=3;
	\flatpearl (4) at (-90:1.5) label=4;
	\pearl (5) at (-40:3) label=5;
	\pearl (6) at (0:3) label=6;
	\flatpearl (7) at (60:3) label=7;
	\draw (1) to[bend right] node[midway,right] {$2$} (2) ;
	\draw (1) to[out=-130,in=130] (3) ;
	\draw (3) to[bend right] (5) ;
	\draw (2) to[bend right] (4) ;
	\draw (4) to[bend right] (6) ;
	\draw (2) to[out=-90,in=180] (-90:2.3) to[out=0,in=180] (5) ;
	\draw (5) to[bend right]  node[midway,right] {$2$} (6) ;
	\draw (6) to[bend right] node[midway,above] {$3$} (7) ;
	\draw (7) to[bend right]  node[midway,above] {$3$} (1) ;
	\end{tikzpicture} \\
\\
$(a)$ & $(b)$ & $(c)$ \\
\end{tabular}
\caption{\label{figure examples pearl diagrams}Examples of pearl diagrams. The first is of respective genus $2$, $4$ and $7$.}
\end{center}
\end{figure}
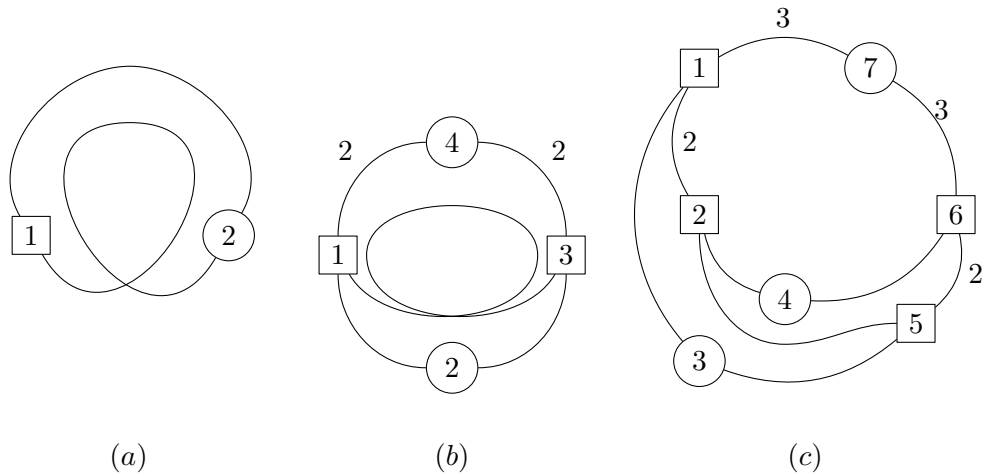

To deal with the invariants for curves in a fixed linear system (FLS), we need the following adaptation of the notion of diagram, which we call \textit{FLS-diagram}. It differs with Definition \ref{def-pearl-diagram} at the two points (a) and (d), replaced by the following:
	\begin{itemize}
	\item[(a')] Vertices are labeled by $\{1,\cdots,g-2,\infty\}$, with the marking $\infty$ lying on a pearl.
	\item[(d')] The complement of flat vertices is connected and contains a unique cycle $\gamma_\Dfk=\sum_e\epsilon_e e$, where $\epsilon_e=0,\pm 1$.
	\end{itemize}
	For a FLS-diagram, its genus $b_1(\Dfk)+|V_\square(\Dfk)|$ is equal to $g$. Let $\Lambda_\Dfk=\sum_e \epsilon_e\ell_e$ be the absolute value of the class realized by $\pi_\Dfk(\gamma_\Dfk)$ inside $\pi_1(\RR/\ZZ)\simeq\ZZ$.

\section{Tropical tori, tropical curves, enumeration and multiplicities}
\label{section tropical invariants}

This section is a reminder of the results proven in \cite{blomme2022abelian1,blomme2022abelian2}, to make the present paper self-contained.

\subsection{Tropical tori and tropical curves}

\subsubsection{Tropical tori.} Let $N$ be a rank $2$ lattice, which may be assumed to be $\ZZ^2$, and $N_\RR=N\otimes\RR$ the associated vector space. Let $\Lambda\subset N_\RR$ be a rank $2$ lattice, which amounts to the choice of a $2\times 2$ matrix, and $S:\Lambda\hookrightarrow N_\RR$ be the inclusion.

\begin{defi}
The quotient $\TT A=N_\RR/\Lambda$ is called a tropical torus.
\end{defi}

\subsubsection{Tropical curves.} An \textit{abstract tropical curve} is a metric graph $\Gamma$. It is \textit{irreducible} if $\Gamma$ is connected. The \textit{genus} of an irreducible abstract tropical curve is its first Betti number. The valency of a vertex is the number of adjacent edges.

\begin{defi}
Let $\TT A=N_\RR/\Lambda$ be a tropical torus. A parametrized tropical curve inside $\TT A$ is a map $h:\Gamma\to\TT A$ where $\Gamma$ is an abstract tropical curve such that:
\begin{itemize}
\item The map $h$ is affine with slope in $N$ on each edge of $\Gamma$. The slope $\partial_e h$ of $h$ along an edge $e$ takes the form $w_e u_e$ where $w_e\in\NN$ is called the weight of the edge, and $u_e\in N$ is some primitive vector.
\item The map $h$ is balanced: at each vertex $V$, one has $\displaystyle\sum_{e\ni V}w_eu_e=0$.
\end{itemize}
The genus of the parametrized curve is the genus of $\Gamma$.
\end{defi}

\begin{rem}
Some authors include the data of a genus at the vertices of the curve. We do not need vertex genus for tropical curves in this paper. The vertex genus was present for the pearls of a pearl diagram.
\end{rem}

Taking each edge $e$ with its slope $\partial_e h$, a parametrized tropical curve realizes a $1$-chain $\sum_e (\partial_e h)e$ in $\TT A$ with coefficients in $N$. Indeed, as the slope is reversed if the orientation is reversed, the quantity $(\partial_e h)e$ does not depend on the choice of orientation of the $e$. The balancing condition states that this $1$-chain is in fact a cycle. The class realized in $H_1(\TT A,N)$ (first homology group with coefficients in $N$) is called the \textit{degree} of the tropical curve. The homology group $H_1(\TT A,N)$ is in fact the tropical homology group $H_{1,1}(\TT A)$ as defined in \cite{itenberg2019tropical}.

Using universal coefficients, we have that $H_1(\TT A,N)\simeq \Lambda\otimes N$, so that the degree of a tropical curve can be seen as a linear map $\Lambda^*\to N$. Choosing basis of $\Lambda$ and $N$, we can consider its matrix $C$. We have a symmetric bilinear form on $\Lambda\otimes N$:
$$\lambda\otimes n,\lambda'\otimes n'\longmapsto \mathrm{det}_\Lambda (\lambda,\lambda')\mathrm{det}_N(n,n'),$$
where $\det_\Lambda$ (resp. $\det_N)$ is the unique skew-symmetric billinear form on $\Lambda$ (resp. $N$), well-defined up to sign, \textit{i.e.} choice of orientation. As $N$ and $\Lambda$ are both lattices in $N_\RR$, if their orientation is induced by the same orientation of $N_\RR$, the above bilinear form is well-defined, and is called the \textit{intersection form} on $H_1(\TT A,N)$. Given two transverse tropical curves in $\TT A$ of respective degrees $C,C'$, the intersection number $C\cdot C'$ is obtained by counting the intersection points $p$ with multiplicity $|\det_N(u_1(p),u_2(p))|$, where $u_1(p),u_2(p)$ are the slopes of the edges of the curves at $p$, as in the tropical Bezout theorem \cite[Section 3.1]{brugalle2014bit}. For an element of $\Lambda\otimes N$ represented by a matrix $C$, its self-intersection is in fact $C^2=|\det C|$.

Having chosen some bases, the image of $\Lambda\hookrightarrow N_\RR$ is spanned by the columns of a matrix $S$. The following propositions give a concrete way to compute the degree of a parametrized tropical curve, and give a condition on the existence of tropical curves of a given class $C$ in the tropical torus $\TT A$.

\begin{prop}
Let $h:\Gamma\to\TT A$ be a parametrized tropical curve. Its degree, viewed as a map $\Lambda^*\to N$ corresponds to the matrix that maps an element $\lambda^*\in\Lambda^*$ to the sum of the slopes of the edges oriented so that they intersect positively a path Poincar\'e dual to $\lambda^*$.
\end{prop}

Concretely, taking a fundamental domain for $\TT A$, the columns of the matrix $C$ are obtained by adding the slopes of the edges intersecting the right side and top side of the parallelogram respectively. This in fact amounts to compute the intersection numbers of $C$ with a basis of $H_1(\TT A,N)$.

\begin{prop}
A class $C\in\Lambda\otimes N$ is realizable by curves in $\TT A$ if and only if the matrix $CS^T$ is symmetric.
\end{prop}

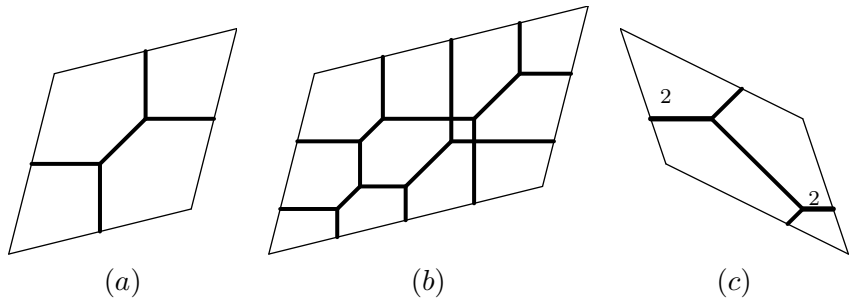
\begin{figure}
\begin{center}
\begin{tabular}{ccc}
\begin{tikzpicture}[line cap=round,line join=round,>=triangle 45,x=0.3cm,y=0.3cm]
\clip(0,0) rectangle (10,10);
\draw [line width=0.5pt] (0,0)-- (8,2);
\draw [line width=0.5pt] (0,0)-- (2,8);
\draw [line width=0.5pt] (8,2)-- (10,10);
\draw [line width=0.5pt] (2,8)-- (10,10);

\draw [line width=1.5pt] (4,1)-- (4,4);
\draw [line width=1.5pt] (1,4)-- (4,4);
\draw [line width=1.5pt] (4,4)-- (6,6);
\draw [line width=1.5pt] (6,6)-- (6,9);
\draw [line width=1.5pt] (6,6)-- (9,6);

\begin{scriptsize}

\end{scriptsize}
\end{tikzpicture}
&
\begin{tikzpicture}[line cap=round,line join=round,>=triangle 45,x=0.3cm,y=0.3cm]
\clip(0,0) rectangle (14,11);
\draw [line width=0.5pt] (0,0)-- (12,3);
\draw [line width=0.5pt] (0,0)-- (2,8);
\draw [line width=0.5pt] (2,8)-- (14,11);
\draw [line width=0.5pt] (12,3)-- (14,11);

\draw [line width=1.5pt] (3,2)-- (4,3);
\draw [line width=1.5pt] (4,3)-- (4,5);
\draw [line width=1.5pt] (4,3)-- (6,3);
\draw [line width=1.5pt] (6,3)-- (8,5);
\draw [line width=1.5pt] (0.5,2)-- (3,2);
\draw [line width=1.5pt] (3,0.75)-- (3,2);
\draw [line width=1.5pt] (4,5)-- (5,6);
\draw [line width=1.5pt] (1.25,5)-- (4,5);
\draw [line width=1.5pt] (5,6)-- (5,8.75);
\draw [line width=1.5pt] (5,6)-- (9,6);
\draw [line width=1.5pt] (9,6)-- (11,8);
\draw [line width=1.5pt] (11,8)-- (13.25,8);
\draw [line width=1.5pt] (11,8)-- (11,10.25);
\draw [line width=1.5pt] (8,5)-- (8,9.5);
\draw [line width=1.5pt] (8,5)-- (12.5,5);
\draw [line width=1.5pt] (6,3)-- (6,1.5);
\draw [line width=1.5pt] (9,6)-- (9,2.25);

\begin{scriptsize}

\end{scriptsize}
\end{tikzpicture}
&
\begin{tikzpicture}[line cap=round,line join=round,>=triangle 45,x=0.3cm,y=0.3cm]
\clip(0,0) rectangle (10,10);
\draw [line width=0.5pt] (2,4)-- (0,10);
\draw [line width=0.5pt] (2,4)-- (10,0);
\draw [line width=0.5pt] (0,10)-- (8,6);
\draw [line width=0.5pt] (10,0)-- (8,6);

\draw [line width=2pt] (1.33333333333,6)-- (4,6);
\draw [line width=1.5pt] (4,6)-- (8,2);
\draw [line width=1.5pt] (4,6)-- (5.333333333,7.33333333);
\draw [line width=1.5pt] (8,2)-- (7.333333333,1.33333333);
\draw [line width=2pt] (8,2)-- (9.3333333333,2);

\begin{scriptsize}
\draw (2,7) node {$2$};
\draw (8.5,2.5) node {$2$};
\end{scriptsize}
\end{tikzpicture}
\\
$(a)$ & $(b)$ & $(c)$\\
\end{tabular}

\caption[truc]{\label{figure example tropical curves}Three examples of tropical curves in tropical tori of respective degrees $\left(\begin{smallmatrix} 1 & 0 \\ 0 & 1 \\ \end{smallmatrix}\right)$, $\left(\begin{smallmatrix} 2 & 0 \\ 0 & 3 \\ \end{smallmatrix}\right)$ and $\left(\begin{smallmatrix} 2 & 1 \\ 0 & 1 \\ \end{smallmatrix}\right)$ .}
\end{center}
\end{figure}

\begin{expl}
On Figure \ref{figure example tropical curves} we depict three tropical curves along with their respective degree.
\end{expl}

\subsubsection{Simple tropical curves.} The enumeration of curves solution to the enumerative problems that we consider only involves \textit{simple} tropical curves, whose definition we now recall.

\begin{defi}
A parametrized tropical curve $h:\Gamma\to\TT A$ is \textit{simple} if $\Gamma$ is trivalent and $h$ is an immersion.
\end{defi}

\begin{prop}
The dimension of the deformation space of a simple tropical curve of genus $g$ is $g$. If the curve is not simple but $h$ is still an immersion, the dimension is strictly less.
\end{prop}

Let $C$ be a class with $C^2\neq 0$ and let $\TT A$ be a generic tropical abelian surface with curves in the class $C$. As in the classical case, any tropical curve in $\TT A$ of degree $\TT A$ has genus at least $2$: rational curves are mapped to a constant, and elliptic curves only occur for non-generic choices of $\TT A$.

\subsubsection{Multiplicities.} Let $h:\Gamma\to \TT A$ be a simple tropical curve:
\begin{itemize}
\item The gcd of the curve $\delta_\Gamma$ is defined to be the gcd of the weights of its edges: $\delta_\Gamma=\mathrm{gcd}_e w_e$. If $\delta_\Gamma=1$, the curve is called \textit{primitive}. If $\delta_\Gamma\neq 1$, the curve can be written as a $\delta_\Gamma$ times a primitive curve. In particular, the set of tropical curves is endowed with a $\NN^*$ action and we can use definitions from Section \ref{section convolution}.

\item The \textit{usual multiplicity} of the curve is
$$m_\Gamma=\prod_V m_V,$$
where the product is over the trivalent vertices of $\Gamma$, and $m_V=|\det_N(a_V,b_V)|$ is the determinant of two out of the three outgoing edges.

\item The \textit{usual refined multiplicity} is
$$m_\Gamma^\refined=\prod_V (q^{m_V/2}-q^{-m_V/2})\in\ZZ[q^{\pm 1/2}],$$
obtained by replacing the vertex multiplicities by their quantum analog. It is a Laurent polynomial multiplicity, which is used in \cite{blomme2022abelian1,blomme2022abelian2}, up to a denominator $(q^{1/2}-q^{-1/2})^{2g-2}$.

\item The second refined multiplicity
$$M_\Gamma=\sum_{k|\delta_\Gamma}k^{2g-2}\varphi(k)m_\Gamma^\refined(q^{1/k}),$$
where $\varphi$ is the Euler function. This multiplicity is not used in \cite{blomme2022abelian1,blomme2022abelian2} but is going to be more useful. See Remark \ref{remark refined multiplicities definition} below.

\item If $\Sigma\subset\Gamma$ is a subgraph of genus $2$, let $\Lambda_\Gamma^\Sigma$ be the index of $H_1(\Sigma,\ZZ)$ inside $H_1(\TT A,\ZZ)\simeq\Lambda$, obtained by taking the determinant inside $\Lambda$ between two loops of $\Sigma$.
\end{itemize}

\begin{rem}\label{remark refined multiplicities definition}
In \cite{blomme2022abelian1,blomme2022abelian2}, up to the denominator $(q^{1/2}-q^{-1/2})^{2g-2}$, we use the usual refined multiplicity $m_\Gamma^\refined$, obtained by replacing the vertex multiplicities by their quantum analogs. Due to the lack of presence of $\delta_\Gamma$, this multiplicity does not specialize to the complex multiplicity $\delta_\Gamma m_\Gamma$ provided by the correspondence theorem. To get a multiplicity that refines the complex multiplicity, in \cite{blomme2022abelian1}, we take $\delta_\Gamma m_\Gamma^\refined$. For reasons due to the multiple cover formula, we prefer here to use the second refined multiplicity as a refinement of the complex multiplicity $\delta_\Gamma m_\Gamma$. Once we divide by $(q^{1/2}-q^{-1/2})^{2g-2}$, it indeed specializes to $\delta_\Gamma m_\Gamma$ when $q$ goes to $1$ since $\sum_{k|\delta}\varphi(k)=\delta$. Notice that both multiplicities coincide if the curve is primitive.
\end{rem}

\subsubsection{Tropical curves and linear systems.}\label{sec-curves-in-FLS} There exists a theory of tropical line bundles and divisors over tropical abelian surfaces (or varieties), related to the theory of tropical $\theta$-functions. We refer the interested reader to \cite[Section 4]{mikhalkin2008tropical} or to \cite[Section 2.2]{blomme2022abelian2}. For our purpose, we only a criterion to say when two tropical curves of the same degree belong to the same linear system. Such a criterion is the content of \cite[Theorem 2.21]{blomme2022abelian2}, and may be here taken as a definition.

For $i=1,2$, let $h_i:\Gamma_i\to\TT A$ be two tropical curves of the same degree $C$. There exists a covering map $\pi_i:\widetilde{\Gamma}_i\to\Gamma_i$ of $\Gamma_i$ such that $h_i\circ\pi_i$ can be lifted to $N_\RR$, yielding a periodic tropical curves in $N_\RR$. Choose a point $p_0$ in the complement of $\widetilde{h}_1(\widetilde{\Gamma}_1)\cup\widetilde{h}_2(\widetilde{\Gamma}_2)$, an element $\lambda\in\Lambda\subset N_\RR$, and a path $\gamma:p_0\rightsquigarrow p_0+\lambda$ transverse to $\widetilde{h}_1(\widetilde{\Gamma}_1)$ and $\widetilde{h}_2(\widetilde{\Gamma}_2)$. At each intersection point $p$ between $\gamma$ and $\widetilde{h}_1(\widetilde{\Gamma}_1)$ (resp. $\widetilde{h}_2(\widetilde{\Gamma}_2)$), let $u_p$ be the slope of the edge of $\widetilde{h}_1(\widetilde{\Gamma}_1)$ (resp. $\widetilde{h}_2(\widetilde{\Gamma}_2)$) oriented such that the intersection index with $\gamma$ is positive (resp. negative). We consider
$$K_{p_0}(\lambda)=\sum_{p\in\gamma\cap \left(\widetilde{h}_1(\widetilde{\Gamma}_1)\cup\widetilde{h}_2(\widetilde{\Gamma}_2)\right)} \det(p,u_p),$$
where for $\det(p,u_p)$, $p$ and $u_p$ are seen as element of the vector space $N_\RR$. The fact that the above quantity does not depend on the choice of $\gamma$ is shown in \cite{blomme2022abelian2}. The map $K_{p_0}:\Lambda\to\RR$ is linear, and thus defines an element of $\Lambda^*_\RR$. As we have an inclusion $\Lambda\subset N_\RR$, we also have $M=N^*\subset \Lambda^*_\RR$. The map $K_{p_0}$ does not depend on $p_0$ up to the addition of an element of $M$. The following Definition is \cite[Theorem 2.21]{blomme2022abelian2}, but we take it here as definition of linear equivalence.

\begin{defi}
The tropical curves $h_{1/2}:\Gamma_{1/2}\to\TT A$ of common degree $C$ belong to the same linear system if $K_{p_0}\in\Lambda^*_\RR$ belongs to the image of $M$, \textit{i.e.} takes integer values on elements of $N$.
\end{defi}

\subsubsection{A finiteness technical result.}
The following Lemma is similar to \cite[lemma 5.7]{blomme2021floor} and deals with the slopes of the edges of a tropical curve.

\begin{lem}\label{lem-finite-slopes}
Consider a parametrized tropical curve $h:\Gamma\to\TT A$ in a class $C$. The slope of the edges of $(\Gamma,h)$ can only take a finite number of values.
\end{lem}

\begin{proof}
The idea is to use tropical Bezout theorem in $\TT A$ to bound the coordinates of the slopes. Tropical Bezout states that given two transverse tropical curves $h_{1/2}:\Gamma_{1/2}\to\TT A$ in the classes $C_{1/2}$, we have
$$C_1\cdot C_2=\sum_{h_1(p_1)=h_2(p_2)}|\det(u_1(p_1),u_2(p_2))|,$$
where $u_i(p_i)$ is the slope of $h_i$ at $p_i$. Consider a fixed preferred tropical curve $\Gamma_0$ in $\TT A$ in some class $C_0$ whose slope span $\RR^2$. Let $e_1,e_2$ be edges of $\Gamma_0$ and $v_1,v_2$ their slopes, which form a $\RR$-basis of $N_\RR$. Now, consider a curve $h:\Gamma\to\TT A$ in the class $C$ and let $u$ be the slope of one of its edges $e$. It is possible to find two different translates of $\Gamma_0$ such that $e$ and $e_1$ (resp. $e_2$) intersect. Therefore, the intersection number $C\cdot C_0$ is at least $|\det(v_1,u)|$ (resp. $|\det(v_2,u)|$). In particular, coordinates of $u$ are bounded by $C\cdot C_0$. Thus, $u\in N$ can only take a finite number of values.
\end{proof}

\subsection{Enumerative problems and invariants}

We now recall the enumerative problems and the corresponding invariants.

\subsubsection{Curves passing through $g$ points.} The first enumerative problem is to count genus $g$ curves of degree $C\in\Lambda\otimes N$ that pass through a generic configuration $\P$ of $g$ points inside $\TT A$. For a generic choice, \cite{blomme2022abelian1} states that the tropical curves passing through a generic configuration of points $\P$ are simple, so that they have well-defined multiplicities. Moreover, if $h:\Gamma\to\TT A$ is any solution, the complement of $\P$ in $\Gamma$ is connected without cycle.

We consider the following counts that, according to \cite{blomme2022abelian1}, do not depend on the choice of the point configuration $\P$ nor the abelian surface $\TT A$, as long as these choices are generic, and only depend on the class $C$ through its self-intersection $C^2$ and its divisibility:

$$\begin{array}{lr>{\displaystyle}lr>{\displaystyle}l}
\triangleright\text{ counts with fixed gcd: } & N_{g,C,k}=& \sum_{\substack{\Gamma\supset\P \\ \delta_\Gamma=k}}m_\Gamma, & BG_{g,C,k}= & \sum_{\substack{\Gamma\supset\P \\ \delta_\Gamma=k}}m_\Gamma^\refined, \\
\triangleright\text{ counts with true multiplicity: } & N_{g,C}= & \sum_{\Gamma\supset\P}\delta_\Gamma m_\Gamma, & R_{g,C}= & \sum_{\Gamma\supset\P} M_\Gamma, \\
\triangleright\text{ counts with auxiliary multiplicity: } & M_{g,C}= & \sum_{\Gamma\supset\P}m_\Gamma, & BG_{g,C}= & \sum_{\Gamma\supset\P}m_\Gamma^\refined. \\
\end{array}$$

The notation ``$N$" is for the ``N"umber of tropical curves counted with a multiplicity so that it coincides with the complex count. The ``$M$" was just close in the alphabet. The ``$BG$" stands for Block-G\"ottsche, who introduced the refined multiplicity, while the ``$R$"stands for ``R"efined as we use a refined multiplicity, chosen so that we get a multiple cover formula.

\begin{rem}
The definition of the refined count $R_{g,C}$ differs from the definition in \cite{blomme2022abelian1} since we use the second refined multiplicity $M_\Gamma$ as a refinement of the complex multiplicity rather than the naive refinement $\delta_\Gamma m_\Gamma^\refined$. For a fixed value of the gcd of the curve, the multiplicity $M_\Gamma$ depends linearly on $m_\Gamma^\refined$. Thus, the invariance for the counts of genus $g$ curves passing through $\P$ with multiplicity $M_\Gamma$ is implied by the invariance for counts of curves with a fixed gcd.
\end{rem}

We prove several relations between these invariants, seen as functions on the set of classes.

\begin{prop}\label{proposition relation invariants counts point case}
For any class we have
$$\begin{array}{>{\displaystyle}r>{\displaystyle}l>{\displaystyle}r>{\displaystyle}l}
N_{g,C,k} & =k^{4g-4}N_{g,C/k,1}, & BG_{g,C,k} & =BG_{g,C/k,1}(q^{k^2}). \\
\end{array}$$
Furthermore, we have the following identities:
$$\begin{array}{>{\displaystyle}r>{\displaystyle}l>{\displaystyle}r>{\displaystyle}l}
M_{g,C} & =\sum_{k|C}k^{4g-4}N_{g,C/k,1}, & BG_{g,C} & =\sum_{k|C} BG_{g,C/k,1}(q^{k^2}), \\
N_{g,C} & =\sum_{k|C}k^{4g-3}N_{g,C/k,1}, & R_{g,C} & =\sum_{lk|C}k^{2g-2}\varphi(k) BG_{g,C/kl,1}(q^{kl^2}). \\
\end{array}$$
\end{prop}

\begin{proof}
Given a tropical curve $h:\Gamma\to\TT A$ in the class $C$ with gcd $k$, we obtain a primitive tropical curve $\widehat{h}:\widehat{\Gamma}\to\TT A$ in the class $C/k$ which has the same image but the weight of every edge has been divided by $k$. The graph are the same except the lengths has been multiplied by $k$. At each of the $2g-2$ vertices, if $m_V$ and $\widehat{m}_V$ denote the multiplicity in $\Gamma$ and $\widehat{\Gamma}$, we have $m_V=k^2\widehat{m}_V$. Thus, we deduce that
$$m_\Gamma=k^{4g-4}m_{\widehat{\Gamma}} \text{ and } m_\Gamma^\refined(q)=m_{\widehat{\Gamma}}^\refined(q^{k^2}).$$
There is a bijection between the set of genus $g$ tropical curves having gcd $k|C$ in the class $C$ passing through $\P$, and the set of genus $g$ tropical curve of gcd $1$ in the class $C/k$: the images are the same, only the weights (and thus lengths) are changed. We deduce equalities of the first row.

The remaining identities are obtained summing over all the divisors of $C$. Only the last is obtained differently:
\begin{align*}
R_{g,C} & = \sum_{\delta| C}\sum_{\Gamma:\delta_\Gamma=\delta}\sum_{k|\delta}k^{2g-2}\varphi(k)m^\refined_\Gamma(q^{1/k}) \\
& = \sum_{\delta| C}\sum_{\Gamma:\delta_\Gamma=1}\sum_{k|\delta}k^{2g-2}\varphi(k)m^\refined_\Gamma(q^{\delta^2/k}) \\
& = \sum_{\delta| C}\sum_{k|\delta}k^{2g-2}\varphi(k)BG_{g,C/\delta,1}(q^{\delta^2/k}) \\
& =\sum_{lk|C}k^{2g-2}\varphi(k) BG_{g,C/kl,1}(q^{kl^2}). \\
\end{align*}
\end{proof}

The last identities may be interpreted as convolution between the following functions over the set of degrees:
$$\begin{array}{rccl}
N_g: & C & \longmapsto & N_{g,C} , \\
M_g: & C & \longmapsto & M_{g,C} , \\
N^\prim_g: & C & \longmapsto & N_{g,C,1} . \\
\end{array}$$
So that the end of Proposition \ref{proposition relation invariants counts point case} rewrites
$$\begin{array}{rcl}
M_g &=& \epsilon_{4g-4}\ast N_g^\prim ,\\
N_g &=& \epsilon_{4g-3}\ast N_g^\prim .\\
\end{array}$$

\subsubsection{Curves in a fixed linear system.} The second enumerative problem is to count genus $g$ curves in a class in a fixed linear system that pass through a generic configuration $\P$ of $g-2$ points inside $\TT A$. According to \cite{blomme2022abelian2}, tropical curves solution are simple and the complement of the marked points retracts on a subgraph of genus $2$ that we denote by $\Sigma$. We have the following counts that, according to \cite{blomme2022abelian2}, do not depend on the choice of the point configuration $\P$, the linear system, nor the abelian surface $\TT A$, as long as these choices are generic:

$$\begin{array}{lr>{\displaystyle}lr>{\displaystyle}l}
\triangleright \text{ count with fixed gcd: } & N^{FLS}_{g,C,k}= & \sum_{\substack{\Gamma\supset\P \\ \delta_\Gamma=k}}\Lambda_\Gamma^\Sigma m_\Gamma, & BG^{FLS}_{g,C,k}= & \sum_{\substack{\Gamma\supset\P \\ \delta_\Gamma=k}} \Lambda_\Gamma^\Sigma m_\Gamma^\refined, \\
\triangleright \text{ count with true multiplicity: } & N^{FLS}_{g,C}= & \sum_{\Gamma\supset\P} \delta_\Gamma \Lambda_\Gamma^\Sigma m_\Gamma, & R^{FLS}_{g,C}= & \sum_{\Gamma\supset\P} \Lambda_\Gamma^\Sigma M_\Gamma, \\
\triangleright \text{ count with auxiliary multiplicity: } & M^{FLS}_{g,C}= & \sum_{\Gamma\supset\P}\Lambda_\Gamma^\Sigma m_\Gamma, & BG^{FLS}_{g,C}= & \sum_{\Gamma\supset\P}\Lambda_\Gamma^\Sigma m_\Gamma^\refined. \\
\end{array}$$

As in the case of curves passing through $g$ points, we use the second refined multiplicity $\Lambda_\Gamma^\Sigma M_\Gamma$ as a refinement of the complex multiplicity $\delta_\Gamma\Lambda_\Gamma^\Sigma m_\Gamma$ rather than the naive refinement $\delta_\Gamma\Lambda_\Gamma^\Sigma m_\Gamma^\refined$. The invariance using the second refined multiplicity follows from its relation to the usual one.

\begin{prop}
For any $k|C$, we have
$$\begin{array}{>{\displaystyle}r>{\displaystyle}l>{\displaystyle}r>{\displaystyle}l}
N^{FLS}_{g,C,k} & =k^{4g-2} N_{g,C/k,1}, & BG^{FLS}_{g,C,k} & =k^{2} BG_{g,C/k,1}(q^{k^2}). \\
\end{array}$$
Furthermore, we have the following identities:
$$\begin{array}{>{\displaystyle}r>{\displaystyle}l>{\displaystyle}r>{\displaystyle}l}
M^{FLS}_{g,C} & =\sum_{k|C}k^{4g-2}N^{FLS}_{g,C/k,1}, & BG^{FLS}_{g,C} & =\sum_{k|C} k^2 BG^{FLS}_{g,C/k,1}(q^{k^2}), \\
N^{FLS}_{g,C} & =\sum_{k|C}k^{4g-1}N^{FLS}_{g,C/k,1}, & R^{FLS}_{g,C} & =\sum_{kl|C} k^{2g}\varphi(k)BG^{FLS}_{g,C/kl,1}(q^{kl^2}). \\
\end{array}$$
\end{prop}

\begin{proof}
The proof is similar to Proposition \ref{proposition relation invariants counts point case} with some small differences. A curve $\Gamma$ of genus $g$ in the class $C$ and gcd $k$ is obtained multiplying by $k$ a curve of gcd $1$ in the class $C/k$. However, if $\Gamma$ belongs to a fixed linear system $\L$, the curve $\Gamma/k$ belongs to one of the $k^2$ linear system $\L_0$ such that $\L_0^{\otimes k}=\L$. This is due to the torsion of the tropical Picard group $\mathrm{Pic}(\TT A)=\Lambda^*_\RR/M$. This accounts for the additional $k^2$ for identities in the first row. Remaining identities are obtained by adding the homogeneity relations of the first row for $k|C$ as in Proposition \ref{proposition relation invariants counts point case}.
\end{proof}

\subsection{Stretching in abelian surfaces}

\subsubsection{Complex setting for primitive classes.} In \cite{bryan1999generating}, the authors compute the complex invariants $\N_{g,C}$ for a primitive class $C$. As the answer does not depend on the chosen abelian surface, they choose an abelian surface that is a product of two generic elliptic curves $E_1\times E_2$, and then compute the invariants for the class $[E_1]+n[E_2]$, corresponding to the matrix $C=\left(\begin{smallmatrix}
1 & 0 \\
0 & n \\
\end{smallmatrix}\right)$ in our setting. The computation is made possible by the fact that stable maps $h:\CCC\to E_1\times E_2$ of genus $g$ that passing through $g$ points have a very special form: the projection on the first factor gives a degree $1$ map from $\CCC$ to $E_1$. Thus, $\CCC$ is reducible and is the union of a fiber of the projection to $E_2$ (that projects one to one onto $E_1$), and several covers of the fibers of the projection to $E_1$. The degree $n$ has to be split among the various components. We get
\begin{align*}
\N_{g,\left(\begin{smallmatrix}
1 & 0 \\
0 & n \\
\end{smallmatrix}\right)} & =g\sum_{a_1+\cdots+a_{g-1}=n}\prod_{i=1}^{g-1}a_i\sigma_1(a_i), \\
\N^{FLS}_{g,\left(\begin{smallmatrix}
1 & 0 \\
0 & n \\
\end{smallmatrix}\right)} & =\sum_{a_1+\cdots+a_{g-2}+a_\infty=n} a_\infty^2\sigma_1(a_\infty)\prod_{i=1}^{g-2}a_i\sigma_1(a_i), \\
\end{align*}
where $\sigma_1(a)=\sum_{d|a}d$. However, the method does not apply for the classes $d_1[E_1]+d_2[E_2]$ since solutions are not necessarily a union of elliptic curves.

\subsubsection{Tropical rectangular abelian surfaces.} Tropically, looking at a product of tropical elliptic curves means that we look at the ``rectangular" abelian surface
$$\TT A_{L,l}=\RR^2/\left\langle Le_1,le_2\right\rangle = \RR/L\ZZ\times\RR/l\ZZ.$$
In other words, the map $S:\Lambda\to \RR^2\simeq N_\RR$ is diagonal in the canonical basis of $\RR^2$: $\left(\begin{smallmatrix}
L & 0 \\
0 & l \\
\end{smallmatrix}\right)$. The computation of \cite{bryan1999generating} does not translate well to the tropical setting since the curves in the class $\left(\begin{smallmatrix}
1 & 0 \\
0 & n \\
\end{smallmatrix}\right)$ are of the form depicted on Figure \ref{figure product of elliptic curve} $(a)$: one has a horizontal component and several vertical components with a splitting of $n$ among the different fibers. In particular, the curves are not simple and lack a well-defined multiplicity.

\begin{figure}[h]
\begin{center}
\begin{tabular}{c}
\begin{tikzpicture}[line cap=round,line join=round,>=triangle 45,x=0.5cm,y=0.5cm]
\clip(-0.5,-0.5) rectangle (21,5);

\draw [line width=1pt] (0,0)-- ++(20,0) -- ++(0,4)-- ++(-20,0) --++(0,-4);

\draw [line width=1.5pt] (0,1.5)-- ++(20,0);

\draw [line width=1.5pt] (4,0)--++ (0,4);
\draw [line width=2pt] (9,0)--++ (0,4);

\draw [line width=2pt] (16,0)--++ (0,4);

\draw (16,3) node[right] {$2$};
\draw (9,3) node[right] {$3$};

\end{tikzpicture} \\
(a) \\
\begin{tikzpicture}[line cap=round,line join=round,>=triangle 45,x=0.5cm,y=0.5cm]
\clip(-0.5,-0.5) rectangle (21,5);

\draw [line width=1pt] (0,0) node{$\bullet$} -- ++(19.8,0) node{$\bullet$}  --++ (0.2,1.2) node{$\bullet$} -- ++(0,2.8) node{$\bullet$} -- ++(-19.8,0) node{$\bullet$}  --++ (-0.2,-1.2) node{$\bullet$} -- ++(0,-2.8);

\draw [line width=1.5pt] (0,1.5)-- (4,1.5)-- ++(0.2,0.2)-- (9,1.7)-- ++(0.6,0.6)-- (16,2.3)-- ++(0.2,0.4)-- (20,2.7);

\draw [line width=1.5pt] (4,0)-- (4,1.5) ++(0.2,0.2)-- (4.2,4);
\draw [line width=1.5pt] (9,0)-- (9,1.7) ++(0.6,0.6)-- (9.6,4);
\draw [line width=1.5pt] (9.2,0)-- ++(0,4);
\draw [line width=1.5pt] (9.4,0)-- ++(0,4);
\draw [line width=2pt] (16,0)-- (16,2.3) ++(0.2,0.4)-- (16.2,4);

\draw (16.2,3.5) node[right] {$2$};

\end{tikzpicture} \\
(b) \\
\end{tabular}
\caption{\label{figure product of elliptic curve}A tropical curve realizing the class $(1,6)$ admitting a pearl decomposition in a ``rectangular" abelian surface (a) and inside a small deformation of it (b).}
\end{center}
\end{figure}
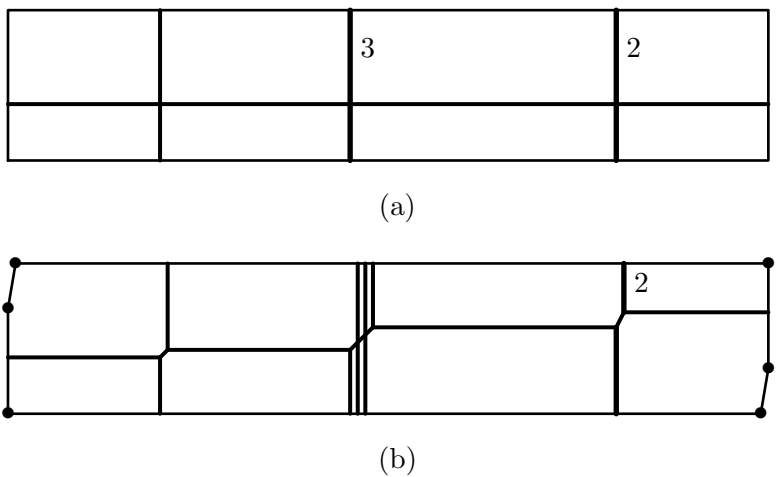

\subsubsection{Stretched tropical abelian surfaces.} The problem of ``rectangular" abelian surfaces is that they are not generic, and thus not suited for the tropical computations. Already in the above primitive case, it is not possible to compute the refined invariant since the curves are not trivalent. One way to handle this problem is to slightly deform the tropical abelian surface, so that the considered class stays realizable. Take a basis of $\Lambda$ and $N$ and choose $\Lambda\subset N_\RR$ to be the image of $\left(\begin{smallmatrix}
L & d_1\varepsilon \\
d_2\varepsilon & l \\
\end{smallmatrix}\right)$, with $L,l,\varepsilon\in\RR$. It contains curves in the class $\left(\begin{smallmatrix}
d_1 & 0 \\
0 & d_2 \\
\end{smallmatrix}\right)$ A fundamental domain in $\RR^2$ of the torus $\TT A$ can be taken to be an hexagon with pairs of opposite sides glued together, as depicted on Figure \ref{figure product of elliptic curve} $(b)$. We see that the curve inside is now trivalent.

\medskip

Consider instead the class $\left(\begin{smallmatrix}
d_1 & a \\
0 & d_2 \\
\end{smallmatrix}\right)$ for some non-zero top-right coefficient. We later care specifically for the cases $a=1$. The matrices $S$ satisfying $CS^T$ being symmetric are of the form
$$\begin{pmatrix}
L & l a/d_2 \\
0 & l \\
\end{pmatrix} +\varepsilon \begin{pmatrix}
0 & d_1 \\
d_2 & 0 \\
\end{pmatrix},\text{ with }\varepsilon,l,L\in\RR.$$

If $\varepsilon=0$, the corresponding abelian surface can be seen as the gluing of a rectangle as depicted on Figure \ref{figure product of elliptic curve deformed}(a). Changing $\varepsilon$ means slightly deforming the rectangle into an hexagon whose opposite edges are glued together, as can be seen on Figure \ref{figure product of elliptic curve deformed}(b). We now come to the stretching.

\begin{figure}
\begin{center}
\begin{tabular}{cc}
\begin{tikzpicture}[line cap=round,line join=round,>=triangle 45,x=0.5cm,y=0.5cm]
\clip(-0.5,-0.5) rectangle (10.5,4.5);

\draw [line width=1pt] (0,0) node{$\bullet$} -- node{$>$} ++(8,0) node{$\bullet$} -- node{$\gg$} ++(2,0) node{$\bullet$} -- node[sloped]{$>|$} ++(0,4) node{$\bullet$};
\draw [line width=1pt] (0,0) -- node[sloped]{$>|$} ++(0,4) node{$\bullet$} -- node{$\gg$} ++(2,0) node{$\bullet$} -- node{$>$} ++(8,0) ;
\end{tikzpicture} &
\begin{tikzpicture}[line cap=round,line join=round,>=triangle 45,x=0.5cm,y=0.5cm]
\clip(-0.5,-0.5) rectangle (10.5,4.8);

\draw [line width=1pt] (0,0) node{$\bullet$} -- node{$>$} ++(8,0) node{$\bullet$} -- node[sloped]{$\gg$} ++(2,0.4) node{$\bullet$} -- node[sloped]{$>|$} ++(0,4) node{$\bullet$};
\draw [line width=1pt] (0,0) -- node[sloped]{$>|$} ++(0,4) node{$\bullet$} -- node[sloped]{$\gg$} ++(2,0.4) node{$\bullet$} -- node{$>$} ++(8,0) ;
\end{tikzpicture} \\
(a) & (b) \\
\end{tabular}
\caption[product of elliptic curve deformed]{\label{figure product of elliptic curve deformed}A tropical torus associated to the matrix $\left(\begin{smallmatrix}
L & la/d_2 \\
0 & l \\
\end{smallmatrix}\right)$ and a small deformation.}
\end{center}
\end{figure}
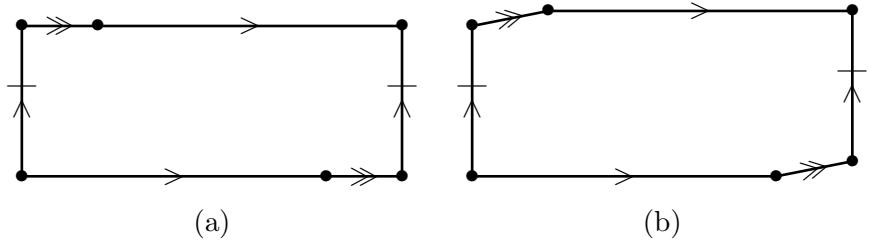

\begin{defi}
Let $C=\left(\begin{smallmatrix}
d_1 & a \\
0 & d_2 \\
\end{smallmatrix}\right)$ be a fixed class.
\begin{itemize}
\item A tropical abelian surface with a given curve class $C$ is \textit{stretched} if $S:\Lambda\to N_\RR$ is of the form $\left(\begin{smallmatrix}
L & la/d_2 + d_1\varepsilon \\
d_2\varepsilon & l \\
\end{smallmatrix}\right),$ with $L\gg l\gg\varepsilon$.
\item A point configuration $\P$ of $g$ points inside a stretched tropical abelian surface is \textit{stretched} if the difference between the horizontal coordinates in the fundamental domain is far bigger than $l$: let $(x_i,y_i)$ be the coordinates inside the hexagon, it means
$$L,|x_i-x_j| \gg l\gg\varepsilon.$$
\item A point configuration $\P$ of $g-2$ points inside a stretched tropical abelian surface is \textit{doubly stretched} if the difference between the horizontal coordinates in the fundamental domain is far bigger than $l$, and far smaller than $L$:
$$L\gg |x_i-x_j| \gg l\gg\varepsilon.$$
\end{itemize}
\end{defi}


\subsubsection{Bounded lengths in stretched abelian surfaces.} We finish this section with a lemma concerning the lengths of the edges of tropical curves in a stretched abelian surface. It plays a role similar to Lemma \cite[Lemma 5.8]{blomme2021floor}. Before going to the Lemma, notice that in the above stretched tropical abelian surfaces, the boundary of the hexagonal fundamental domain is in fact the image of a genus $2$ tropical curve: the slopes of the edges are respectively $\left(\begin{smallmatrix}
d_1 \\
0 \\
\end{smallmatrix}\right),\left(\begin{smallmatrix}
0 \\
d_2 \\
\end{smallmatrix}\right)$ and $\left(\begin{smallmatrix}
d_1 \\
d_2 \\
\end{smallmatrix}\right)$. We call this curve $\Gamma_0$.

\begin{lem}\label{lem-finite-length}
Let $h:\Gamma\to\TT A$ be a parametrized tropical curve realizing the class $C$ in the stretched abelian surface $\TT A$. The length of a non-horizontal edge is bounded by a constant that only depends on $C,l$, not on $L$ nor the curve.
\end{lem}

\begin{proof}
Consider the genus $2$ curve $\Gamma_0$ in the class $C$. It has an horizontal edge $e_0$. The preimage of $h_0(\Gamma_0)$ by the covering map $N_\RR\to\TT A$ is a honeycomb lattice. By Lemma \ref{lem-finite-slopes}, we know that the slopes may take a finite number of values. Let $h:\Gamma\to\TT A$ be a curve in the class $C$ and let $e$ be an edge of $\Gamma$ with slope $\binom{a}{b}$, $b\neq 0$, and length $\ell_e$. Consider a lift $\widetilde{e}$ of $e$ by the covering map $N_\RR\to\TT A$. By stretching, we have that $L\gg l$, and hence $\frac{L}{l}$ is big enough in front of the finitely possible slopes that can take the edges. Therefore, we have at least $\frac{|b|\ell_e}{l}$ intersection points between $\widetilde{e}$ and hrizontal edges in the honeycomb lattice, as can be seen on Figure \ref{fig-honeycomb}, and each one yields an intersection point between $e$ and $h_0(\Gamma_0)$. As $|b|\geqslant 1$ and the intersection index is at least one, we have $\ell_e\leqslant l C^2$.
\end{proof}

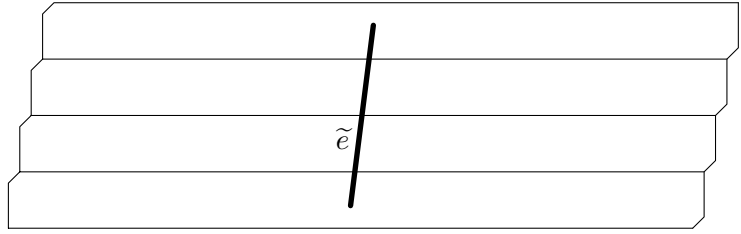
\begin{figure}
\centering
\begin{tikzpicture}[line cap=round,line join=round,>=triangle 45,x=0.3cm,y=0.3cm]
\draw (0,0)--++(30,0)++(-30,0)--++(0,2)--++(0.5,0.5)--++(30,0)++(-30,0)--++(0,2)--++(0.5,0.5)--++(30,0)++(-30,0)--++(0,2)--++(0.5,0.5)--++(30,0)++(-30,0)--++(0,2)--++(0.5,0.5)--++(30,0)++(-30,0);
\draw (30,0)--++(0.5,0.5)--++(0,2)--++(0.5,0.5)--++(0,2)--++(0.5,0.5)--++(0,2)--++(0.5,0.5)--++(0,2);

\draw[line width=2pt] (15,1)--++(1,8) node[midway,below left] {$\widetilde{e}$};
\end{tikzpicture}
\caption{\label{fig-honeycomb}Intersection between a genus $2$ curve and the lift of an edge of some tropical curve in Lemma \ref{lem-finite-length}.}
\end{figure}

\section{Pearl diagrams for curves passing through $g$ points}

This section is comprised of four steps. The first section is how to get a pearl diagram from a tropical curve in a stretched abelian surface. the second part is especially technical and consists in recovering the set of tropical curves encoded in a fixed pearl diagram. With this data, in the next section, we derive different multiplicities for the pearl diagrams, so that their count using  these multiplicities give the desired invariants. Last, we prove the multiple cover formulas for the punctual case.

\subsection{From curves to pearl diagrams}

Consider a stretched tropical abelian surafce $\TT A$. Its first homology group $\Lambda$ is endowed with a natural basis $(\lambda_1,\lambda_2)$, represented by paths going in the horizontal and vertical direction. Let $h:\Gamma\to\TT A$ be a genus $g$ parametrized tropical curve in the class $C$ passing through a stretched configuration of $g$ points $\P=\{p_1,\cdots,p_g\}$ in $\TT A$. Let $\gamma\subset\TT A$ be a cycle in the class $\lambda_2$ that passes between the marked points $p_g$ and $p_1$. We now describe how to get a diagram from $(\Gamma,h)$. 
\begin{itemize}
\item The edges of $\Gamma$ with horizontal slope are called \textit{horizontal edges}. A component of the complement of the horizontal edges is called a \textit{pearl}.
\item The \textit{diagram} $\Dfk$ of $\Gamma$ is the quotient graph of $\Gamma$ where each pearl has been contracted to a single vertex. It has two kinds of vertices: those coming from contraction of components (also called pearls), and marked points lying on horizontal edges.
\item Two vertices of $\Dfk$ are linked if the corresponding parts of $\Gamma$ are adjacent to the same horizontal edge.
\item The weight $w_e$ of an edge in $\Dfk$ is its weight as an edge of $\Gamma$, \textit{i.e.} the lattice length of the slope of $h$ on it. Edges are oriented such that their slope has a positive coordinate in $\ZZ^2$.
\item The length $\ell_e$ of an edge $e$ is the number of times its image intersects the cycle $\gamma\subset\TT A$.
\item The degree $d_\pfk$ of a pearl is obtained as follows: take a path $\alpha$ in the class $\lambda_1$ transverse to $\pfk$, orient the edges of $\pfk$ so that they intersect positively and add their slopes. The result depend on the choice of $\alpha$ but its vertical coordinate does not by balancing condition. This vertical coordinate is $d_\pfk$.
\item The marked points $\{p_1,\dots,p_g\}$ induce a labeling of vertices of $\Dfk$ by $\{1,\dots,g\}$, although a priori, a pearl may receive several labels or none.
\end{itemize}

\begin{expl}
On Figure \ref{figure example curve to pearl} we can see a tropical curve admitting a pearl decomposition and the corresponding pearl diagram. The degree of the pearls labeled 1,3,4 are respectively $d_1=1$, $d_3=3$ and $d_4=2$.
\end{expl}

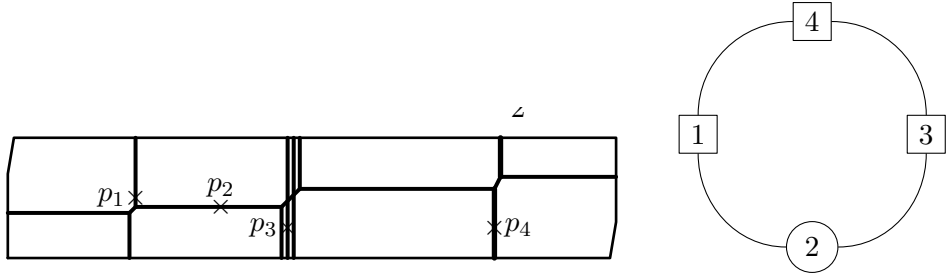
\begin{figure}
\begin{center}
\begin{tabular}{cc}
\begin{tikzpicture}[line cap=round,line join=round,>=triangle 45,x=0.4cm,y=0.4cm]
\clip(-0.5,-0.5) rectangle (21,5);

\draw [line width=1pt] (0,0)-- ++(19.8,0) --++ (0.2,1.2)-- ++(0,2.8)-- ++(-19.8,0) --++ (-0.2,-1.2)-- ++(0,-2.8);

\draw [line width=1.5pt] (0,1.5)-- (4,1.5)-- ++(0.2,0.2)-- (9,1.7)-- ++(0.6,0.6)-- (16,2.3)-- ++(0.2,0.4)-- (20,2.7);

\draw [line width=1.5pt] (4,0)-- (4,1.5) ++(0.2,0.2)-- (4.2,4);
\draw [line width=1.5pt] (9,0)-- (9,1.7) ++(0.6,0.6)-- (9.6,4);
\draw [line width=1.5pt] (9.2,0)-- ++(0,4);
\draw [line width=1.5pt] (9.4,0)-- ++(0,4);
\draw [line width=2pt] (16,0)-- (16,2.3) ++(0.2,0.4)-- (16.2,4);

\draw (16.2,5) node[right] {$2$};
\draw (4.2,2) node {$\times$} node[left] {$p_1$};
\draw (7,1.7) node {$\times$} node[above] {$p_2$};
\draw (9.2,1) node {$\times$} node[left] {$p_3$};
\draw (16,1) node {$\times$} node[right] {$p_4$};

\end{tikzpicture}
&
\begin{tikzpicture}[line cap=round,line join=round,x=0.75cm,y=0.75cm]
	\pearl (1) at (180:2) label=1;
	\flatpearl (2) at (-90:2) label=2;
	\pearl (3) at (0:2) label=3;
	\pearl (4) at (90:2) label=4;
	
	\draw (1) to[out=-90,in=180] (2);
	\draw (2) to[out=0,in=-90] (3);
	\draw (3) to[out=90,in=0] (4);
	\draw (1) to[out=90,in=180] (4);

	\end{tikzpicture}
\\
\end{tabular}
\caption{\label{figure example curve to pearl}Example of a tropical curve and the associated pearl diagram.}
\end{center}
\end{figure}

We now prove that the diagram obtained from a tropical curve passing through $\P$ is indeed a pearl diagram of the right genus and degree.

\begin{prop}\label{proposition p marked diagram}
Let $h:\Gamma\to\TT A$ be a genus $g$ parametrized tropical curve in the class $\left(\begin{smallmatrix} d_1 & a \\
0 & d_2\\
\end{smallmatrix}\right)$ passing through the stretched configuration $\P$ of $g$ points. Let $\Dfk$ be the diagram obtained from $h:\Gamma\to\TT A$. Then $\Dfk$ is a pearl diagram of genus $g$ and bidegree $(d_1,d_2)$. 
Furthermore, each pearl $\pfk\subset\Gamma$ is homeomorphic to a circle.
\end{prop}

\begin{proof}
The proof is similar to the proof of \cite[Proposition 5.10]{blomme2021floor}. We need to show that $\Dfk$ satisfies the constraints from Definition \ref{def-pearl-diagram}.
\begin{enumerate}[label=(\alph*)]
\item Vertices are already split into pearls and flat vertices, which are bivalent. We need to show that each pearl has a unique label.
	\begin{itemize}
	\item Assume a pearl has more than two labels. This means there is a pearl in $\Gamma$ containing two marked points, and thus we have a path between $p_i\neq p_j$ consisting only of non-horizontal edges. By Lemmas \ref{lem-finite-slopes} and \ref{lem-finite-length}, there is only a finite number of possible slopes and the lengths are bounded. As the point configuration $\P$ is sretched, it is not possible to find such a path. Thus, there is at most one label per pearl.
	\item Assume a pearl $\pfk$ has no label. Then, it is possible to deform $\Gamma$ and get a $1$-parameter family of curves passing through $\P$ by changing the lengths of the horizontal edges adjacent to $\pfk$, which amounts to slightly translate the image of the edges inside the pearl. This contradicts the genericity of $\P$ for which we have a finite number of solutions. Thus, a pearl has at least one label.
	\end{itemize}
	
\item We need to show the balancing condition, which is automatically satisfied at the flat vertices. For a pearl $\pfk$, at every vertex $V\in\pfk\subset \Gamma$, we have the balancing condition in $N$. Adding the balancing conditions over all vertices yields the balancing condition for the horizontal edges adjacent to $\pfk$.

\item The orientation on the edges ensures that an edge $e$ going from a vertex with label $i$ to a vertex with label $j$ satisfies $\ell_e\geqslant 0$, and strict inequality if $i>j$.

\item To conclude, we need to show that the complement of flat vertices is connected and without cycle. To prove it, we use the fact that the complement of $\P$ in $\Gamma$ is connected without cycle. Consider a pearl $\pfk\subset\Gamma$.
	\begin{itemize}
	\item As a subgraph of $\Gamma$, $\pfk$ has genus at most $1$: it contains a unique point of $\P$ by (a), and the complement is without cycle since the complement of $\P$ in $\Gamma$ is without cycle.
	\item Furthermore, the subgraph $\pfk$ is homeomorphic to a circle: no edge is disconnecting. To see that, assume there is a disconnecting edge $e\subset\pfk$. Adding the balancing conditions for vertices on one side of $e$, the slope of $e$ needs to be horizontal, which is impossible. Thus, no edge is disconnecting. As the genus of $\pfk$ is at most $1$, it is a circle.
	\end{itemize}
	This proves the last part of the proposition. As the complement of $\P$ in $\Gamma$ is connected, so is the complement of flat vertices in $\Dfk$. If there was a cycle in the complement of flat vertices, there would be one in the complement of $\P$ in $\Gamma$ as well. As the genus of each pearl is $1$, $\Dfk$ has also genus $g$.
\end{enumerate}

To conclude, we need to show that the bidegree of $\Dfk$ is $(d_1,d_2)$.
\begin{itemize}
\item The sum of slopes of edges intersecting $\gamma$ is equal to $\left(\sum_e w_e\ell_e\right)e_1$, but is also by definition equal to $\left(\begin{smallmatrix} d_1 \\ 0 \\ \end{smallmatrix}\right)$.
\item By definition of $d_\pfk$, $\sum_\pfk d_\pfk$ is the vertical coordinate of the sum of slopes of edges intersecting a path in the class $\lambda_1$. As the latter is equal to $\left(\begin{smallmatrix} a \\ d_2 \\ \end{smallmatrix}\right)$, we get $\sum_\pfk d_\pfk=d_2$. 
\end{itemize}
\end{proof}

Notice that when getting a pearl diagram from a tropical curve in the class $\left(\begin{smallmatrix}
d_1 & a \\
0 & d_2 \\
\end{smallmatrix}\right)$, $a$ disappears. In particular, even if a curve is primitive, its pearl diagram may not be primitive.

\subsection{From pearl diagrams back to tropical curves}

We now explain how to recover the tropical curves in a class $\left(\begin{smallmatrix}
d_1 & a \\
0 & d_2 \\
\end{smallmatrix}\right)$ encoded by a pearl diagram $\Dfk$. This is achieved by the main (and most technical) following result.

\begin{prop}\label{proposition curves in diagram point case}
Let $\Dfk$ be a fixed pearl diagram. Let $(k_\pfk)$ be a family such that $k_\pfk|d_\pfk$ for each pearl $\pfk$. Each such choice is called \textit{loop data}. Let
$$\mathrm{gcd}(w_e,k_\pfk)=\mathrm{gcd}(\{w_e\}_{e}\cup\{k_\pfk\}_\pfk).$$
We set $k_v=1$ if $v\in V_\circ(\Dfk)$ is a flat vertex. If $\gcd(w_e,k_\pfk)$ divides $a$, the diagram gives a number of genus $g$ tropical curves in the class $\left(\begin{smallmatrix}
d_1 & a \\
0 & d_2 \\
\end{smallmatrix}\right)$ and passing through the stretched configuration $\P$ equal to
$$\mathrm{gcd}(w_e,k_\pfk)\prod_\pfk k_\pfk^{\val\pfk-1}\prod_{E_\square(\Dfk)}w_e,$$
where $E_\square(\Dfk)$ is the set of edges between two pearls. Moreover, the usual and refined multiplicities $m_\Gamma$ and $m_\Gamma^\refined$ of these curves are all the same and their respective values are
$$m_\Gamma=\prod_\pfk \left(\frac{d_\pfk}{k_\pfk}\right)^{\val\pfk}\prod_{e} w_e^{|\{\pfk\in e\}|}
\text{ and }
m_\Gamma^\refined=\prod_\pfk\prod_{e\ni\pfk}\left[ \frac{d_\pfk}{k_\pfk}w_e \right]_-,$$
where we have set $[\alpha]_-=q^{\alpha/2}-q^{-\alpha/2}$. If $\mathrm{gcd}(w_e,k_\pfk)$ does not divide $a$, the number of curves is $0$.
\end{prop}

The loop data $k_\pfk$ corresponds to the number of turns that the pearl $\pfk$ makes around the vertical direction. It divides $\delta_\Dfk=\mathrm{gcd}(w_e,d_\pfk)$ but may not be equal to it, for instance if some $k_\pfk=1$.

The pearl diagram approach does not allow to easily recover the gcd of the tropical curves encoded in the diagram. Some of the curves might be primitive, but maybe not all of them. This a priori justifies the use of the usual refined multiplicity $m_\Gamma^\refined$ since we do not have a grasp on the gcd of the curves. The count of solutions using the complex multiplicity $\delta_\Gamma m_\Gamma$ or the second refined multiplicity $M_\Gamma$ can be recovered by making a disjunction over the value of gcd. It is the content of the next section.

\begin{proof}[Idea of proof]
To recover tropical curves from a pearl diagram $\Dfk$, one proceeds as follows.
\begin{itemize}[label=$\circ$]
\item First, following Proposition \ref{proposition p marked diagram}, the subgraph associated to a pearl is a circle mapped around the vertical direction of $\TT A$, realizing some class $k_\pfk\lambda_2$, where $k_\pfk|d_\pfk$. It can thus be lifted under a degree $k_\pfk$ cover that makes only one turn around the vertical direction.
\item The complement of the flat vertices inside $\Dfk$ is a tree. The position of each flat vertex is fixed by a point constraint. We then prune this tree to reconstruct the curve inductively. As a neighbourhood of a pearl $\pfk$ in the abelian surface is exactly an open set in a cylinder, this step is analog to what happens in \cite{blomme2021floor}, where we can use the \textit{Menelaus relation} \cite[Theorem 3.18]{blomme2021floor}. The latter is a relation between the positions of the ends of a tropical curve on a cylinder, which here become the horizontal edges adjacent to $\pfk$.
\item We now explain the pruning: for a pearl $\pfk$ all of whose adjacent edges but one are fixed, the pearl is recovered as follows. One chooses lifts of the edges under the covering map of degree $k_\pfk$ that makes the pearl do only one turn around the vertical direction, as in \cite{blomme2021floor}. Using the \textit{Menelaus relation} between the edges, this determines the position of the remaining edge $e$ up to $w_e$-torsion. Accounting for deck transformations and lift of the marked point on the pearl, there are $k_\pfk^{\val\pfk-1}$ lifting possibilities. The $\prod_{E_\square(\Dfk)}w_e$ is for these torsion choices.
\item Depending on how we prune the branches, we end up at one of the following possibilities:
	\begin{itemize}[label=$\blacktriangleright$]
		\item Either we have an edge both whose adjacent pearls have been pruned. The position of the edge has thus been determined in two different ways that have to agree. The fact that $C$ is realizable ensures that the positions agree, but only up to torsion. Hence, it might not be possible to glue.
		\item Or, we have to find a pearl all of whose adjacent edges are fixed. The fact that the class $C$ is realizable ensures that it is possible to find a pearl completing the curve, ensuring that the Menelaus condition is satisfied mod $l$. However, when choosing lifts under the degree $k_\pfk$-cover, the Menelaus relation has to be satisfied mod $k_\pfk l$, so the lifts might not be chosen independently, and it might not be possible to choose them at all.
	\end{itemize}
\end{itemize}

The last steps of the algorithm suggest that all the lifts should be chosen simultaneously. Such an information is contained in a map between real tori that we consider at the beginning of the proof.
\end{proof}

\begin{proof}
\textbf{Step A: formulation of the problem.} For an oriented edge $e$, let $e_+$ and $e_-$ be the vertices it contains, and for a flat vertex $v\in V_\circ(\Dfk)$, let $e_v^\pm$ be the edges adjacent to it. Recovering the tropical curves is done using the following map:
$$\Phi_{\RR/\ZZ}:\left|\begin{array}{>{\displaystyle}cc>{\displaystyle}l}
 \bigoplus_{E(\Dfk)} \RR/k_{e_-}l\ZZ\oplus\RR/k_{e_+}l\ZZ & \longrightarrow & \bigoplus_{E(\Dfk)} \RR/l\ZZ \oplus \bigoplus_{V_\circ(\Dfk)} (\RR/l\ZZ)^2 \oplus \bigoplus_{V_\square(\Dfk)} \RR/ k_\pfk l\ZZ \\
 (x_e^-,x_e^+)_e & \longmapsto & \left( (x_e^+-x_e^-), (x_{e_v^-}^-,x_{e_v^+}^+)_v , \left( \sum_{e\ni\pfk} \pm w_ex_e^\pm \right)_\pfk  \right) \\
\end{array} \right. .$$
The domain has rank equal to the number of \textit{flags}, which are pairs of $(v,e)\in V(\Dfk)\times E(\Dfk)$ with $v\in e$. It measures the vertical coordinate of the horizontal edge of the flag in the degree $k_\pfk$ cover of $\RR/l\ZZ$. The codomain is comprised of three terms:
\begin{itemize}
\item The first term is indexed by edges and is the difference between vertical position of its flags in $\RR/l\ZZ$. For the gluing to be possible, we need:
$$x_e^+-x_e^-=\ell_e\cdot d_2\varepsilon.$$
If $\ell_e=0$, this amounts to say that the vertical coordinate agree. If $\ell_e>0$, there is a monodromy of $d_2\varepsilon$ in the vertical coordinate when going around the loop $\lambda_1$, yielding the $\ell_ed_2\varepsilon$.
\item The second term is the vertical position of the flags adjacent to a flat vertex. Both are fixed by the point constraints, so that we need to impose $x_{e_v^-}^-=x_{e_v^+}^+=y_v$.
\item The last term is the Menelaus condition at the pearl $\pfk$: there is a unique local tropical curve doing $k_\pfk$ turns around the direction $\lambda_2$ with the chosen vertical positions of edges in $\RR/k_\pfk l\ZZ$ if and only if we have
$$\sum_{e\ni\pfk} \pm w_e x_e^\pm \equiv d_\pfk\left(\frac{a}{d_2}l+d_1\varepsilon\right) \text{ modulo } k_\pfk l .$$
\end{itemize}

The coordinate corresponding to a flag $v\in e$ is involved for the summands of the codomain corresponding to $v$ and $e$, where the terms of $V_\circ(\Dfk)$ have been doubled for each of its flags.

The number of tropical curves encoded in the diagram is equal to the number of preimages of a point $\left((\ell_e d_2\varepsilon)_e,(y_v,y_v)_v,(d_\pfk \left(\frac{a}{d_2}l+d_1\varepsilon\right))_\pfk\right)$, where $(y_v)_v$ are chosen generically. To do this, we have to determine if this point lies in the image of $\Phi_{\RR/\ZZ}$, and what is the number of preimages of a point in the image.

\textbf{Step B: Image of the map $\Phi_{\RR/\ZZ}$ between tori.} The map $\Phi_{\RR/\ZZ}$ can be written $\Phi\otimes\RR/\ZZ$ where $\Phi$ is the following lattice map
$$\Phi:\left| \begin{array}{>{\displaystyle}cc>{\displaystyle}l}
(\ZZ\oplus\ZZ)^{E(\Dfk)} & \longrightarrow & \ZZ^{E(\Dfk)} \oplus \ZZ^{2V_\circ(\Dfk)} \oplus \ZZ^{V_\square(\Dfk)} \\
 (x_e^-,x_e^+)_e & \longmapsto & \left( (k_{e_+}x_e^+ - k_{e_-}x_e^-)_e, (x^-_{e_v^-},x^+_{e_v^+})_v , \left( \sum_{e\ni\pfk} \pm w_ex_e^\pm \right)_\pfk  \right) \\
\end{array} \right. .$$
In fact, $\Phi_{\RR/\ZZ}$ is a map between two real tori, and $\Phi=H_1(\Phi_{\RR/\ZZ})$ is the map between their first homology groups. The domain has rank $2|E(\Dfk)|$ while the codomain has rank $|E(\Dfk)|+2|V_\circ(\Dfk)|+|V_\square(\Dfk)| $. The difference between ranks of codomain and domain is
$$(|V(\Dfk)|-|E(\Dfk)|)+|V_\circ(\Dfk)|=1.$$
It means that this map cannot be surjective. There is indeed a relation between the first and third coordinates of the image. Let $\left( (y_e)_e,(y_v,y'_v)_v,(\mu_\pfk)_\pfk\right)$ belong to the codomain of $\Phi_{\RR/\ZZ}$, and set
$$\psi_{\RR/\ZZ}\left( (y_e)_e,(y_v,y'_v)_v,(\mu_\pfk)_\pfk\right)= \sum_\pfk \mu_\pfk -\sum_e w_ey_e \in\RR/l\ZZ ,$$
where $\mu_\pfk\in\RR/k_\pfk l\ZZ$ is reduced in $\RR/l\ZZ$ before making the sum. We have that $\mathrm{Im}\ \Phi_{\RR/\ZZ}\subset\ker\psi_{\RR/\ZZ}$. Geometrically, it means the following: interpreting $y_e$ as flows on the graph, the sum of the flows on the edges is equal to the sum of the divergence at the pearls.

The torus map $\psi_{\RR/\ZZ}$ may be expressed as $\psi\otimes\RR/\ZZ$ where $\psi=H_1(\psi_{\RR/\ZZ})$ is the following lattice map:
$$\psi\left( (y_e)_e,(y_v,y'_v)_v,(\mu_\pfk)_\pfk\right)=\sum_\pfk k_\pfk\mu_\pfk - \sum_e w_e y_e.$$
This way, we see that $\psi$ may not be primitive since $w_e$ and $k_\pfk$ may not be coprime. If $\delta=\gcd(w_e,k_\pfk)$, $\ker\psi_{\RR/\ZZ}$ has in fact $\delta$ components. As $\mathrm{Im}\ \Phi_{\RR/\ZZ}$ is connected, we rather consider $\widehat{\psi}=\frac{1}{\delta}\psi$, and $\widehat{\psi}_{\RR/\ZZ}=\widehat{\psi}\otimes\RR/\ZZ$, whose kernel is now connected:
$$\widehat{\psi}_{\RR/\ZZ}\left( (y_e)_e,(y_v,y'_v)_v,(\mu_\pfk)_\pfk\right)=\sum_\pfk \frac{1}{\delta}\mu_\pfk-\sum_e \frac{w_e}{\delta}y_e,$$
where $\frac{1}{\delta}\mu_\pfk$ is well-defined in $\RR/(k_\pfk/\delta)l\ZZ$. We still have $\mathrm{Im}\ \Phi_{\RR/\ZZ}\subset\ker\widehat{\psi}_{\RR/\ZZ}$. Thus, assuming $\Phi$ is injective (which follows from the computation of index in Step C), both have the same dimension and we have equality. To determine if $\left((\ell_e d_2\varepsilon)_e,(y_v,y'_v)_v,\left(d_\pfk \left(\frac{a}{d_2}l+d_1\varepsilon\right)\right)_\pfk\right)$ belongs to $\mathrm{Im}\ \Phi_{\RR/\ZZ}$, we just have to check that its image under $\widehat{\psi}_{\RR/\ZZ}$ is $0$:
\begin{align*}
\widehat{\psi}_{\RR/\ZZ}\left((\ell_e d_2\varepsilon)_e,(y_v,y'_v)_v,\left(d_\pfk \left(\frac{a}{d_2}l+d_1\varepsilon\right)\right)_\pfk\right) \equiv & \sum_\pfk \frac{1}{\delta}d_\pfk\left(\frac{a}{d_2}l+d_1\varepsilon\right) - \sum_e \frac{w_e}{\delta}d_e d_2 \varepsilon  \text{ mod }l \  \\
 \equiv & \frac{1}{\delta}d_2\left(\frac{a}{d_2}l+d_1\varepsilon\right) -\frac{d_1d_2\varepsilon}{\delta}  \text{ mod }l\\
 \equiv & \frac{a}{\delta}l \text{ mod }l.\\
\end{align*}
Thus, there are curves in the class $\left(\begin{smallmatrix}
d_1 & a \\
0 & d_2 \\
\end{smallmatrix}\right)$ if and only if $\mathrm{gcd}(w_e,k_\pfk)|a$.

\textbf{Step C: degree of $\Phi_{\RR/\ZZ}$.} Let $n=2|E(\Dfk)|$ so that $n+1=|E(\Dfk)|+2|V_\circ(\Dfk)|+|V_\square(\Dfk)|$. The map $\Phi_{\RR/\ZZ}$ is a group morphism between real tori coming from the (soon to be injective) lattice map
$$\Phi:\ZZ^n\to \ZZ^{n+1}.$$
Therefore, it is a covering map over its image. Choosing bases where it is in diagonal form, we see that the number of preimages of an element (or degree of the covering) is equal to the lattice index of $\mathrm{Im}\ \Phi$ inside the primitive lattice that contains it. This index can also be computed as the integral length of the generator of the image of
$$\Lambda^n\Phi:\Lambda^n\ZZ^n\simeq\ZZ\to\Lambda^n\ZZ^{n+1}\simeq\ZZ^{n+1}.$$
The coefficients of this linear map are the cofactors of $\Phi$, so that we take their gcd, which we now compute. If this gcd is non-zero, one cofactor is non-zero, proving the injectivity of $\Phi$.

We start by formally splitting the flat vertices into two, since they have two adjacent flags. The complement of the flat vertices is a tree. The minors are obtained by deleting a row, thus corresponding to an edge, flag of a flat vertex, or a pearl. However, minors corresponding to a flag of a flat vertex are $0$ since we still have the relation between rows given by $\psi$. Thus, choose an edge or a pearl and the associated minor. We now explain the computation of the minors, as outlined before the proof. We compute the minors inductively with the two following steps. Each pruning step corresponds to an determinant expansion with respect to a row.
\begin{enumerate}[label=(\arabic*)]
\item \textbf{Pruning of leafs:} If a flag is the only one adjacent to a given vertex, the coefficient in the row corresponding to the vertex is the only non-zero coefficient of the row, since there is a unique flag. We say that we have a \textit{leaf}. Thus, we may expand with respect to the row, and delete the corresponding row/column. This coefficient is $1$ for the flags adjacent to flat vertices, and $w_e$ for flags adjacent to a pearl. On the tree, we delete the flag.

\item \textbf{Pruning of lonely flags:} After iterations of the first step, there are flags $(\pfk,e)$ for which the other flag containing $e$ has been deleted, called \textit{lonely flag}. For such a flag, the only non-zero element in the row corresponding to the edge $e$ is the one coming from $(\pfk,e)$, since the other has been deleted. We can thus expand with respect to the row, whose coefficient is $k_\pfk$, and delete the flag from the graph.
\end{enumerate}
We alternate the pruning of leafs and lonely flags. The process ends as follows:
\begin{itemize}
\item For the minor associated to a pearl $\pfk_0$, we end by pruning the adjacent flags, that at that point have become lonely flags, and get
$$k_{\pfk_0}^{\val\pfk_0}\prod_{\pfk\neq\pfk_0} k_\pfk^{\val\pfk-1}\prod_{e\in E_\square(\Dfk)}w_e.$$
\item For the minor associated to an edge $e_0\in E_\square(\Dfk)$, we end by pruning both adjacent flags, that at that point have become leafs, yielding
$$w_{e_0}^2\prod_{\pfk} k_\pfk^{\val\pfk-1}\prod_{\substack{e\in E_\square(\Dfk) \\ e\neq e_0}}w_e.$$
\item For the minor associated to an edge $e_0$ adjacent to a flat vertex, we end by pruning the flag adjacent to a pearl, that has become a leaf at that point. We get
$$w_{e_0}\prod_{\pfk} k_\pfk^{\val\pfk-1}\prod_{e\in E_\square(\Dfk)}w_e.$$
\end{itemize}
Taking the gcd, we get
$$\mathrm{gcd}(w_e,k_\pfk)\prod_\pfk k_\pfk^{\val\pfk-1}\prod_{E_\square(\Dfk)} w_e.$$
This number is non-zero, so that $\Phi$ is indeed injective. To get the number of tropical curves, one has to divide by the deck transformations at each pearl since translation by a torsion element yields the same tropical curve, but we also have to account for the position of the marked point located on each pearl, that can have $k_\pfk$ different values. Both cancel each other. Therefore, the above formula gives the right result.

\textbf{Step D: multiplicities of the tropical curves.} To finish, we deal with the multiplicities of the tropical curve encoded by the diagram. Vertices of tropical curve are in bijection with flags $(\pfk,e)$ of the diagram, with $\pfk\in V_\square(\Dfk)$ a pearl. If the pearl is a cover of degree $k_\pfk$, the vertex has multiplicity $m_V=\frac{d_\pfk}{k_\pfk}w_e$. The multiplicity $m_\Gamma$ and $m_\Gamma^\refined$ follow.
\end{proof}

\begin{rem}
For readers of \cite{blomme2022abelian1} or \cite{blomme2022abelian2}, the computation of the number of tropical curves encoded by a diagram may look like the computation of the complex multiplicity of a tropical curve inside an abelian surface: in both cases we take the gcd of the minors of a matrix.

More precisely, the gcd is just a practical method to compute a lattice index: in each situation, we have a lattice map given by a matrix $\phi:\ZZ^n\to\ZZ^{n+1}$, with natural bases indexed by edges or vertices, and the image is known to lie in a hyperplane with some explicit equation. Such a lattice map appears because in both situation we are trying to put \textit{phases} (\textit{i.e.} elements of a group $G$ with many torsion elements: $\RR/\ZZ$, $\CC^*$ or $S^1\subset\CC^*$) on the edges of a graph, with some conditions: evaluation of some phases, or balancing condition such as the Menelaus relation. In any case, we wish to determine the number of preimages of a given point, \textit{i.e.} the kernel of the group morphism $\phi\otimes G$, which in our case is the lattice index of $\mathrm{Im}\phi$ inside the primitive lattice it spans.
	\begin{itemize}[label=$\triangleright$]
	\item In the tropical curve setting, phases are complex numbers that encode the position of the nodes dual to the edges of the tropical curve.
	\item In the case of pearl diagrams, phases correspond to the second coordinates inside the tropical abelian surface.
	\end{itemize}
\end{rem}

\subsection{Recovering the invariants}
\label{sec-diagram-mult}

\subsubsection{First diagram multiplicities.} For a diagram $\Dfk$, let $m_a(\Dfk)$ and $M_a(\Dfk)$ be the counts of curves in the class $\left(\begin{smallmatrix}
d_1 & a \\
0 & d_2 \\
\end{smallmatrix}\right)$ encoded in $\Dfk$ with multiplicities $m_\Gamma$ and $m_\Gamma^\refined$ respectively. To simplify notations, we set $\mathds{1}_\square(e)$ to be $1$ if $e\in E_\square(\Dfk)$ and $0$ else, and consider the product of the number of curves for a fixed loop data with their (refined) multiplicity, up to the gcd term:
\begin{align*}
\xi_\Dfk(k_\pfk)= & \prod_\pfk d_\pfk^{\val\pfk-1}\frac{d_\pfk}{k_\pfk}\prod_e w_e^{|\{\pfk\in e\}|+\mathds{1}_\square(e)}, \\
\Xi_\Dfk(k_\pfk)= & \prod_\pfk k_\pfk^{\val\pfk-1}\left(\prod_{e\ni\pfk} \left[\frac{d_\pfk}{k_\pfk}w_e\right]_-\right)\prod_{E_\square(\Dfk)} w_e, \\
\end{align*}
which are the product of the number of tropical curves encoded by a diagram for a given loop data, and their common (refined) multiplicity. The term $|\{\pfk\in e\}|=1+\mathds{1}_\square(e)$ is the number of pearls adjacent to $e$. Thus, in $\xi_\Dfk(k_\pfk)$, the exponent is either $1$ or $3$ depending on whether $e$ is adjacent to a flat vertex or not. Following Proposition \ref{proposition curves in diagram point case}, we have the following. 

\begin{coro}
One has
$$m_a(\Dfk) = \sum_{\substack{k_\pfk|d_\pfk \\ \mathrm{gcd}(w_e,k_\pfk)|a }}\mathrm{gcd}(k_\pfk,w_e)\xi_\Dfk(k_\pfk), \text{ and }
M_a(\Dfk)   = \sum_{\substack{k_\pfk|d_\pfk \\ \mathrm{gcd}(w_e,k_\pfk)|a }} \mathrm{gcd}(k_\pfk,w_e)\Xi_\Dfk(k_\pfk).$$
In particular, for $a=0$, the sum is over $k_\pfk|d_\pfk$ without gcd condition, and for $a=1$, the condition is for the gcd to be $1$.
\end{coro}

\begin{proof}
The formula for $m_a$ (resp. $M_a$) is obtained by making the sum over all loop data $k_\pfk|d_\pfk$ that contribute solutions of the product between the common multiplicity $m_\Gamma$ (resp. $m_\Gamma^\refined$) of the tropical curves and their number. Following Proposition \ref{proposition curves in diagram point case}, one needs $\mathrm{gcd}(w_e,k_\pfk)|a$ for the loop data to contribute.
\end{proof}

\subsubsection{Auxiliary multiplicities.} The presence of a gcd makes the computation of the sums in the above proposition quite difficult. To help us give compact expressions of these multiplicities, we introduce the following auxiliary multiplicities where we fix the value of $\mathrm{gcd}(w_e,k_\pfk)$. These expressions do not have an enumerative interpretation in terms of counts of tropical curves encoded in the diagram $\Dfk$. Only $\omega_1=m_1$ and $\Omega_1=M_1$ correspond to the counts of curve in the class $\left(\begin{smallmatrix}
d_1 & 1 \\
0 & d_2 \\
\end{smallmatrix}\right)$:
\begin{align*}
\omega(\Dfk) & =\sum_{k_\pfk|d_\pfk}\xi_\Dfk(k_\pfk), &
\Omega(\Dfk) & =\sum_{k_\pfk|d_\pfk}\Xi_\Dfk(k_\pfk), \\
\omega_k(\Dfk) & =\sum_{\substack{k_\pfk|d_\pfk \\ \mathrm{gcd}(w_e,k_\pfk)=k}}\xi_\Dfk(k_\pfk), &
\Omega_k(\Dfk) & =\sum_{\substack{k_\pfk|d_\pfk \\ \mathrm{gcd}(w_e,k_\pfk)=k}}\Xi_\Dfk(k_\pfk).\\
\end{align*}

The following proposition relates these various multiplicities through the arithmetic convolution.

\begin{prop}
We have the following expression:
$$
\omega(\Dfk)  =\prod_\pfk d_\pfk^{\val\pfk-1}\sigma_1(d_\pfk)\prod_e w_e^{|\{\pfk\in e\}|+\mathds{1}_\square(e)} ,\ \Omega(\Dfk)  =\prod_\pfk\left(\sum_{k|d_\pfk}k^{\val\pfk-1}\prod_{e\ni\pfk}\left[\frac{d_\pfk}{k}w_e\right]_-\right)\prod_{E_\square(\Dfk)}w_e\ ,$$
and identities:
\begin{align*}
\omega_k(\Dfk) & = k^{4g-5}\omega_1(\Dfk/k) ,& \Omega_k(\Dfk) & = k^{2g-3}\Omega_1(\Dfk/k)(q^k) ,\\
\omega & =\epsilon_{4g-5}\ast\omega_1 , & \Omega& = \epsilon_{2g-3}\ast\Omega_1 ,\\
m_0 & =\epsilon_{4g-4}\ast\omega_1 , & M_0 & = \epsilon_{2g-2}\ast\Omega_1 ,\\
\end{align*}
where the last two rows are equalities between functions on the set of diagrams, and the convolution is made over the divisors of a diagram, as explained in Section \ref{section convolution}.
\end{prop}

\begin{proof}
\begin{itemize}
\item As both $\xi$ and $\Xi$ express as a product over the pearls, we can factor the $\sum_{k_\pfk|d_\pfk}$ into a product over all the pearls $\pfk$, which yields the first two identities.

\item The homogeneity relations are obtained as follows: assume $\gcd(w_e,k_\pfk)=k$, then we can write $k_\pfk=k\widehat{k_\pfk}$, $d_\pfk=k\widehat{d_\pfk}$ and $w_e=k\widehat{w_e}$, and we now have $\mathrm{gcd}(\widehat{w_e},\widehat{k_\pfk})=1$. As we have all divided by $k$, the sum is over the loop data with gcd $1$ for the diagram $\Dfk/k$. We just have to determine the power of $k$ that comes out of the sum. Replacing the above in the expressions for $\xi$ and $\Xi$, we have
\begin{align*}
\xi_\Dfk(k_\pfk)= & k^{\sum_\pfk(\val\pfk-1)+\sum_e |\{\pfk\in e\}|+|E_\square(\Dfk)|}\xi_{\Dfk/k}(\widehat{k_\pfk}) , \\
\Xi_\Dfk(k_\pfk)= & k^{\sum_\pfk (\val\pfk-1)+|E_\square(\Dfk)|}\xi_{\Dfk/k}(\widehat{k_\pfk})(q^k). \\
\end{align*}
To conclude, we only need to see that
\begin{align*}
\sum_\pfk\val\pfk = \sum_e |\{\pfk\in e\}| = & 2g-2,\\
|E_\square(\Dfk)|-|V_\square(\Dfk)| = & 1 .\\
\end{align*}
Those easily come from the following relations for the graph $\Dfk$:
$$\begin{array}{lrcl}
\blacktriangleright\text{ label of vertices: } & |V_\square|+|V_\circ|=|V| &=& g , \\[8pt]
\blacktriangleright\text{ flat vertices are bivalent: } & |E|-|E_\square| &=& 2|V_\circ| \\
 & \Rightarrow |E|&=& |E_\square|+2|V_\circ|, \\[8pt]
\blacktriangleright\text{ complement without cycle: } & |V|-|E|+|V_\circ| &=& 1  \\
 & \Rightarrow |E_\square|+|V_\circ| &=& g-1 \\
 & \text{ and } |E_\square|-|V_\square| &=& -1 , \\[8pt]
\blacktriangleright\text{ count of flags: } & \sum_\pfk\val\pfk + 2|V_\circ| &=& 2|E| \\
 & \Rightarrow \sum\val\pfk &=& 2|E_\square|+2|V_\circ|=2g-2. \\
\end{array}$$
\item Identities on the last two rows follow from a writing of the sums of $\omega$, $m_0$, $\Omega$ and $M_0$ in terms of the value of $\mathrm{gcd}(w_e,k_\pfk)$ and the homogeneity relation proven above:
\begin{align*}
\omega(\Dfk) & =\sum_{k|\Dfk}\omega_k(\Dfk)=\sum_{k|\Dfk}k^{4g-5}\omega_1(\Dfk/k)=\epsilon_{4g-5}\ast\omega_1(\Dfk), \\
m_0(\Dfk) & =\sum_{k|\Dfk}k\omega_k(\Dfk)=\sum_{k|\Dfk}k^{4g-4}\omega_1(\Dfk/k)=\epsilon_{4g-4}\ast\omega_1(\Dfk), \\
\Omega(\Dfk) & =\sum_{k|\Dfk}\Omega_k(\Dfk)=\sum_{k|\Dfk}k^{2g-3}\Omega_1(\Dfk/k)(q^k)=\epsilon_{2g-3}\ast\Omega_1(\Dfk),\\
M_0(\Dfk) & =\sum_{k|\Dfk}k\Omega_k(\Dfk)=\sum_{k|\Dfk}k^{2g-2}\Omega_1(\Dfk/k)(q^k)=\epsilon_{2g-2}\ast\Omega_1(\Dfk).\\
\end{align*}
\end{itemize}
\end{proof}

\subsubsection{Final multiplicities.} Using the homogeneity of the curve multiplicities, we have already seen the following convolution relations between the invariants counts of tropical curves, where the invariants are seen as functions on the set of curve classes, and we use the convolution from Section \ref{section convolution}:
$$\begin{array}{ccc}
M_g=\epsilon_{4g-4}\ast N_g^\prim, & & N_g=\epsilon_{4g-3}\ast N_g^\prim.\\
\end{array}$$
As these relations come from a bijection between tropical curves, they can be passed to the level of diagrams, and be extended to the refined case. Let $\mu(\Dfk)$ (resp. $\Upsilon(\Dfk)$) be the count of curves in the class $\left(\begin{smallmatrix} d & 0 \\ 0 & dn \\ \end{smallmatrix}\right)$ encoded in the diagram $\Dfk$  using multiplicity $\delta_\Gamma m_\Gamma$ (resp. $M_\Gamma$). Similarly, let $\mu_1(\Dfk)$ (resp. $\Upsilon_1(\Dfk)$) be the count of \textit{primitive} curves in the class $\left(\begin{smallmatrix} d & 0 \\ 0 & dn \\ \end{smallmatrix}\right)$ encoded in the diagram $\Dfk$ using multiplicity $\delta_\Gamma m_\Gamma(=m_\Gamma)$ (resp. $M_\Gamma(=m_\Gamma^\refined)$).

\begin{prop}\label{proposition final multiplicities point case}
We have the following equalities between functions on the space of diagrams:
$$\begin{array}{>{\displaystyle}r>{\displaystyle}lc>{\displaystyle}r>{\displaystyle}l}
m_0 & = \epsilon_{4g-4}\ast\mu_1 , & & M_0(\Dfk) & =\sum_{k|\Dfk} \Upsilon_1(\Dfk/k)(q^{k^2}), \\
\mu & = \epsilon_{4g-3}\ast\mu_1 , & & \Upsilon & = (\epsilon_{2g-2}\varphi)\ast M_0. \\
\end{array}$$
\end{prop}

\begin{proof}
A marked diagram $\Dfk$ encodes a set $\CCC(\Dfk)$ of tropical curves in the class $\left(\begin{smallmatrix} d & 0 \\ 0 & dn \\ \end{smallmatrix}\right)$. The count of these curves using the multiplicities $m_\Gamma$ is by definition $m_0(\Dfk)$, while their count using multiplicity $m_\Gamma^\refined$ gives $M_0(\Dfk)$. We can group together the tropical curves according to the value of their gcd: if $\CCC_k(\Dfk)$ denotes the subset of curves with gcd $k$, we have
\begin{align*}
m_0(\Dfk)  = \sum_{k|\Dfk}\sum_{\Gamma\in \CCC_k(\Dfk)}m_\Gamma, & & M_0(\Dfk)  =  \sum_{k|\Dfk}\sum_{\Gamma\in \CCC_k(\Dfk)}m_\Gamma^\refined .\\
\end{align*}
Using the fact that $m_{k\Gamma}=k^{4g-4}m_\Gamma$ and $m^\refined_{k\Gamma}=m_\Gamma^\refined(q^{k^2})$, we get
\begin{align*}
m_0(\Dfk)  = \sum_{k|\Dfk}k^{4g-4}\sum_{\Gamma\in \CCC_1(\Dfk/k)}m_\Gamma , & & M_0(\Dfk)  =  \sum_{k|\Dfk} \sum_{\Gamma\in \CCC_1(\Dfk/k)}m_\Gamma^\refined(q^{k^2}),\\
\end{align*}
which give identities on the first row. Multiplying by $k$ the multiplicity, we get similarly $\mu=\epsilon_{4g-3}\ast\mu_1$. Concerning the last identity, one has:
\begin{align*}
(\epsilon_{2g-2}\varphi)\ast M_0(\Dfk) & = \sum_{k|\Dfk}k^{2g-2}\varphi(k)\left(\sum_{\Gamma\in\CCC(\Dfk/k)}m_\Gamma^\refined\right)(q^k) & \\
& = \sum_{k|\Dfk}k^{2g-2}\varphi(k)\left(\sum_{l|\Dfk/k}\sum_{\Gamma\in\CCC_l(\Dfk/k)}m_\Gamma^\refined\right)(q^k) & \\
& = \sum_{k|\Dfk}k^{2g-2}\varphi(k)\sum_{l|\Dfk/k}\sum_{\Gamma\in\CCC_1(\Dfk/kl)}m_\Gamma^\refined(q^{kl^2}) & \\
& = \sum_{\delta|\Dfk}\sum_{k|\delta}k^{2g-2}\varphi(k)\sum_{\Gamma\in\CCC_1(\Dfk/\delta)}m_\Gamma^\refined(q^{\delta^2/k}) & \text{ with }\delta=kl \\
& = \sum_{\delta|\Dfk}\sum_{\Gamma\in\CCC_\delta(\Dfk)}\sum_{k|\delta}k^{2g-2}\varphi(k)m_\Gamma^\refined(q^{1/k}) & \\
 & = \sum_{\Gamma\in\CCC(\Dfk)} M_\Gamma. & \\
\end{align*}
\end{proof}

Making the sum over the genus $g$ bidegree $(d,dn)$ diagrams, we get the following.

\begin{coro}
One has:
\begin{align*}
M_g=\epsilon_{4g-4}\ast N_g^\prim, & & BG_{g,C}=\sum_{k|C}BG_{g,C/k,1}(q^{k^2}) , \\
N_g=\epsilon_{4g-3}\ast N_g^\prim, & & R_{g,C}=(\epsilon_{2g-2}\varphi)\ast BG_{g,C}. \\
\end{align*}
\end{coro}

These relations mean that in the classical case as well as in the refined case, the knowing of one family of invariants is sufficient to determine the others. From the above identities, we have simultaneously
$$m_0=\epsilon_{4g-4}\ast m_1 \text{ and }m_0=\epsilon_{4g-4}\ast\mu_1.$$
Thus, we get that $m_1=\mu_1$. This identity may appear as a miracle because making the sum over all diagrams, it asserts that the count of \textit{primitive} curves in an arbitrary class $\left(\begin{smallmatrix} d_1 & 0 \\ 0 & d_2 \\ \end{smallmatrix}\right)$ is equal to the number of curves in the primitive class having same square: $\left(\begin{smallmatrix} 1 & 0 \\ 0 & d_1 d_2 \\ \end{smallmatrix}\right)$. In particular, it only depends on $d_1 d_2$ and not on the divisibility of the class.

We now give the main statement concerning the enumeration of curves via the diagrams.

\begin{theo}\label{theorem diagram enumeration g points} We have the following table giving a correspondence between the multiplicity with which to count the pearl diagrams of genus $g$ and bidegree $(d_1,d_2)$, and the invariants that their count give:
\begin{center}
\begin{tabular}{cc}
\begin{tabular}{c|c}
Multiplicity & Invariant computed \\
\hline
$m_0$ & $M_{g,\left(\begin{smallmatrix}
d_1 & 0 \\
0 & d_2 \\
\end{smallmatrix}\right) }$ \\
$m_1(=\mu_1=\omega_1)$ & $N_{g,\left(\begin{smallmatrix}
d_1 & 1 \\
0 & d_2 \\
\end{smallmatrix}\right) }$\\
$m_a$ & $M_{g,\left(\begin{smallmatrix}
d_1 & a \\
0 & d_2 \\
\end{smallmatrix}\right) }$ \\
\hline
$\mu_1(=m_1=\omega_1)$ & $N_{g,\left(\begin{smallmatrix}
d_1 & 0 \\
0 & d_2 \\
\end{smallmatrix}\right),1 }$ \\
$\mu$ & $N_{g,\left(\begin{smallmatrix}
d_1 & 0 \\
0 & d_2 \\
\end{smallmatrix}\right) }$ \\
\end{tabular} &
\begin{tabular}{c|c}
Multiplicity & Invariant computed \\
\hline
$M_0$ & $BG_{g,\left(\begin{smallmatrix}
d_1 & 0 \\
0 & d_2 \\
\end{smallmatrix}\right) }$ \\
$M_1$ & $BG_{g,\left(\begin{smallmatrix}
d_1 & 1 \\
0 & d_2 \\
\end{smallmatrix}\right) }$\\
$M_a$ & $BG_{g,\left(\begin{smallmatrix}
d_1 & a \\
0 & d_2 \\
\end{smallmatrix}\right) }$ \\
\hline
$\Upsilon_1$ & $BG_{g,\left(\begin{smallmatrix}
d_1 & 0 \\
0 & d_2 \\
\end{smallmatrix}\right),1 }$ \\
$\Upsilon$ & $R_{g,\left(\begin{smallmatrix}
d_1 & 0 \\
0 & d_2 \\
\end{smallmatrix}\right) }$ \\
\end{tabular}
\end{tabular}
\end{center}
\end{theo}

\begin{proof}
The theorem follows from the definition of the diagram multiplicities.
\end{proof}

\subsection{The multiple cover formulas}

The diagram multiplicities give an algorithm to compute the invariants. However, the numerous relations between these multiplicities allow one to prove \textit{multiple cover formulas} which once proven provide a far more effective algorithm for the computation of the invariants. One can then forget about all the diagram considerations.

\begin{theo}\label{theorem multiple cover formulas point}
We have the following multiple cover formulas
$$\begin{array}{c>{\displaystyle}r>{\displaystyle}l}
(i) & M_{g,\left(\begin{smallmatrix}
d_1 & 0 \\
0 & d_2 \\
\end{smallmatrix}\right)} & =\sum_{k|\mathrm{gcd}(d_1,d_2)}k^{4g-4}N_{g,\left(\begin{smallmatrix}
1 & 0 \\
0 & d_1 d_2/k^2 \\
\end{smallmatrix}\right)} , \\
(ii) & N_{g,\left(\begin{smallmatrix}
d_1 & 0 \\
0 & d_2 \\
\end{smallmatrix}\right)} & =\sum_{k|\mathrm{gcd}(d_1,d_2)}k^{4g-3}N_{g,\left(\begin{smallmatrix}
1 & 0 \\
0 & d_1 d_2/k^2 \\
\end{smallmatrix}\right)} , \\
(iii) & BG_{g,\left(\begin{smallmatrix}
d_1 & 0 \\
0 & d_2 \\
\end{smallmatrix}\right)} & =\sum_{k|\mathrm{gcd}(d_1,d_2)}k^{2g-2}BG_{g,\left(\begin{smallmatrix}
1 & 0 \\
0 & d_1 d_2/k^2 \\
\end{smallmatrix}\right)}(q^k). \\
(iv) & R_{g,\left(\begin{smallmatrix}
d_1 & 0 \\
0 & d_2 \\
\end{smallmatrix}\right)} & =\sum_{k|\mathrm{gcd}(d_1,d_2)}k^{2g-1}R_{g,\left(\begin{smallmatrix}
1 & 0 \\
0 & d_1 d_2/k^2 \\
\end{smallmatrix}\right)}(q^k). \\
\end{array}$$
\end{theo}

\begin{proof}
\begin{enumerate}[label=(\roman*)]
\item The first formula is obtained by taking the sum over all diagrams of the identity $m_0=\epsilon_{4g-4}\ast m_1$, which is true since $m_1=\omega_1$. On the left side we get $M_{g,\left(\begin{smallmatrix}
d_1 & 0 \\
0 & d_2 \\
\end{smallmatrix}\right)}$ and on the right side $N_{g,\left(\begin{smallmatrix}
d_1/k & 1 \\
0 & d_2/k \\
\end{smallmatrix}\right)}=N_{g,\left(\begin{smallmatrix}
1 & 0 \\
0 & d_1d_2/k^2 \\
\end{smallmatrix}\right)}$ for each summand in the convolution.
\item The second formula follows from the identity $\mu_1=m_1$, which are both equal to $\omega_1$. The number of primitive curves $N_{g,\left(\begin{smallmatrix}
d_1 & 0 \\
0 & d_2 \\
\end{smallmatrix}\right),1 }$ in the class $\left(\begin{smallmatrix}
d_1 & 0 \\
0 & d_2 \\
\end{smallmatrix}\right)$ is thus equal to $N_{g,\left(\begin{smallmatrix}
d_1 & 1 \\
0 & d_2 \\
\end{smallmatrix}\right)}=N_{g,\left(\begin{smallmatrix}
1 & 0 \\
0 & d_1d_2 \\
\end{smallmatrix}\right)}$. The identity $N_g=\epsilon_{4g-3}\ast N_g^\prim$ yields the desired relation.
\item As $M_1=\Omega_1$, the relation $M_0=\epsilon_{2g-2}\ast\Omega_1$ summed over all the pearl diagrams can already be interpreted as a multiple cover formula.
\item Last, following Proposition \ref{proposition final multiplicities point case}, one has the relation $\Upsilon=(\epsilon_{2g-2}\varphi)\ast M_0$ among diagram multiplicities. Thus, one has
$$\Upsilon=(\epsilon_{2g-2}\varphi)\ast\epsilon_{2g-2}\ast M_1.$$
Moreover,
\begin{align*}
(\epsilon_{2g-2}\varphi)\ast\epsilon_{2g-2}(d) & =\sum_{k|d}k^{2g-2}\varphi(k)\left(\frac{d}{k}\right)^{2g-2} \\
 & = d^{2g-2}\sum_{k|d}\varphi(k) \\
 & = d^{2g-1}.\\
\end{align*}
As multiplicity $\Upsilon$ yields $R_{g,C}$ and $M_1$ yields the count for the primitive class, the sum over all pearl diagrams gives the second refined multiple cover formula.
\end{enumerate}

\end{proof}

\begin{rem}
The refined invariants introduced in \cite{blomme2022abelian1} using the refined multiplicity $\delta_\Gamma m_\Gamma^\refined$ can also be recovered by convoluting the previous functions. However, they do not seem to satisfy such a cover formula.
\end{rem}

Using the correspondence theorem from \cite{nishinou2020realization} stating $N_{g,C}=\N_{g,C}$, we immediately deduce the following particular case of Conjecture \ref{conj oberdieck formula} from G. Oberdieck.

\begin{coro}
The complex invariant $\N_{g,C}$ satisfy the multiple cover formula
$$\N_{g,C}=\sum_{k|C}k^{4g-3}\N_{g,\varphi_k(C/k)},$$
where $\varphi_k(C/k)$ is a primitive class having the same square as $C/k$.
\end{coro}



\section{FLS-diagrams for curves in a linear system}

We now adapt the setting from the previous section to work for the invariants dealing with curves in a fixed linear system. We follow the exact same steps. Thus, we skip some of the details and refer to the previous section for more.

\subsection{From tropical curves to FLS-diagrams}

We consider a stretched abelian surface $\TT A$ with the corresponding curve class $C=\left(\begin{smallmatrix} d_1 & a \\
0 & d_2\\
\end{smallmatrix}\right)$ and an hexagonal chosen fundamental domain. Choose a doubly stretched configuration $\P$ of $g-2$ points: points are spread horizontally, but not in a uniform way as they stay close to the left of the fundamental domain. Let $0<l\ll x_1<\cdots<x_{g-2}=x\ll L$ be their horizontal coordinate in the fundamental domain, so that $|x_i-x_j|\gg l$ and $L\gg x$. Let $\gamma\subset\TT A$ be the vertical loop in the class $\lambda_2$ that consists of the two left side of the hexagon. Recall that the boundary of the hexagon is the image of a genus $2$ tropical curve $\Gamma_0$ in the class $C$. These choices are recapped on Figure \ref{fig-FLS-choice}. We call the \textit{point zone} the points in the fundamental domain with horizontal coordinate in $[x_1/2;2x]$, the \textit{zero-zone} those with first coordinate in $[0;x_1/2]$, and \textit{far zone} those with horizontal coordinate in $[2x;L]$.

\begin{figure}
\centering
\begin{tikzpicture}[line cap=round,line join=round,>=triangle 45,x=0.4cm,y=0.4cm]
\draw (0,0)--++(25,0);
\draw[line width=2pt] (0,0)--++(0,2)--++(0.5,0.5);
\draw (0.5,2.5)--++(25,0);
\draw (0,1) node[left] {$\gamma$};
\draw (25,0)--++(0.5,0.5)--++(0,2);
\draw (3,1.25) node {$\times$} node[below] {$p_1$};
\draw (6,1.25) node {$\cdots$};
\draw (9,1.25) node {$\times$} node[below] {$p_{g-2}$};
\draw[<->] (0,-1)--(1.5,-1) node[midway,below] {0-zone};
\draw[<->] (1.5,-1)--(12,-1) node[midway,below] {point-zone};
\draw[<->] (12,-1)--(25,-1) node[midway,below] {far-zone};
\end{tikzpicture}
\caption{\label{fig-FLS-choice}Choice of constraints for the fixed linear system problem}
\end{figure}
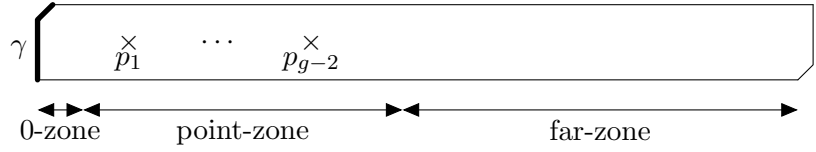

If $h:\Gamma\to\TT A$ is a genus $g$ tropical curve in the class $C$ passing through $\P$ and belonging to the linear system of $\Gamma_0$, we get a diagram $\Dfk$ out of $\Gamma$ exactly as in the point case. The only difference is that the labeling of vertices is indexed by $\P=\{p_1,\dots,p_{g-2}\}$, and a pearl may receive several labels or none. The following proposition is analog to Proposition \ref{proposition p marked diagram} for the case of curves in a linear system.

\begin{prop}
Let $h:\Gamma\to\TT A$ be a genus $g$ parametrized tropical curve in the class $\left(\begin{smallmatrix} d_1 & a \\
0 & d_2\\
\end{smallmatrix}\right)$ and linear system of $\Gamma_0$ passing through the doubly stretched configuration $\P$ of $g-2$ points. Let $\Dfk$ be the diagram obtained from $h:\Gamma\to\TT A$. Then there is a unique pearl without label in the far zone or the zero zone. We label it with $\infty$. with this labeling, $\Dfk$ is a FLS-diagram of genus $g$ and bidegree $(d_1,d_2)$. Furthermore, each pearl is homeomorphic to a circle.
\end{prop}

\begin{rem}\label{remark index loop diagram}
The cycle in the complement of the marked points plays a role analog to the genus $2$ graph $\Sigma$ on which $\Gamma\backslash\P$ retracts. Similarly to the notation $\Lambda_\Gamma^\Sigma$, we denote by $\Lambda_\Dfk$ the class realized by the cycle inside $H_1(S^1,\ZZ)$.
\end{rem}

\begin{proof}
We show that the diagram $\Dfk$ satisfies the definition of a FLS-diagram.
\begin{enumerate}[label=(\alph*)]
\item[(a')] The vertices are split into flat vertices and pearls. We need to show that each pearl has a unique label, except one in the zero or far zone. As in the point case, since the point configuration is stretched, by Lemmas \ref{lem-finite-slopes} and \ref{lem-finite-length}, each pearl has at most one label. Furthermore, as in the point case, pearls have no disconnecting edges and are at least of genus $1$. We know that the complement of marked points on $\Gamma$ is connected and retracts on a genus $2$ subgraph $\Sigma$, with $H_1(\Sigma,\ZZ)\hookrightarrow H_1(\TT A,\ZZ)$. Therefore, at most one pearl is without label.

\textbf{There is unique pearl without label in the far zone:} As $\Gamma$ is in the linear system of $\Gamma_0$, the element $K\in\Lambda_\RR^*$ defined in Section \ref{sec-curves-in-FLS} lies in the image of $S^T:M\to\Lambda_\RR^*$. Assume there is no pearl in the far zone, \textit{i.e.} every edge in the far zone is horizontal. We use the information to bound $K$. Let $A$ be a bound on the coordinates of the slopes of curves in $C$, provided by Lemma \ref{lem-finite-slopes}. As $\gamma$ realizes the class $\lambda_2$, we have
$$|K(\lambda_2)|\leqslant \sum_{p\in\gamma\cap\Gamma}|\det(p,u_p)| \leqslant C^2\cdot A\cdot l,$$
where $A$ bounds the coordinates of the slopes $u_p$, and $l$ the coordinates of the intersection points with $\gamma$, and $C^2$ the number of intersection points. As there are only horizontal edges in the far zone, one can choose $\gamma_1$ in the class $\lambda_1$ such that all intersection points lie in the point zone. Let $p=(x(p),y(p))\in [0;x]\times [0;l]$ be an intersection point between such a $\gamma_1$ and $\Gamma$, and let $u_p=(\alpha(p),\beta(p))$ be the slope, with $\beta(p)\geqslant 1$. We have
$$\det(p,u_p)=\beta(p)x(p)-\alpha(p)y(p)<2A x.$$
Furthermore, we have $\det(p,u_p)>\beta(p)x(p)-Al$. Adding over the intersection points, we get that 
$$|K(\lambda_1)|\leqslant C^2\cdot 2Ax\ll L.$$
As $K\in M\subset\Lambda_\RR^*$, which is the image of $S^T=\left(\begin{smallmatrix} L & d_2\varepsilon \\ la/d_2+d_1\varepsilon & l \\ \end{smallmatrix}\right)$, and $L\gg x\gg l$, these two inequalities force $K$ to be of the form $S^T(ke_2^*)$ for some integer $k\ll x$. In particular, we get that
$$K(\lambda_1)=kd_2\varepsilon \ll l.$$
As $\det(p,u_p)>\beta(p)x(p)-Al$, we get that
$$\sum_{p\in\Gamma\cap\gamma_1} \beta(p)x(p)\leqslant C^2\cdot Al+kd_2\varepsilon \ll x_1.$$
Hence, all the non-horizontal edges (and hence all pearls) intersected by $\gamma_1$ lie in the zone with first coordinate between $0$ and $\frac{x_1}{2}$, where there are no marked points. In this situation, there are no pearl labeled by $1,\dots,g-2$, and thus every marked point lies on a horizontal edge. Therefore, if this is not the case, then there is some non-horizontal edge in the far zone, and thus some pearl, which cannot contain a marked point. In any case, we have a unique pearl without label, which we can abstractly label $\infty$, which lies in the far zone or the zero-zone.

\item[(b)(c)] As in the point case: pearls are balanced thanks to the balancing condition and lengths are non-negative (positive if the edges intersect $\gamma$).
\item[(d')] We need to show that the complement of flat vertices is connected with a unique cycle. We know that the complement of marked points on $\Gamma$ is connected and retracts on a genus $2$ subgraph $\Sigma$, with $\iota:H_1(\Sigma,\ZZ)\hookrightarrow H_1(\TT A,\ZZ)$. Furthermore, by (a'), there is a unique pearl without label. Every cycle in a pearl realizes some class proportional to $\lambda_2$. Thus, by injectivity of $\iota$, the pearl labeled $\infty$ cannot be of genus bigger than $2$. Same, no labeled pearl can be of genus bigger than $2$, so that every pearl is thus homeomorphic to a circle. Finally, as the complement of $\P$ in $\Gamma$ has genus $2$, there is a unique cycle in the complement of flat vertices in $\Dfk$, and it completes the cycle in the pearl $\infty$ into a basis of $H_1(\Sigma,\ZZ)$.
\end{enumerate}
The genus and degree statement are proven as in the point case.
\end{proof}

\subsection{From FLS-diagrams to tropical curves}
\label{section from pearl diagrams to tropical curves}

\begin{prop}\label{proposition curves in diagram linear system case}
Let $\Dfk$ be an FLS-diagram of genus $g$ and bidegree $(d_1,d_2)$. Let $\Lambda_\Dfk$ be the index of $\pi_\Dfk\left(H_1(\Dfk-V_\circ)\right)\subset H_1(\RR/\ZZ)$. Let $k_\pfk|d_\pfk$ be a choice of loop data. If $\mathrm{gcd}(w_e,k_\pfk)$ divides $a$, the diagram gives a number of genus $g$ tropical curves in the class $\left(\begin{smallmatrix}
d_1 & a \\
0 & d_2 \\
\end{smallmatrix}\right)$ and fixed linear system passing through the stretched configuration $\P$ equal to
$$\mathrm{gcd}(w_e,k_\pfk)\Lambda_\Dfk \frac{d_{\pfk_\infty}}{k_{\pfk_\infty}}\prod_\pfk k_\pfk^{\val\pfk-1}\prod_{E_\square(\Dfk)}w_e.$$
Moreover, for each one of them, one has
$$\Lambda_\Gamma^\Sigma= \Lambda_\Dfk k_{\pfk_\infty}, \ m_\Gamma=\prod_\pfk \left(\frac{d_\pfk}{k_\pfk}\right)^{\val\pfk}\prod_e w_e^{|\{\pfk\in e\}|} \text{ and } m_\Gamma^\refined=\prod_\pfk \prod_{e\ni\pfk}\left[ \frac{d_\pfk}{k_\pfk}w_e  \right]_- .$$
If $\mathrm{gcd}(w_e,k_\pfk)$ does not divide $a$, there are no curve encoded by the diagram.
\end{prop}

\begin{proof}
We proceed as in the point case. For each pearl $\pfk$ of degree $d_\pfk$, choose $k_\pfk |d_\pfk$, the number of rounds that the cycle in the pearl does.

\medskip

\textbf{Step A: introduction of a lattice map.} The tropical curves encoded by $\Dfk$ are determined by the following map between real tori, playing the analog of $\Phi_{\RR/\ZZ}$ in the point case:
$$\Psi_{\RR/\ZZ}:\left| \begin{array}{>{\displaystyle}cc>{\displaystyle}l}
  \bigoplus_{E(\Dfk)} \RR/k_{e_-}l\ZZ\oplus\RR/k_{e_+}l\ZZ & \longrightarrow & \RR/l\ZZ\oplus\bigoplus_{E(\Dfk)} \RR/l\ZZ \oplus \bigoplus_{V_\circ(\Dfk)} (\RR/l\ZZ)^2 \oplus \bigoplus_{V_\square(\Dfk)} \RR/ k_\pfk l\ZZ \\
 (x_e^-,x_e^+)_e & \longmapsto & \left( \sum_{e\cap\gamma\neq\emptyset} \ell_e w_e x_e^+ , (x_e^+-x_e^-)_e, (x_{e_v^-}^-,x_{e_v^+}^+)_v , \left( \sum_{e\ni\pfk} \pm w_ex_e^\pm \right)_\pfk \right) \\
\end{array} \right. .$$
As before, the domain corresponds to the position of flags in the degree $k_\pfk$ cover associated to adjacent pearl (or $\RR/l\ZZ$ for a flat vertex). The codomain corresponds to gluing of flags, marked point evaluation, Menelaus relations for each pearl, and a new condition coming from the fixed linear system condition. The latter is obtained as follows. The linear system condition imposes two constraints:
\begin{itemize}
\item The first is on the sum of the $\det(p,u_p)$ of the edges intersected by a loop going in the horizontal direction. It says that the unmarked pearl $\pfk_\infty$ has $d_{\pfk_\infty}$ possible positions.
\item The second is on the sum of the moments of the edges intersected by a loop going in the vertical direction, which we can assume to be $\gamma$. As the curves have a pearl decomposition, we can assume that this loop only intersects horizontal edges. For each intersected edge, we choose one of the adjacent pearls to measure its position, and we then make the sum of the moments, yielding
$$\sum_{e\cap\gamma\neq\emptyset} \ell_e w_e x_e^+.$$
\end{itemize}

The map $\Psi_{\RR/\ZZ}$ is obtained from a map between lattices $\Psi$ whose codomain has rank exactly one more as the rank of its domain. We want to determine the number of preimages of a point $\mathbf{y}=\left(0, (\ell_e d_2\varepsilon)_e,(y_v,y_v)_v,\left(d_\pfk\left(\frac{a}{d_2}l+d_1\varepsilon\right)\right)_\pfk \right)$.

\textbf{Step B: Image of the map $\Psi_{\RR/\ZZ}$ between tori.} As in the case of curves passing through $g$ points, the image of $\Psi_{\RR/\ZZ}$ is the kernel of the morphism $\widehat{\psi}_{\RR/\ZZ}$, and we get that $\mathbf{y}$ is in the image of $\Psi_{\RR/\ZZ}$ if and only if $\mathrm{gcd}(w_e,k_\pfk)$ divides $a$.

\textbf{Step C: degree of $\Psi_{\RR/\ZZ}$.} We now compute the number of tropical curves encoded in $\Dfk$ for a fixed choice of loop data. Notice that the unmarked pearl $\pfk_\infty$ has $d_{\pfk_\infty}$ possible positions because of the linear system condition. Up to multiplication by $d_{\pfk_\infty}$, the number of tropical curves is given by the gcd of the minors of the lattice map $\Psi$ associated to $\Psi_{\RR/\ZZ}$:
$$\Psi: \left| \begin{array}{>{\displaystyle}cc>{\displaystyle}l}
 \bigoplus_{E(\Dfk)} \ZZ\oplus\ZZ & \longrightarrow & \ZZ\oplus \ZZ^{E(\Dfk)} \oplus \ZZ^{2V_\circ(\Dfk)} \oplus \ZZ^{V_\square(\Dfk)}\\
 (x_e^-,x_e^+)_e & \longmapsto & \left( \sum_{e\cap\gamma\neq\emptyset} \ell_e w_e k_{e_+}x_{e_+} ,  (k_{e_+}x_e^+ - k_{e_-}x_e^-), (x_{e_v^-}^-,x_{e_v^+}^+)_v , \left( \sum_{e\ni\pfk} \pm w_ex_e^\pm \right)_\pfk \right) \\
\end{array} \right. .$$
The computation of the index is as in the point case. First, notice that the minor corresponding to deleting the first row is $0$, since $\psi$ still vanishes on the image. Before starting, we do some little modifications to simplify the computation of the minors. The first row of $\Psi$ comes from the edges $e$ with $e\cap\gamma\neq \emptyset$. An edge in $\Dfk$ having length $\ell_e$ has $\ell_e$ intersection points with $\gamma$, hence the factor $\ell_e$.

Let $\widetilde{\Dfk}$ be the genus $1$ graph where the bivalent vertices have been separated, and each remaining edge has been subdivided. The first row of $\Psi$ may be seen as a 1-cochain on $\widetilde{\Dfk}$, \textit{i.e.} an element in $C^1(\Dfk,\ZZ)$, associating to every edge the number of intersection points it has with $\gamma$. Furthermore, the remaining rows of $\Psi$ are actually indexed by the vertices of $\widetilde{\Dfk}$: two per flat vertex, one per edge and the pearls. The rows itself are precisely the coboudaries by which we quotient to get $H^1(\Dfk,\ZZ)$ from $C^1(\Dfk,\ZZ)$, up to the weights $k_\pfk$ along each edge. Therefore, as $\Dfk-V_\circ$ contains a unique cycle, by adding a suitable linear combination of the remaining rows, we may assume that there is a unique non-zero coefficient in the first row, corresponding to a flag belonging to the cycle of $\Dfk-V_\circ$. This amounts to deform $\gamma$ so that it intersects the cycle at a unique edge, and the edges adjacent to bivalent vertices. Furthermore, as the intersection number with the cycle is preserved, the coefficient at the column corresponding to the flag is $\Lambda_\Dfk w_ek_{e_+}$. These are the only coordinates having a non-zero coefficient in the first row.

The computation of the minors is now carried out through the pruning algorithm, with a new step.
\begin{enumerate}[label=(\arabic*)]
\item \textbf{Pruning of leafs:} expand minors with respect to rows corresponding to flags adjacent to a unique vertex.
\item \textbf{Pruning of lonely flags:} expand minors with respect to edgesfor which only one flag remains.
\item \textbf{Opening the cycle:} The first two steps allow one to cut branches of the complement of flat vertices. However, the complement has one cycle, so that the pruning algorithm does not terminate. Once every branches has been pruned, there is only one non-zero coefficient in the first row, corresponding to a unique flag. We expand with respect to this coefficient, and delete the corresponding flag. The cycle is now open, and we can continue with steps (1) and (2) to prune it and finish the computation.
\end{enumerate}

We get the following minors:
\begin{itemize}
\item For a pearl $\pfk_0$ we get
$$k_{\pfk_0}^{\val\pfk_0}\Lambda_\Dfk\prod_{\pfk\neq\pfk_0}k_\pfk^{\val\pfk-1}\prod_{e\in E_\square(\Dfk)}w_e.$$
\item For an edge $e_0\in E_\square(\Dfk)$, we get
$$w_{e_0}^2\Lambda_\Dfk\prod_{\pfk}k_\pfk^{\val\pfk-1}\prod_{\substack{e\in E_\square(\Dfk) \\ e\neq e_0}}w_e.$$
\item For an edge $e_0$adjacent to a flat vertex,
$$w_{e_0}\Lambda_\Dfk\prod_{\pfk}k_\pfk^{\val\pfk-1}\prod_{e\in E_\square(\Dfk)}w_e.$$
\end{itemize}
Taking the gcd yields
$$\gcd(w_e,k_\pfk)\Lambda_\Dfk\prod_{\pfk}k_\pfk^{\val\pfk-1}\prod_{e\in E_\square(\Dfk)}w_e.$$
Finally, this time, we have to divide by the deck transformations of the unmarked pearl $\pfk_\infty$ since there is no marked point to distinguish the lifts that differ from a torsion element.

\textbf{Step D: multiplicities of the tropical curves.} To conclude, we deal with the multiplicities of the tropical curves. As in the case of curves passing through $g$ points, each vertex has multiplicity $\frac{d_\pfk}{k_\pfk}w_e$, yielding the formulas. Moreover, the graph $\Sigma$ on which each curve retracts is provided by the unmarked pearl $\pfk_\infty$, that contains a unique cycle realizing the class $k_{\pfk_\infty}\lambda_2$, and the unique cycle of the complement of marked edges in the diagram $\Dfk$, realizing the class $\Lambda_\Dfk^\mfk$ inside $H_1(S^1,\ZZ)$. We then have $\Lambda_\Gamma^\Sigma=\Lambda_\Dfk k_{\pfk_\infty}$.
\end{proof}

\subsection{Recovering the invariants}

\subsubsection{First diagram multiplicities.} For a diagram $\Dfk$, let $m_a^{FLS}(\Dfk)$ and $M_a^{FLS}(\Dfk)$ be the count of curves in the class $\left(\begin{smallmatrix}
d_1 & a \\
0 & d_2 \\
\end{smallmatrix}\right)$ encoded in $\Dfk$ with multiplicities $\Lambda_\Gamma^\Sigma m_\Gamma$ and $\Lambda_\Gamma^\Sigma m_\Gamma^\refined$ respectively. To simplify notations, for a chosen FLS-diagram $\Dfk$ and loop data $(k_\pfk)$, let us set
\begin{align*}
\xi^{FLS}_\Dfk(k_\pfk) = & \Lambda_\Dfk^2 d_{\pfk_\infty}\prod_\pfk d_\pfk^{\val\pfk-1}\frac{d_\pfk}{k_\pfk}\prod_e w_e^{|\{\pfk\in e\}|+\mathds{1}_\square(e)} ,\\
\Xi^{FLS}_\Dfk(k_\pfk) = & \Lambda_\Dfk^2 d_{\pfk_\infty}\prod_\pfk k_\pfk^{\val\pfk-1}\left(\prod_{e\ni\pfk} \left[\frac{d_\pfk}{k_\pfk}w_e\right]_-\right)\prod_{E_\square(\Dfk)} w_e. .\\
\end{align*}
Following Proposition \ref{proposition curves in diagram linear system case}, we have the following.

\begin{coro}
One has
\begin{align*}
m_a^{FLS}(\Dfk) & =\sum_{\substack{k_\pfk|d_\pfk \\ \mathrm{gcd}(w_e,k_\pfk)|a }}\mathrm{gcd}(k_\pfk,w_e)\xi^{FLS}_\Dfk(k_\pfk), \\
M_a^{FLS}(\Dfk) & =\sum_{\substack{k_\pfk|d_\pfk \\ \mathrm{gcd}(w_e,k_\pfk)|a }}\mathrm{gcd}(k_\pfk,w_e)\Xi^{FLS}_\Dfk(k_\pfk). \\
\end{align*}
For $a=0$, we have no gcd condition, and for $a=1$ we only sum for gcd equal to $1$.
\end{coro}

\subsubsection{Auxiliary multiplicities.} We introduce the auxiliary counts:
\begin{align*}
\omega^{FLS}(\Dfk) & =\sum_{k_\pfk|d_\pfk }\xi^{FLS}_\Dfk(k_\pfk), & \Omega^{FLS}(\Dfk) & =\sum_{k_\pfk|d_\pfk}\Xi^{FLS}_\Dfk(k_\pfk), \\
\omega_k^{FLS}(\Dfk) & =\sum_{\substack{k_\pfk|d_\pfk \\ \mathrm{gcd}(w_e,k_\pfk)=k }}\xi^{FLS}_\Dfk(k_\pfk), & \Omega_k^{FLS}(\Dfk) & =\sum_{\substack{k_\pfk|d_\pfk \\ \mathrm{gcd}(w_e,k_\pfk)=k }}\Xi^{FLS}_\Dfk(k_\pfk).\\
\end{align*}

Only the formulas $\omega^{FLS}_1=m^{FLS}_1$ and $\Omega^{FLS}_1=M^{FLS}_1$ have an enumerative interpretation as a count of curves, the other are just computation additives.

\begin{prop}
Factorizing the products, we have the following expressions:
\begin{align*}
\omega^{FLS}(\Dfk)  &= \Lambda_\Dfk^2 d_{\pfk_\infty}\prod_\pfk d_\pfk^{\val\pfk-1}\sigma_1(d_\pfk)\prod_e w_e^{|\{\pfk\in e\}|+\mathds{1}_\square(e)}  ,\\
 \Omega^{FLS}(\Dfk) & = \Lambda_\Dfk^2 d_{\pfk_\infty}\prod_\pfk \left(\sum_{k|d_\pfk} k^{\val\pfk-1}\prod_{e\ni\pfk}\left[\frac{d_\pfk}{k}w_e\right]_-\right)\prod_{E_\square(\Dfk)} w_e ,\\
\end{align*}
along with the following identities:
\begin{align*}
\omega^{FLS}_k(\Dfk) & = k^{4g-3}\omega_1(\Dfk/k) , & \Omega^{FLS}_k(\Dfk) & = k^{2g-1}\Omega^{FLS}_1(\Dfk/k)(q^{k}) , \\
\omega^{FLS} & = \epsilon_{4g-3}\ast\omega^{FLS}_1 , & \Omega^{FLS} & = \epsilon_{2g-1}\ast\Omega^{FLS}_1 ,\\
m_0^{FLS} & = \epsilon_{4g-2}\ast\omega^{FLS}_1, & M_0^{FLS} & = \epsilon_{2g}\ast\Omega^{FLS}_1 ,\\
\end{align*}

where the last two rows are convolutions on the set of functions on diagrams.
\end{prop}

\begin{proof}
Proof is verbatim to the point case, with different exponent appearing due to the different shape of the multiplicity, and the fact that the complement of flat vertices has genus $1$.
\end{proof}

\subsubsection{Final multiplicities.} Using the homogeneity of the curve multiplicities, we have the following relations between functions over the set of classes:
\begin{align*}
M^{FLS}_g=\epsilon_{4g-2}\ast N^{\prim,FLS}_g, & & N^{FLS}_g=\epsilon_{4g-1}\ast N^{\prim,FLS}_g.\\
\end{align*}
The above equalities come from a bijection at the level of the tropical curves and can be passed to the pearl diagram level and have refined analogs. Let $\mu^{FLS}$ (resp. $\Upsilon^{FLS}$) be the count of curves in the diagram $\Dfk$ using multiplicity $\delta_\Gamma\Lambda_\Gamma^\Sigma m_\Gamma$ (resp. $\Lambda_\Gamma^\Sigma M_\Gamma$). Similarly, let $\mu^{FLS}_1$ (resp. $\Upsilon^{FLS}_1$) be the count of \textit{primitive} curves in the diagram $\Dfk$ using the same multiplicity.

\begin{prop}
We have the following identities:
$$\begin{array}{rlcr>{\displaystyle}l}
m^{FLS}_0 & =\epsilon_{4g-2}\ast \mu^{FLS}_1, & & M^{FLS}_0(\Dfk) & = \sum_{k|\Dfk} k^2\Upsilon^{FLS}_1(\Dfk/k)(q^{k^2}), \\
\mu^{FLS}(\Dfk) & =\epsilon_{4g-1}\ast \mu^{FLS}_1(\Dfk), & & \Upsilon^{FLS} & = (\epsilon_{2g}\varphi)\ast M_0^{FLS}.\\
\end{array}$$
\end{prop}

\begin{proof}
The proof is verbatim to Proposition \ref{proposition final multiplicities point case}.
\end{proof}

Making the sum over all the genus $g$ bidegree $(d_1,d_2)$ pearl diagrams, we get the following.

\begin{coro}
One has
$$\begin{array}{rlcr>{\displaystyle}l}
M_g^{FLS} & =\epsilon_{4g-2}\ast N^{\prim,FLS}_g, & & BG_{g,C}^{FLS} & =\sum_{k|C}BG_{g,C/k,1}(q^{k^2}) , \\
N_g^{FLS} & =\epsilon_{4g-1}\ast N^{\prim,FLS}_g, & & R_{g,C}^{FLS} & =(\epsilon_{2g}\varphi)\ast BG_{g,C} . \\
\end{array}$$
\end{coro}

\begin{rem}
As in the $g$ point case, the surprising identity $\mu_1^{FLS}=m_1^{FLS}(=\omega_1^{FLS})$ that asserts that the number of primitive curves solution of the problem in a class $C$ does not depend on the divisibility of the class: it is equal to the number for a primitive class.
\end{rem}

\begin{theo}
\label{theorem diagram enumeration linear system} We have the following table giving a correspondence between the multiplicity with which we count the FLS-diagrams of genus $g$ and bidegree $(d_1,d_2)$, and the invariants that their count give:
\begin{center}
\begin{tabular}{cc}
\begin{tabular}{c|c}
Multiplicity & Invariant computed \\
\hline
$m^{FLS}_0$ & $M^{FLS}_{g,\left(\begin{smallmatrix}
d_1 & 0 \\
0 & d_2 \\
\end{smallmatrix}\right) }$ \\
$m^{FLS}_1(=\mu_1^{FLS})$ & $M^{FLS}_{g,\left(\begin{smallmatrix}
d_1 & 1 \\
0 & d_2 \\
\end{smallmatrix}\right) }$\\
$m^{FLS}_a$ & $M^{FLS}_{g,\left(\begin{smallmatrix}
d_1 & a \\
0 & d_2 \\
\end{smallmatrix}\right) }$ \\
\hline
$\mu^{FLS}_1(=m_1^{FLS})$ & $N^{FLS}_{g,\left(\begin{smallmatrix}
d_1 & 0 \\
0 & d_2 \\
\end{smallmatrix}\right),1 }$ \\
$\mu^{FLS}$ & $N^{FLS}_{g,\left(\begin{smallmatrix}
d_1 & 0 \\
0 & d_2 \\
\end{smallmatrix}\right) }$ \\
\end{tabular} &
\begin{tabular}{c|c}
Multiplicity & Invariant computed \\
\hline
$M^{FLS}_0$ & $BG^{FLS}_{g,\left(\begin{smallmatrix}
d_1 & 0 \\
0 & d_2 \\
\end{smallmatrix}\right) }$ \\
$M^{FLS}_1$ & $BG^{FLS}_{g,\left(\begin{smallmatrix}
d_1 & 1 \\
0 & d_2 \\
\end{smallmatrix}\right) }$\\
$M^{FLS}_a$ & $BG^{FLS}_{g,\left(\begin{smallmatrix}
d_1 & a \\
0 & d_2 \\
\end{smallmatrix}\right) }$ \\
\hline
$\Upsilon^{FLS}_1$ & $BG^{FLS}_{g,\left(\begin{smallmatrix}
d_1 & 0 \\
0 & d_2 \\
\end{smallmatrix}\right),1 }$ \\
$\Upsilon^{FLS}$ & $R^{FLS}_{g,\left(\begin{smallmatrix}
d_1 & 0 \\
0 & d_2 \\
\end{smallmatrix}\right) }$ \\
\end{tabular}
\end{tabular}
\end{center}
\end{theo}

\subsection{The multiple cover formulas}

\subsubsection{Tropical multiple cover formulas.} From the identities between the pearl diagram multiplicities, we deduce the following cover formulas.

\begin{theo}\label{theorem multiple cover formula linear case}
We have the following multiple cover formulas:
$$\begin{array}{r>{\displaystyle}r>{\displaystyle}l}
(i) & M^{FLS}_{g,\left(\begin{smallmatrix}
d_1 & 0 \\
0 & d_2 \\
\end{smallmatrix}\right)} & =\sum_{k|\mathrm{gcd}(d_1,d_2)}k^{4g-2}N^{FLS}_{g,\left(\begin{smallmatrix}
1 & 0 \\
0 & d_1 d_2/k^2 \\
\end{smallmatrix}\right)} ,\\
(ii) & N^{FLS}_{g,\left(\begin{smallmatrix}
d_1 & 0 \\
0 & d_2 \\
\end{smallmatrix}\right)} & =\sum_{k|\mathrm{gcd}(d_1,d_2)}k^{4g-1}N^{FLS}_{g,\left(\begin{smallmatrix}
1 & 0 \\
0 & d_1 d_2/k^2 \\
\end{smallmatrix}\right)} ,\\
(iii) & BG^{FLS}_{g,\left(\begin{smallmatrix}
d_1 & 0 \\
0 & d_2 \\
\end{smallmatrix}\right)} & =\sum_{k|\mathrm{gcd}(d_1,d_2)}k^{2g}BG^{FLS}_{g,\left(\begin{smallmatrix}
1 & 0 \\
0 & d_1 d_2/k^2 \\
\end{smallmatrix}\right)}(q^k) ,\\
(iv) & R^{FLS}_{g,\left(\begin{smallmatrix}
d_1 & 0 \\
0 & d_2 \\
\end{smallmatrix}\right)} & =\sum_{k|\mathrm{gcd}(d_1,d_2)}k^{2g+1}R^{FLS}_{g,\left(\begin{smallmatrix}
1 & 0 \\
0 & d_1 d_2/k^2 \\
\end{smallmatrix}\right)}(q^k).\\
\end{array}$$
\end{theo}

\begin{proof}
\begin{enumerate}
\item As we have $m_1^{FLS}=\omega_1^{FLS}$, the identity is obtained by making the sum over all diagrams of the identity $m_0^{FLS}=\epsilon_{4g-2}\ast m_1^{FLS}$.
\item As we have $\mu^{FLS}_1=m^{FLS}_1$, which are both equal to $\omega_1^{FLS}$, the relation $\mu^{FLS}=\epsilon_{4g-1}\ast m_1^{FLS}$ provides the result by making the sum over all pearl diagrams yields the relation.
\item As $M_1^{FLS}=\Omega_1^{FLS}$, the identity $M_0=\epsilon_{2g}\ast M_1^{FLS}$ yields the relation by making the sum over all diagrams.
\item The last formula is obtained as in the $g$ points case.
\end{enumerate}
\end{proof}

\subsubsection{Complex multiple cover formulas.} In \cite{bryan1999generating}, it is proven that the linear system condition can be replaced by four insertions of $1$-cycles whose classes form a basis of the homology of $\CC A$. Taking $n=g+2$, $\gamma_1,\dots,\gamma_{g-2}$ to be points and the four remaining cycles to be such a basis, the multiple cover formula predicts
$$\N^{FLS}_{g,\left(\begin{smallmatrix}
d_1 & 0 \\
0 & d_2 \\
\end{smallmatrix}\right)}  =\sum_{k|\mathrm{gcd}(d_1,d_2)}k^{4g-1}\N^{FLS}_{g,\left(\begin{smallmatrix}
1 & 0 \\
0 & d_1 d_2/k^2 \\
\end{smallmatrix}\right)}.$$
By the correspondence theorem asserting that $N_{g,C}^{FLS}=\N_{g,C}^{FLS}$, it is equivalent to relation (ii) in Theorem \ref{theorem multiple cover formula linear case}, which proves the multiple cover formula for this particular case.

We now insert a $\lambda$-class, and take $g_0-2$ point insertions, and four $1$-cycle insertions forming a basis of the first homology group. The multiple cover formula from \cite{oberdieck2022gromov} predicts
$$\langle \lambda_{g-g_0};\mathrm{pt}^{g_0-2}\rangle^{FLS}_{g,C}=\sum_{k|C} k^{2g-1+2g_0}\langle \lambda_{g-g_0};\mathrm{pt}^{g_0-2}\rangle^{FLS}_{g,(C/k)_\mathrm{prim}}.$$
For a fixed choice of $C$ and $g_0$, let
$$\R^{FLS}_{g_0,C}=\sum _{g\geqslant g_0}\langle \lambda_{g-g_0};\mathrm{pt}^{g_0-2}\rangle^{FLS}_{g,C} u^{2g-2},$$
which is the generating series of the Gromov-Witten invariants with point insertions and a $\lambda$-class. Its first term is $N^{FLS}_{g_0,C}u^{2g_0-2}$. As in the point case, the multiple cover formula predicts
\begin{align*}
\R^{FLS}_{g_0,C}(u) & =\sum_{g\geqslant g_0}\sum_{k|C} k^{2g_0+1}\langle \lambda_{g-g_0};\mathrm{pt}^{g_0}\rangle^{FLS}_{g,(C/k)_\mathrm{prim}}(ku)^{2g-2} \\
 & = \sum_{k|C} k^{2g_0+1}\R^{FLS}_{g_0,(C/k)_\mathrm{prim}}(ku). \\
\end{align*}
The generating series $\R^{FLS}_{g_0,C}$ have been computed in \cite{bryan2018curve} for primitive classes and for any class if $g_0=2$. It is possible to check that if $C$ is a primitive class, we have in fact
$$\R^{FLS}_{g_0,C}(u)=(-1)^{g_0-1}R_{g_0,C}(q),$$
where $q=e^{iu}$ is the same change of variable that occurs in \cite{bousseau2019tropical}. The same is true for any $C$ if $g_0=2$. The multiple cover formulas is thus equivalent to the refined tropical multiple cover formula (iv) from Theorem \ref{theorem multiple cover formula linear case}. Thus, the multiple cover formula for invariants with a $\lambda$-class insertion is proven by Theorem \ref{theorem multiple cover formulas point} if the following conjecture analog to the main theorem of \cite{bousseau2019tropical} is true.

\begin{conj}
Through the change of variable $q=e^{iu}$, one has
$$\R^{FLS}_{g_0,C}=(-1)^{g_0-1} R^{FLS}_{g_0,C}(q).$$
\end{conj}

\section{Computations and regularity results}

\subsection{Computation in the primitive case}
\label{section computation primitive case}

\subsubsection{Explicit formulas for the primitive invariants.} To compute the enumerative invariants, the multiple cover formula reduces their computation to the primitive classes. The generating series of the invariants $N_{g,\left(\begin{smallmatrix}
1 & 0 \\
0 & n \\
\end{smallmatrix}\right)}$ and $N^{FLS}_{g,\left(\begin{smallmatrix}
1 & 0 \\
0 & n \\
\end{smallmatrix}\right)}$ have already been computed in \cite{bryan1999generating}. The result can be recovered using the pearl diagrams since the pearl diagram of bidegree $(1,n)$ are especially simple.

\begin{prop}
We have
$$\begin{array}{>{\displaystyle}r>{\displaystyle}l}
N_{g,\left(\begin{smallmatrix}
1 & 0 \\
0 & n \\
\end{smallmatrix}\right)} & = g\sum_{a_1+\cdots+a_{g-1}=n}\prod_{i=1}^{g-1}a_i\sigma_1(a_i), \\
R_{g,\left(\begin{smallmatrix}
1 & 0 \\
0 & n \\
\end{smallmatrix}\right)} & = g\sum_{a_1+\cdots+a_{g-1}=n}\prod_{i=1}^{g-1}\left(\sum_{k|a_i}\frac{a_i}{k_i}\left[k_i\right]_-^2\right), \\
N^{FLS}_{g,\left(\begin{smallmatrix}
1 & 0 \\
0 & n \\
\end{smallmatrix}\right)} & = \sum_{a_1+\cdots+a_{g-1}=n}a_{g-1}\prod_{i=1}^{g-1}a_i\sigma_1(a_i), \\
R^{FLS}_{g,\left(\begin{smallmatrix}
1 & 0 \\
0 & n \\
\end{smallmatrix}\right)} & = \sum_{a_1+\cdots+a_{g-1}=n}a_{g-1}\prod_{i=1}^{g-1}\left(\sum_{k|a_i}\frac{a_i}{k_i}\left[k_i\right]_-^2\right). \\
\end{array}$$
\end{prop}

\begin{proof}
\begin{itemize}
\item A diagram of genus $g$ has a unique flat vertex which has $g$ possible labels, and $g-1$ pearls whose labels are determined by this choice, as depicted on Figure \ref{fig-primitive-pearl-diagrams}(a). We take factor $g$ to account for its marking. We then sum over the partitions of $n$ in $g-1$ parts of the diagram multiplicities. The diagram being primitive, the latter are provided by the formulas for $\omega$ and $\Omega$, yielding the expected formula.
\item For a FLS-diagram, there is no flat vertex, $g-1$ pearls with a unique possible labeling. The counts are obtained by making the sum over the partition of $n$ into $g-1$ parts of the diagram multiplicities, which are this time provided by the formula for $\omega^{FLS}$ and $\Omega^{FLS}$ respectively.
\end{itemize}
\end{proof}

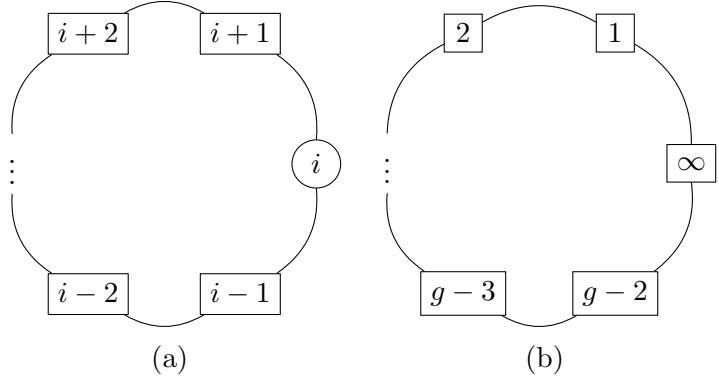
\begin{figure}
\begin{center}
\begin{tabular}{cc}
\begin{tikzpicture}[line cap=round,line join=round,x=1cm,y=1cm]
	\pearl (1) at (-120:2) label=i-2;
	\pearl (2) at (-60:2) label=i-1;
	\flatpearl (3) at (0:2) label=i;
	\pearl (4) at (60:2) label=i+1;
	\pearl (5) at (120:2) label=i+2;
	\node (6) at (180:2) {$\vdots$};
	
	\draw (1) to[bend right] (2);
	\draw (2) to[bend right] (3);
	\draw (3) to[bend right] (4);
	\draw (4) to[bend right] (5);
	\draw (5) to[bend right] (6);
	\draw (6) to[bend right] (1);
\end{tikzpicture}
&
\begin{tikzpicture}[line cap=round,line join=round,x=1cm,y=1cm]
	\pearl (1) at (-120:2) label=g-3;
	\pearl (2) at (-60:2) label=g-2;
	\pearl (3) at (0:2) label=\infty;
	\pearl (4) at (60:2) label=1;
	\pearl (5) at (120:2) label=2;
	\node (6) at (180:2) {$\vdots$};
	
	\draw (1) to[bend right] (2);
	\draw (2) to[bend right] (3);
	\draw (3) to[bend right] (4);
	\draw (4) to[bend right] (5);
	\draw (5) to[bend right] (6);
	\draw (6) to[bend right] (1);
\end{tikzpicture}
\\
(a) & (b) \\
\end{tabular}
\caption{\label{fig-primitive-pearl-diagrams}Primitive pearl diagram and FLS-diagram.}
\end{center}
\end{figure}

\begin{rem}
Concretely, the tropical curves are as follows: an horizontal edge that meets $g-1$ loops that each make a certain number of loops around the vertical direction. This certain number is actually the loop data.
\end{rem}

\subsubsection{Generating series of the classical primitive invariants.} If we denote by $G_2(y)=-\frac{1}{24}+\sum_1^\infty \sigma_1(n)y^n$ the first Eisenstein series and $D$ the differential operator $y\frac{\dd}{\dd y}$, the generating series of the classical invariants are
$$\begin{array}{r>{\displaystyle}l}
\sum_1^\infty N_{g,\left(\begin{smallmatrix}
1 & 0 \\
0 & n \\
\end{smallmatrix}\right)}y^n & = g\big(DG_2(y)\big)^{g-1}, \\
\sum_1^\infty N^{FLS}_{g,\left(\begin{smallmatrix}
1 & 0 \\
0 & n \\
\end{smallmatrix}\right)}y^n & = \big(DG_2(y)\big)^{g-2}D^2 G_2(y). \\
\end{array}$$
We have recovered the expressions from \cite{bryan1999generating}.

\subsubsection{Generating series of the refined primitive invariants.} Setting
$$S(q,y)=\sum_{n=1}^\infty \sum_{k|n}\frac{n}{k}[k]_-^2y^n,$$
where $[k]_-^2=q^k-2+q^{-k}$, we get that the generating series of the refined invariants in the primitive case are
$$\begin{array}{>{\displaystyle}r>{\displaystyle}l}
\sum_1^\infty R_{g,\left(\begin{smallmatrix}
1 & 0 \\
0 & n \\
\end{smallmatrix}\right)}y^n & = g S(q,y)^{g-1}, \\
\sum_1^\infty R^{FLS}_{g,\left(\begin{smallmatrix}
1 & 0 \\
0 & n \\
\end{smallmatrix}\right)}y^n & = S(q,y)^{g-2}y\frac{\dd}{\dd y}\big(S(q,y)\big). \\
\end{array}$$
We recover the expressions of \cite{bryan2018curve} that give the generating series of the Gromov-Witten invariants in a fixed linear system with point insertion and a $\lambda$-class insertion, suggesting that in this situation as well the refined invariants have a close relation to these specific Gromov-Witten invariants.

\subsection{Quasi-modularity of generating series}

Let $g$ and $d$ be fixed positive integers. We consider the generating series for the invariants $N_{g,(d,dn)}$ and $N^{FLS}_{g,(d,dn)}$:
$$F_{g,d}(y)=\sum_{n=1}^\infty N_{g,(d,dn)}y^n,$$
$$F^{FLS}_{g,d}(y)=\sum_{n=1}^\infty N^{FLS}_{g,(d,dn)}y^n.$$
The generating series $F_{g,1}$ and $F^{FLS}_{g,1}$ have already been computed in the previous section. As they are a polynomial in $DG_2(y)$, they are quasi-modular forms for $SL_2(\ZZ)$. We generalize this result to the generating series for the non-primitive case.

\begin{theo}\label{theorem quasimodularity}
The series $F_{g,d}$ and $F^{FLS}_{g,d}$ are quasi-modular forms for some finite index subgroup of $SL_2(\ZZ)$ containing $\left(\begin{smallmatrix} 1 & 1 \\ 0 & 1 \\ \end{smallmatrix}\right)$ and they have the following expressions:
\begin{align*}
F_{g,d}(e^{2i\pi\tau}) & =g\sum_{\substack{k|d \\ 0\leqslant l < k^2 }}\left(\frac{d}{k}\right)^{4g-3} \big(DG_2(e^{2i\pi(\tau+l)/k^2})\big)^{g-1}, \\
F^{FLS}_{g,d}(e^{2i\pi\tau}) & =\sum_{\substack{k|d \\ 0\leqslant l < k^2 }}\left(\frac{d}{k}\right)^{4g-1} \big( DG_2(e^{2i\pi(\tau+l)/k^2})\big)^{g-2}D^2G_2(e^{2i\pi(\tau+l)/k^2}). \\
\end{align*}
\end{theo}

\begin{proof}
Taking the generating series of the invariants, the multiple cover formula (ii) from Theorem \ref{theorem multiple cover formulas point} gives that
$$F_{g,d}(y)=\sum_{k|d}\left(\frac{d}{k}\right)^{4g-3}\sum_1^\infty N_{g,(1,k^2 n)}y^n.$$
Thus, we need to compute the partial generating series of $F_{g,1}$:
\begin{align*}
\sum_{0\leqslant l < k^2}F_{g,1}\left(\frac{\tau+l}{k^2}\right) & = \sum_{0\leqslant l < k^2}\sum_{n=1}^\infty N_{g,(1,n)}e^{2i\pi\frac{n}{k^2}(\tau+l)} \\
& = \sum_{n=1}^\infty N_{g,(1,k^2 n)}e^{2i\pi n\tau}.\\
\end{align*}
Making the sum over divisors $k$ of $d$ and replacing $F_{g,1}$ by its expression yields the desired expression. Moreover, it shows that the function is $1$-periodic. According to lemma \ref{lemma uasi-modularity transformation}, the functions on the left are quasi-modular forms for some subgroup of $SL_2(\ZZ)$. Intersecting their modular groups over the diviors $k$ of $d$ proves that the sum is also quasi-modular. Thus, so is $F_{g,d}$. The fixed linear system case is treated similarly.
\end{proof}

\subsection{Polynomiality of the coefficients of the refined invariants}

We now investigate the refined invariants.

\subsubsection{Degree and leading term.} The explicit formulas from Section \ref{section computation primitive case} easily yield the following.

\begin{prop}
The refined invariant $R_{g,(1,n)}(q)$ is of degree $n$. Its leading coefficient is $g\left(\begin{smallmatrix} n-1 \\ g-2 \\ \end{smallmatrix}\right)$.
\end{prop}

\begin{proof}
The polynomial $\sum_{k|a}k\left[\frac{a}{k}\right]_-^2$ is of degree $a$ and its leading term is $1$. Thus, the formula for $R_{g,(1,n)}$ asserts that it is of degree $n$ and its leading term is equal to $g$ times the number of partitions of $n$ into $g-1$ numbers bigger than $1$. The latter is equal to the number of partitions of $n-g+1$ into $g-1$ non-negative numbers, equal to $\left(\begin{smallmatrix} (n-g+1)+(g-2) \\ g-2 \\ \end{smallmatrix}\right)$: writing $n-g+1$ as $1+1+\cdots +1$,  we have to place $g-2$ separations.
\end{proof}

\begin{rem}
The computation is similar for the fixed linear system case. The only difference is that the leading term is now
\begin{align*}
\sum_{a_1+\cdots+a_{g-1}=n}a_{g-1} & = \sum_{a_{g-1}=1}^{n-(g-2)}a_{g-1}\sum_{a_1+\cdots + a_{g-2}=n-a_{g-1}}1 \\
& = \sum_{l=1}^{n-g+2} l\left(\begin{smallmatrix} n-l-1 \\ g-3 \\ \end{smallmatrix}\right).\\
\end{align*}
It is also a polynomial in $n$ for any choice of $g$.
\end{rem}

\begin{coro}
The refined invariant $R_{g,(d_1,d_2)}(q)$ is of degree $d_1d_2$. Its leading coefficient is $g\left(\begin{smallmatrix} d_1d_2-1 \\ g-2 \\ \end{smallmatrix}\right)$.
\end{coro}

\begin{proof}
From the multiple cover formula (iii) and (iv) in Theorem \ref{theorem multiple cover formulas point}, we can see that the only term from the cover formula which contributes is $R_{g,(1,d_1d_2)}$.
\end{proof}

\subsubsection{Coefficients of fixed codegree.} In \cite{brugalle2020polynomiality}, E. Brugall\'e and A. Puentes prove the following regularity result for the refined invariants of specific toric surfaces. For simplicity, we state it for the projective plane. If we cancel the denominators in the Block-G\"ottsche multiplicities, the refined invariant $BG_{g,d}(q)$ is a (anti-)symmetric Laurent polynomial of degree $d^2/2$. For a fixed $p$ and $g$, the coefficient in $q^{d^2/2-p}$ (codegree $p$) is asymptotically a polynomial in $d$. Moreover, they give explicit bounds for the domain on which the coefficient is a polynomial. We here prove a similar result for the coefficients of our invariants $R_{g,C}$.

\begin{theo}\label{theorem polynomiality}
The coefficient $\langle R_{g,(d_1,d_2)}\rangle_p$ of codegree $p$ in $R_{g,(d_1,d_2)}$ is a polynomial in $d_1d_2$ for $d_1d_2>\max(g-1,2p)$.
\end{theo}

\begin{proof}
The formula for $R_{g,(1,n)}$ is a sum over the partitions $a_1+\cdots+a_{g-1}$ of $n$ of a product polynomials each only having terms of degree $\pm a_i/k_i$ and $0$ for every divisor $k_i$ of $a_i$. Each term is of degree $n$. Thus, to get a term of codegree $p$, we need to dispatch the codegree $p$ among the terms of each product:
$$\langle R_{g,(1,n)}\rangle_p = g \sum_{a_1+\cdots+a_{g-1}=n}\sum_{p_1+\cdots+p_{g-1}=p}\prod_1^{g-1} \left\langle \sum_{k|a_i}k\left[\frac{a_i}{k}\right]_-^2\right\rangle_{p_i},$$
where $p_i$ are between $0$ and $p$. For each $a_i$ and $k|a_i$ bigger than $1$, we have $\frac{a_i}{k_i}\leqslant\frac{a_i}{2}$. Thus, if $a_i>2p$, the polynomial $\sum_{k|a}k\left[\frac{a_i}{k}\right]_-^2$ has no term of codegree between $1$ and $p$, and the corresponding term in the product is $0$. Therefore, for each choice of $p_1+\cdots+p_{g-1}=p$, we have a finite number of possibilities for the choice $a_i$ with $p_i\neq 0$. Then, the remaining choice for the other $a_i$ is the number of partitions of $n-\sum_{p_i\neq 0} a_i$ into positive numbers. The latter is a polynomial in $n$, given as some binomial coefficient.

\medskip

In the non-primitive case, it is a consequence of the multiple cover formula from Theorem \ref{theorem multiple cover formulas point}. It asserts that
$$R_{g,(d_1,d_2)}=\sum_{k|\mathrm{gcd}(d_1,d_2)}k^{2g-1}R_{g,(1,d_1d_2/k^2)}(q^k).$$
The term in the sum is of degree $\frac{d_1d_2}{k}$. Thus, if $d_1d_2>2p$, only the $k=1$ term contributes to the codegree $p$ coefficient since none of the other terms has a sufficiently big degree.
\end{proof}

\begin{expl}
Let us compute the coefficients of codegree $1$ and $2$ for $R_{g,(1,n)}$. The leading term has already be computed. For the coefficient of codegree $1$, we can have $a_i=1$ or $2$ for exactly one $i$, since they are the only polynomials $\sum_{k|a}k(q^{a/k}-2+q^{-a/k})$ with a non-zero codegree $1$ term. Thus, the coefficient is
\begin{align*}
\langle R_{g,(1,n)}\rangle_1 & = g(g-1)\left(\begin{smallmatrix} n-2 \\
g-3 \\ \end{smallmatrix} \right)(-2) + g(g-1)\left(\begin{smallmatrix} n-3 \\
g-3 \\ \end{smallmatrix} \right)2 \\
& = -2g(g-1)\left(\begin{smallmatrix} n-3 \\
g-4 \\ \end{smallmatrix} \right)
\end{align*}
For the coefficient of codegree $2$, we now have more possibilities:
\begin{itemize}[label=-]
\item We have one $a_i=1$ and take the codegree $2$ coefficient,
\item We have one $a_i=2$ and take the codegree $2$ coefficient,
\item We have two $a_i$ equal to $1$ where we take the codegree $1$ coefficients,
\item We have two $a_i$ equal to $2$ where we take the codegree $1$ coefficients,
\item We have one $a_i$ equal to $1$ and one equal to $2$ where we take the codegree $1$ coefficients.
\end{itemize}
We thus get the following formula:
\begin{align*}
\langle R_{g,(1,n)}\rangle_2 = & g(g-1)\left(\begin{smallmatrix} n-2 \\
g-3 \\ \end{smallmatrix} \right)(1) + g(g-1)\left(\begin{smallmatrix} n-3 \\
g-3 \\ \end{smallmatrix} \right)(-3)+ g\frac{(g-1)(g-2)}{2}\left(\begin{smallmatrix} n-3 \\
g-4 \\ \end{smallmatrix} \right)(-2)^2 \\
&  + g\frac{(g-1)(g-2)}{2}\left(\begin{smallmatrix} n-5 \\
g-4 \\ \end{smallmatrix} \right)2^2 + g(g-1)(g-2)\left(\begin{smallmatrix} n-4 \\
g-4 \\ \end{smallmatrix} \right)(-2)2. \\
=& g(g-1)\left( \left(\begin{smallmatrix} n-2 \\
g-3 \\ \end{smallmatrix} \right) -3\left(\begin{smallmatrix} n-3 \\
g-3 \\ \end{smallmatrix} \right) \right) +2g(g-1)(g-2)\left(\begin{smallmatrix} n-5 \\
g-6 \\ \end{smallmatrix} \right).\\
\end{align*}
\end{expl}

\bibliographystyle{alpha}
\bibliography{biblio}

\end{document}